\documentclass[10pt, reqno]{amsart}
\usepackage{amsmath, amssymb}
\usepackage{dsfont}
\usepackage{amscd}
\usepackage[mathscr]{euscript} % for the power-set P
\usepackage[all]{xy}
\usepackage{url}
\usepackage{stmaryrd}
\usepackage{comment}
\usepackage[retainorgcmds]{IEEEtrantools}

\usepackage[colorlinks=true]{hyperref}

\usepackage{graphicx}
\usepackage{mathtools}
\usepackage[dvipsnames]{xcolor}
\usepackage{centernot}
\usepackage{marvosym}
\usepackage{tikz}
\usepackage{tikz-cd}
\usetikzlibrary{calc, arrows.meta, positioning}
\usepackage{xspace}
\usepackage{bbm}
\usepackage[title]{appendix}

\usepackage{tkz-fct}
\usepackage{soul}
\usepackage{lineno}
\usepackage[T1]{fontenc}
\usepackage[utf8]{inputenc}

\usepackage{xcolor}
\definecolor{green}{RGB}{0,127,0}
\definecolor{redd}{RGB}{191,0,0}
\definecolor{red}{RGB}{105,89,205}
\usepackage[colorlinks=true]{hyperref}

\usepackage[english]{babel}
\usepackage[autostyle]{csquotes}
\usepackage[shortlabels]{enumitem}

\usepackage[backend=bibtex,  sorting=nyt, style=alphabetic]{biblatex}
\allowdisplaybreaks %Fix horrible vertical spacing issues.

\title[Tame dynamics and GNC]
{On tame amenable flows and generalized Newelski's conjecture}


\usepackage{orcidlink}

\author[D. M. HOFFMANN]{Daniel Max Hoffmann$^{\dagger}$\
\orcidlink{0000-0002-4514-269X}}
\thanks{$^{\dagger}$SDG. The first author is supported
by the National Science Centre (Narodowe Centrum Nauki, Poland) 
grant no. 2021/43/B/ST1/00405,
and by the UW IDUB PSP no. 501-D110-20-0001381.}
\address{$^{\dagger}$
Instytut Matematyki\\
Uniwersytet Warszawski\\
Warszawa\\
Poland}
\email{daniel.max.hoffmann@gmail.com}
\urladdr{{https://sites.google.com/site/danielmaxhoffmann/}}

\author[K. Krupi\'{n}ski]{Krzysztof Krupi\'{n}ski$^{\ddagger}$\ \orcidlink{0000-0002-2243-4411}}
\address{$^{\ddagger}$Instytut Matematyczny, Uniwersytet Wroc{\l}awski, pl. Grunwaldzki 2, 50-384 Wroc{\l}aw, Poland}
\email{Krzysztof.Krupinski@math.uni.wroc.pl}
\urladdr{{https://www.math.uni.wroc.pl/~kkrup/}}

\date{\today}

 \DeclareMathOperator{\aut}{Aut} \DeclareMathOperator{\id}{id}
\DeclareMathOperator{\cl}{cl}

\DeclareMathOperator{\im}{im}

 \DeclareMathOperator{\gal}{Gal}

\DeclareMathOperator{\tp}{tp}

\DeclareMathOperator{\ev}{ev}

\newcommand{\Aut}{\operatorname{Aut}}

\DeclareMathOperator{\Prob}{Prob}

\newtheorem{theorem}{Theorem}[section]
\newtheorem{proposition}[theorem]{Proposition}
\newtheorem{lemma}[theorem]{Lemma}
\newtheorem{cor}[theorem]{Corollary}
\newtheorem{fact}[theorem]{Fact}
\newtheorem{problem}[theorem]{Problem}
\theoremstyle{definition}
\newtheorem{definition}[theorem]{Definition}
\newtheorem{example}[theorem]{Example}
\newtheorem{remark}[theorem]{Remark}
\newtheorem{question}[theorem]{Question}
\newtheorem{conj}[theorem]{Conjecture}

\theoremstyle{remark}
\newtheorem*{theorem*}{Theorem}
\newtheorem*{cor*}{Corollary}

\theoremstyle{definition}
\theoremstyle{definition}

\theoremstyle{definition}

\theoremstyle{remark}

\newtheorem{clm}{Claim}
\newtheorem*{clm*}{Claim}

\AtEndEnvironment{proof}{\setcounter{clm}{0}}
\newenvironment{clmproof}[1][\proofname]{\proof[#1]}{\endproof}

\makeatletter
\providecommand*{\cupdot}{%
  \mathbin{%
    \mathpalette\@cupdot{}%
  }%
}
\newcommand*{\@cupdot}[2]{%
  \ooalign{%
    $\m@th#1\cup$\cr
    \sbox0{$#1\cup$}%
    \dimen@=\ht0 %
    \sbox0{$\m@th#1\cdot$}%
    \advance\dimen@ by -\ht0 %
    \dimen@=.5\dimen@
    \hidewidth\raise\dimen@\box0\hidewidth
  }%
}

\providecommand*{\bigcupdot}{%
  \mathop{%
    \vphantom{\bigcup}%
    \mathpalette\@bigcupdot{}%
  }%
}
\newcommand*{\@bigcupdot}[2]{%
  \ooalign{%
    $\m@th#1\bigcup$\cr
    \sbox0{$#1\bigcup$}%
    \dimen@=\ht0 %
    \advance\dimen@ by -\dp0 %
    \sbox0{\scalebox{2}{$\m@th#1\cdot$}}%
    \advance\dimen@ by -\ht0 %
    \dimen@=.5\dimen@
    \hidewidth\raise\dimen@\box0\hidewidth
  }%
}
\makeatother

\def\Ind#1#2{#1\setbox0=\hbox{$#1x$}\kern\wd0\hbox to 0pt{\hss$#1\mid$\hss}
\lower.9\ht0\hbox to 0pt{\hss$#1\smile$\hss}\kern\wd0}

\def\notind#1#2{#1\setbox0=\hbox{$#1x$}\kern\wd0
\hbox to 0pt{\mathchardef\nn=12854\hss$#1\nn$\kern1.4\wd0\hss}
\hbox to 0pt{\hss$#1\mid$\hss}\lower.9\ht0 \hbox to 0pt{\hss$#1\smile$\hss}\kern\wd0}

\newcommand{\FC}{\mathfrak{C}}

\newcommand{\CM}{{\mathcal M}}

\newcommand{\SL}{{\mathrm{SL}}}

\newcommand{\ff}{\mathfrak{f}}
\newcommand{\fb}{\mathfrak{b}}

\newcommand{\supp}{\mathrm{supp}}

\newcommand{\autf}{\mathrm{Autf}}
\newcommand{\KP}{KP}

\newcommand{\homeo}{\mathrm{Homeo}}

\begin{document}

\begin{abstract}
We are developing some  topological dynamics ``below'' a given ambit. We define and study Bohr compactifications below ambits and other related objects. We describe a natural epimorphism from the Ellis group of an ambit to the Bohr compactification below its Ellis semigroup ambit. The first main result says that, assuming tameness and hereditary amenability, this natural map is an isomorphism of topological groups. This is an abstract dynamical generalization of Newelski's conjecture, or rather a theorem of Chernikov and Simon that a natural map from the Ellis group of the flow $(S_G(M),G)$ to the definable Bohr compactification of $G$ is an isomorphism when $G$ is a definably amenable group definable in a NIP structure $M$. We give two proofs of our main result: a shorter dynamical proof and an analytic proof. In the course of the dynamical proof, we prove some results on the universal compactification flow and on the generalized Bohr compactification below a given ambit which admits a left continuous semigroup operation extending the action of $G$. The considerations around the analytic proof yield a description of the measured structure of amenable tame minimal flows, which is the second main result of this paper.  Our results do not require any metrizability assumption. 

We also study specializations of the above abstract topological dynamics context to model theory contexts to recover some known and to obtain some new notions and results. The main new results in the model theory contexts are: (i) a description of the Bohr compactification of the group of automorphisms of a sufficiently saturated model $\FC$ of a first order theory $T$ below a certain fundamental ambit associated with $\FC$; (ii) transfer of amenability of a NIP theory $T$ (i.e., amenability of the aforementioned ambit) to amenability of the Ellis semigroup ambit; (iii) a version of Newelski's conjecture for spaces of local types with a description of Bohr compactifications below such spaces. 
\end{abstract}

\maketitle
\section{Introduction}

Recall that a {\em $G$-flow} is a pair $(X,G)$ consisting of a compact (Hausdorff) space $X$ and a topological group $G$ together with a continuous action of $G$ on $X$. In this paper, the topology on $G$ is irrelevant, so $G$ is assumed to be a discrete group. A {\em $G$-ambit} is a flow $(X,x_0,G)$ with a distinguished point $x_0$ with dense orbit. The {\em Ellis semigroup} of a flow $(X,G)$ is the closure of $\{g_X: g\in G\}$ in the space $X^X$ equipped with the pointwise convergence topology, where $g_X(x):=gx$, with composition of functions as the semigroup operation. This is also naturally a $G$-ambit (which we call the {\em Ellis flow} or {\em Ellis ambit} of $(X,G)$) with $\id_X$ as the distinguished element with dense orbit.

\subsection{Motivation}
The motivation for the research in this paper comes from several directions. 
One of them is model theory: the goal of the project is to generalize to the abstract context of topological dynamics the various important notions, results, and phenomena known in topological dynamics in model theory (i.e., results about the flows of the form $(S_G(M),G)$ where $G$ is a group definable in a structure $M$ and $S_G(M)$ is the space of complete types over $M$ concentrated on $G$, and also some other natural flows considered  in model theory). 
This was successfully realized in \cite{HoffKru1} under the tameness assumption (which corresponds to the model-theoretic NIP) for results concerning the transfer of amenability of a flow to its Ellis flow, unique ergodicity, and a description of ergodic measures. In this paper, the main point is to generalize the so-called Newelski's conjecture to the abstract context and prove it for hereditarily amenable tame flows. Another, strongly related goal and motivation is to try to develop topological dynamics ``below'' a given ambit. For a given group $G$ the usual topological dynamics automatically takes place below the universal $G$-ambit. A motivation is that if we replace the universal $G$-ambit by other $G$-ambits, we can impose interesting dynamical properties (e.g., tameness) which are almost never satisfied by the universal $G$-ambit. The third motivation is to obtain results (and to develop techniques) without any metrizability assumption on the underlying flow which was an essential assumption in many topological dynamics papers (e.g. in \cite{Glasner2} describing the structure of metrizable tame minimal flows). For example, originally we had some results from \cite{HoffKru1} only in the metrizable case, as we were using \cite{Glasner2} which requires this assumption, but then we developed techniques which allowed us to drop the metrizability assumption. Similarly in the case of the so-called Revised Newelski Conjecture from \cite{Krupinski2023a}, which was first deduced in \cite{CGK} from \cite{Glasner2} in a metrizable context, and then generalized to arbitrary tame minimal flows in \cite{BaZu} by using set-theoretic methods. In this paper, we observe that the Generalized Newelski's Conjecture (which we state in this paper) follows from \cite{Glasner2} for metrizable minimal flows, and then we give two different proofs which work in general (i.e., without the metrizability and minimality assumption). We also provide a description of the measured structure of amenable tame minimal flows, which is a step towards understanding of the structure of such flows without the metrizability assumption.

\subsection{Main results in abstract topological dynamics}

For precise definitions and more references see further sections.

Recall that the {\em Bohr compactification} of a discrete group $G$ is a unique (up to isomorphism) universal group homomorphism from $G$ to a compact group with dense image. We define the Bohr compactification of $G$ below a given $G$-ambit $(X,x_0,G)$ as a unique (up to isomorphism) universal  homomorphism as above which additionally factors through the ambit $(X,x_0,G)$ (see Definition \ref{definition: Bohr compactification below ambit A}).  In Subsections \ref{subsection: Bohr comp} and \ref{subsection: MEFs}, we study such Bohr compactifications and compare them to the maximal equicontinuous factors (MEF's). We observe that in general the MEF of $(X,x_0,G)$ may be strictly bigger than the Bohr compactification of $G$ below  
$(X,x_0,G)$, but these notions coincide when there is a left continuous semigroup operation on $X$ extending the action of $G$ (equivalently, when $(X,x_0,G)$ is isomorphic to an Ellis flow $(E(Y,G),\id_Y,G)$ for some $G$-flow $(Y,G)$).

To state the Generalized Newelski's Conjecture, we need to recall a fundamental theorem of Ellis theory. As already mentioned, with a flow $(X,G)$ one associates its Ellis semigroup $E(X,G)$. This is a left-topological semigroup (i.e., the semigroup operation is continuous in the left coordinate). Ellis theorem tells us that there exists a minimal left ideal $\CM$ in $E(X,G)$ and each such ideal $\CM$ can be written as the disjoint union of all sets $u\CM$, where $u$ ranges over the idempotents in $\CM$. Moreover, each $u\CM$ is a group, and all these groups are isomorphic for varying minimal left ideal $\CM$ and idempotent $u \in \CM$. Each group $u\mathcal{M}$ is equipped with the so-called $\tau$-topology which makes it a quasi-compact semitopological group (i.e., the group operation is separately continuous).  The groups $u\CM$ remain isomorphic also as semitopological groups. The isomorphism type of these groups (or any of these group separately) will be called the {\em Ellis group} of $(X,G)$.

Let us fix a flow $(X, G)$. Set $E:=E(X,G)$, choose a minimal left ideal $\CM$ in $E$ and an idempotent $u \in \CM$. Let $\mathfrak{b} \colon G \to B$ be the Bohr compactification of $G$ below $(E,\id_X,G)$. There is a unique left continuous action of $E$ on $B$  extending the natural action of $G$. We define an abstract counterpart of Newelski's map (whose definition is recalled in Subsection \ref{subsection: definable groups}): 
$$\widehat{\theta} \colon u\CM \to B, \;\;\; \widehat{\theta}(\eta):=\eta e_{B},$$
where $e_{B}$ denotes the neutral element of $B$. In Subsection \ref{subs:category.below}, we explain that $\widehat{\theta}$ is a group epimorphism and a quotient topological map which factors through the largest Hausdorff quotient of $u\CM$ (even through a certain Galois group).

\begin{conj}[Generalized Newelski's Conjecture (GNC)]
Under suitable assumptions on $(X,G)$, the map $\widehat{\theta}$ is a  topological group isomorphism.
\end{conj}

The flow $(X,G)$ is said to be {\em hereditarily amenable} if every subflow admits an invariant Borel probability measure.
Below is our first main result (see Section 2 for the definition of tameness). This theorem is an abstract generalization of  Theorem 1.5 in \cite{ArtemPierre} being one of the main results of that paper. 

\begin{theorem}\label{theorem: main theorem in into}\label{thm:1.2}
The GNC holds when $(X,G)$ is tame and hereditarily amenable. 
\end{theorem}

For an equivalent form of the above conjecture and theorem and for other comments see Subsection \ref{subsection: Formulation of GNC}.

We provide two different proofs of the above theorem: a dynamical proof in Subsection \ref{subsection: dynamical proof of GNC}
 and an analytic one in Subsection \ref{subsection: analytic proof of GNC}. 

A deep theorem which is behind both proofs is a weaker version of GNC which one could call the Generalized Revised Newelski's Conjecture (GRNC); see the discussion below on the model-theoretic origins. This result was proved in \cite[Corollary 5.14]{CGK} in the metrizable context and in \cite[Theorem 11.7]{BaZu} in full generality. For a short deduction of the statement below from \cite[Theorem 11.7]{BaZu} see \cite[Lemma 3.3]{HoffKru1}.

\begin{fact}[Generalized Revised Newelski's Conjecture (GRNC)]\label{fact:RNC}
If $(X,G)$ is tame, then $u\CM$ is Hausdorff in the $\tau$-topology, and so a compact topological group.
\end{fact}

Another non-trivial ingredient of the dynamical proof is the following result contained in \cite[Lemma 5.7]{HoffKru1} (using also \cite[Corollary 5.6]{HoffKru1} which says that if $(X,G)$ is tame and hereditarily amenable, then $E(X,G)$ is hereditarily amenable, and so $\CM$ is amenable.)

\begin{fact} 
If $(X,G)$ is tame and hereditarily amenable, then $u\CM$ is dense in $\CM$.
\end{fact}

Having the above two results and \cite[Corollary 5.6]{HoffKru1}, the dynamical proof is short and uses arguments as in \cite{Glasner1976}. We take the opportunity and also study a bit topological dynamics below a given ambit, yielding some information on the universal compactification flow and on the generalized Bohr compactification in this context (see Subsections \ref{subs:category.below} and \ref{subsection: dynamical proof of GNC}).

The main engine of the analytic proof is Lemma \ref{lemma: L2 Bohr}, which for a tame hereditarily amenable flow $(X,G)$ identifies the $L^2$-spaces $L^2(B,\mathfrak{h}_{B})$ and $L^2(\CM,\mu)$, where $\mathfrak{h}_{B}$ is the unique normalized Haar measure on $B$ and $\mu$ is the unique $G$-invariant probability measure on $\CM$. A final deduction of GNC is then contained in the proof of Theorem \ref{thm: GNC for amenable minimal ideal} and afterwards, which also uses the uniqueness of a $G$-invariant probability measure on $\CM$. This last fact is the non-trivial Theorem 4.2 in \cite{HoffKru1} (which relies on Fact \ref{fact:RNC}).

In passing from the study of metrizable minimal tame amenable flows (e.g. as in \cite{Glasner2}) to the context of general
minimal tame amenable flows, we use measure-theoretic substitutes
for the algebra $C(X)$ and for techniques based on tameness.
The spaces $L^2(X,\mu)$ and the von Neumann algebras
$L^\infty(X,\mu)$ provide such a framework, allowing us to study
the measured structure without metrizability assumptions.
In Subsection \ref{subsection: structure of amenable tame flows}, Lemma \ref{lemma: L2 Bohr} is used to prove the second main result of this paper, Theorem \ref{thm: measured vNeumann}, i.e. 
a characterization of the measured structure of the minimal, tame and amenable flows. We establish there

\begin{theorem}[Theorem \ref{thm: measured vNeumann}]
If $(X,G)$ is an amenable tame minimal flow, then 
$$L^\infty(X,\mu) \cong_G L^\infty(B/\gal_B,\mathfrak h_{B/\gal_B}),$$
where $\mu$ and $\mathfrak h_{B/\gal_B}$ are the unique
$G$-invariant Borel probability measures on $X$ and $B/\gal_B$,
respectively (for the definition of $\gal_B$, see
Subsection~\ref{subsection: structure of amenable tame flows}),
and $\cong_G$ denotes a normal trace-preserving $G$-equivariant
$\ast$-isomorphism of von Neumann algebras.
\end{theorem}

This can be expressed also as follows.
For a probability space $(Y,\nu)$, let $\operatorname{MAlg}(Y,\nu)$ denote its complete Boolean algebra of measurable sets modulo null sets.

\begin{cor}[Corollary \ref{rem: homogeneous MAlg}]
If $(X,G)$ is an amenable tame minimal flow, then 
$$\operatorname{MAlg}(X,\mu)\cong_G\operatorname{MAlg}(B/\gal_B,\mathfrak h_{B/\gal_B}).$$
Here $\cong_G$ denotes a measure-preserving $G$-equivariant
isomorphism of complete Boolean algebras.
\end{cor}

In Subsection \ref{subsec: X-MEF 2}, we use Theorem \ref{thm: measured vNeumann} to get a similar identification of $\operatorname{MAlg}(X,\mu)$ and $\operatorname{MAlg}(Z,\lambda_Z)$\footnote{Alternatively, the corresponding von Neumann algebras.}, where $\pi \colon X \to Z$ is the MEF of $(X,G)$ and $\lambda_Z$ is the unique $G$-invariant Borel probability measure on $Z$. 
Using GNC (i.e., Theorem \ref{thm:1.2}), we also show
that  $\pi$ is a proximal map, 
and this is already used in an ongoing project of the first author with Andy Zucker on extending the main results of \cite{Glasner2} to non-metrizable flows.

All these results form a contribution to understanding of the structure of amenable tame minimal flows without the metrizability assumption.

\subsection{Original Newelski's Conjecture and new results in model-theoretic contexts}

First, we recall some history related to Newelski's conjecture and its variants. For more details (including definitions) see Subsection \ref{subsection: definable groups}.

Let $G=G(M)$ be a group definable in a structure $M$. We consider the ambit $(S_{G,\textrm{ext}}(M),\tp_{\textrm{ext}}(e/M),G)$ of all complete external types over $M$ concentrated on $G$ (i.e., the space of ultrafilters in the Boolean algebra of externally definable subsets of $G$). It turns out that there is a left continuous semigroup operation $*$ on  $S_{G,\textrm{ext}}(M)$ extending the action of $G$, and so $S_{G,\textrm{ext}}(M)$ can be identified with its Ellis semigroup (as a semigroup and $G$-flow). Let $\CM$ be a minimal left ideal in $S_{G,\textrm{ext}}(M)$ and $u \in \CM$ an idempotent. Let $\bar{G}^{00}_M$ be the smallest $M$-type-definable subgroup of bounded index in $\bar G$, where $\bar{G}$ is the interpretation of $G$ in a sufficiently saturated model (so-called monster model). Then $\bar{G}/\bar{G}^{00}_M$ equipped with the {\em logic topology} (a subset is closed if its preimage in $\bar G$ is type-definable) is a compact (Hausdorff) topological group. The component $\bar{G}^{00}_M$ and the quotient $\bar{G}/\bar{G}^{00}_M$ play a very important role in model theory (especially in theories with NIP).

In \cite{newelski09}, Newelski introduced a natural epimorphism $\rho \colon u\CM \to  \bar{G}/\bar{G}^{00}_M$ (see Definition \ref{definition: Newelski's map}). {\em Newelski's conjecture} predicted that it is an isomorphism under suitable model-theoretic assumptions. This was an interesting connection between the  topological dynamics object $u\CM$ and the model-theoretic object $\bar{G}/\bar{G}^{00}_M$. Later, the various refinements and variants of Newelski's map \cite{Krupinski2017}, \cite{KruPiRze}, \cite{KrRz}, \cite{Krupinski2023a} led to important contributions in model theory (e.g., to understand the complexity of strong types) and to additive combinatorics (to construct generalized locally compact models for arbitrary approximate subgroups in a simpler and ``more classical'' way than originally done by Hrushovski).

Newelski's conjecture was refuted in \cite{GismatullinPenazziPillay2015Some} for groups definable in NIP theories, concretely for $G:=\SL_2(\mathbb{R})$ treated as a group definable in $M:=(\mathbb{R},+,\cdot)$. The main positive result is the following, where definable amenability of $G$ means that the flow $(S_G(M),G)$ is amenable (i.e., $S_G(M)$ admits an invariant Borel probability measure).

\begin{fact}[Theorem 1.5 in \cite{ArtemPierre}]\label{fact: Theorem 1.5 from CS in intro}
Newelski's conjecture holds for definably amenable groups definable in NIP theories.
\end{fact}

A very important refinement of Newelski's map was obtained in \cite{Krupinski2017}:
$$u\CM \to u\CM/H(u\CM) \to \bar{G}/\bar{G}^{000}_M \to \bar{G}/\bar{G}^{00}_M$$
with all the maps being epimorphisms of groups and topological quotient maps, where $u\CM/H(u\CM)$ is the largest Hausdorff quotient of $u\CM$ and 
$\bar{G}^{000}_M$ is the smallest $M$-invariant bounded index subgroup of $\bar{G}$. This yields many counter-examples to Newelski's conjecture. The variant of $u\CM/H(u\CM)$ for definable approximate subgroups in place of definable groups led to the aforementioned dynamical construction of generalized locally compact models for arbitrary approximate subgroups in \cite{Krupinski2023a}. The above refinement also led to a weakening of Newelski's conjecture stated as \cite[Conjecture 5.3]{Krupinski2023a} and called {\em Revised Newelski's conjecture} (RNC):  If $G$ is definable in a NIP structure $M$, then $u\CM$ is Hausdorff in the $\tau$-topology (equivalently, $H(u\CM) = \{u\}$). It was first confirmed in \cite{CGK} for countable structures. The additional countability assumption was needed, as the proof relied on a structure theorem by Glasner from \cite{Glasner2} stated for metric tame minimal flows. Then \cite{BaZu} used forcing techniques to remove this countability assumption. In fact, Theorem 11.7 from \cite{BaZu} yields a generalization of RNC to the abstract context of tame flows which we stated in  Fact \ref{fact:RNC} and called GRNC. The GRNC was fundamental for the results in \cite{HoffKru1} and is also fundamental in this paper. 
RNC also played a fundamental role in the recent beautiful paper \cite{Stonestrom2026} by A. Stonestrom to describe the structure of Ellis groups and of quotients of $\bar{G}$ by certain local components in the NIP context (under some additional assumption).

\medskip

There is also a variant of the above theory for first order theories in place of definable groups (see Subsection \ref{subsection: first order theories} for definitions, more details, and references). Let $T$ be a first order theory and $\FC \models T$ a sufficiently saturated model. Let $\bar c$ be an enumeration of $\FC$ and $S_{\bar c}(\FC)$ the space of complete types over $\FC$ which contain $\tp(\bar c/\emptyset)$. Then $(S_{\bar c}(\mathfrak{C}),\tp(\bar c/\mathfrak{C}), \aut(\mathfrak{C}))$ is an ambit with respect to the obvious action. This ambit plays a central role in the topological dynamics of first order theories, as it captures a lot of information about the underlying theory $T$ (e.g., see \cite{KruPiRze,Krupinski2019,Hrushovski2019,KrRz,HruKruPi,GHK}).

Let $E:=E(S_{\bar c}(\FC),\aut(\FC))$, $\CM$ be a minimal left ideal in $E$, and $u \in \CM$ an idempotent. A counterpart of Newelski's map was introduced in \cite{KruPiRze}; it is an epimorphism $\rho \colon u\CM \to \gal_{KP}(T)$ from the Ellis group to the {\em Kim-Pillay Galois group} of $T$. An analogous refinement of $\rho$ to the one described above for definable groups was also obtained in the context of theories, which led in \cite{KruPiRze,KrRz} to important applications to the study of complexity of strong types which were not accessible by the previously known methods.

We then have an obvious counterpart of Newelski's Conjecture, and Theorem 0.7 in \cite{Krupinski2019} is a counterpart of Fact \ref{fact: Theorem 1.5 from CS in intro}.

\begin{fact}[Theorem 0.7 in \cite{Krupinski2019}]\label{fact: NC for theories in intro}
If $T$ is NIP and amenable (meaning that the flow $(S_{\bar c}(\mathfrak{C}), \aut(\mathfrak{C}))$ is amenable), then $\rho$ is an isomorphism.
\end{fact}

\medskip

Coming back to definable groups, in Subsection \ref{subsection: definable groups} we explain that Fact \ref{fact: Theorem 1.5 from CS in intro} follows from Theorem \ref{theorem: main theorem in into}. It is not completely direct, as the quotient map $G \to \bar{G}/\bar{G}^{00}_M$ is the Bohr compactification of $G$ below the ambit $(S_G(M),\tp(e/M),G)$ (see \cite{GPPb})  which is a factor of the Bohr compactification of $G$ below the ambit $(S_{G,\textrm{ext}}(M),\tp_{\textrm{ext}}(e/M),G)$, and we observe that the map $\rho$ is the composition of the map $\widehat{\theta} \colon u\CM \to B$ with this factor map. Moreover, in Theorem \ref{theorem: main theorem in into}, we assume hereditary amenability (equivalently amenability) of $(S_{G,\textrm{ext}}(M),G)$ while in  Fact \ref{fact: Theorem 1.5 from CS in intro} only amenability of $(S_G(M),G)$ is assumed. However, by the results from \cite{CPS14}, assuming NIP, the aforementioned factor map is an isomorphism and amenability transfers from $S_G(M)$ to $S_{G,\textrm{ext}}(M)$, so all the obstacles vanish.

In the context of first order theories, we prove a counterpart of the aforementioned result from \cite{GPPb}: the natural map $\Aut(\FC) \to \gal_{KP}(T)$ is the Bohr compactification of $\aut(\FC)$ below the ambit $(S_{\bar c}(\mathfrak{C}),\tp(\bar c/\mathfrak{C}), \aut(\mathfrak{C}))$. We also prove that under NIP amenability of $T$ implies amenability of the Ellis flow $E:=E(S_{\bar c}(\FC),\aut(\FC))$, which is a counterpart of the aforementioned result from \cite{CPS14}. We conjecture that NIP implies that the Bohr compactifications below  the ambits $(S_{\bar c}(\mathfrak{C}),\tp(\bar c/\mathfrak{C}), \aut(\mathfrak{C}))$ and $(E,\id_{S_{\bar c}(\FC)},\aut(\FC))$ coincide (as for definable groups). Because we have not proved this last conjecture, Theorem \ref{theorem: main theorem in into} specializes to a slightly weaker statement than Fact \ref{fact: NC for theories in intro}: the map $\rho$ is the composition of $\widehat{\theta}$ with an epimorphism from the Bohr compactification below $(E,\id_{S_{\bar c}(\FC)},\aut(\FC))$ to the one below  $(S_{\bar c}(\mathfrak{C}),\tp(\bar c/\mathfrak{C}), \aut(\mathfrak{C}))$ (which is $\gal_{KP}(T))$, so the fact that $\widehat{\theta}$ is an isomorphism is weaker than the fact that $\rho$ is an isomorphism.

Example \ref{example: B_1 not equal  to B_2} shows that in the general context one cannot improve Theorem \ref{theorem: main theorem in into} by saying that the composition of $\widehat{\theta}$ with the canonical epimorphism from the Bohr compactification below $(E(X,G),\id_X,G)$ to the one below $(X,x_0,G)$ is an isomorphism. This leads to an interesting general problem on extracting a dynamical reason why something like that is true in the above two model-theoretic contexts. This and another problem in this spirit is stated in Subsection \ref{subsection: two problems}.

Besides recovering Fact \ref{fact: Theorem 1.5 from CS in intro} and a weaker form of Fact \ref{fact: NC for theories in intro}, Theorem \ref{theorem: main theorem in into} may lead to variants of Newelski's conjecture for some more specific model-theoretic ambits (e.g, the ones considered in \cite{GHK}). In Subsection \ref{subsection: local NC}, we obtain a local variant of Newelski's conjecture for definable groups, where the ambit in question is a hereditarily  amenable ambit of the form $(S_{G,\Delta}(M),\tp_{\Delta}(e/M),G)$ for a definable group $G$ and a set of formulas $\Delta$ each of which has NIP and the set of the instances of the formulas from $\Delta$ is closed under the left translations (where $S_{G,\Delta}(M)$ denotes the space of complete $\Delta$-types over $M$ concentrated on $G$). Under the additional assumption that the set of the instances of the formulas from $\Delta$ is closed under group inversion, we give a description of the Bohr compactification of $G$ below  $(S_{G,\Delta}(M),\tp_{\Delta}(e/M),G)$.

\subsection{Acknowledgements} 
The first author is thankful to Alessandro Codenotti, Eli Glasner 
and Andy Zucker for sharing their insights on related topics.

\subsection{No AI declaration} AI was not used for any part of this paper, neither in the mathematical
content nor in the writing.

\section{Preliminaries}
\label{sec:preliminaries}
Throughout, $G$ denotes a discrete group. All compact spaces are assumed to be Hausdorff, all probability measures are regular Borel measures.
We are following the notation and conventions from our previous work, \cite{HoffKru1}, some of them shortly recalled below.

\subsection{Flows and invariant measures}
\label{subsec:flows-and-invariant-measures}

 A \emph{subflow} of a flow is a nonempty closed $G$-invariant subset, and a \emph{factor map} is a continuous surjective $G$-map between $G$-flows.

For a compact space $X$, let $C(X):=C(X,\mathbb C)$, endowed with the
uniform norm and the involution $f^*(x):=\overline{f(x)}$.

We write $\Prob(X)$ for the compact convex space of regular Borel
probability measures on $X$, equipped with the weak-$*$ topology inherited
from $C(X)^*$. If $\pi:X\to Y$ is Borel and $\mu\in\Prob(X)$, then
$\pi_*\mu$ denotes the pushforward of $\mu$. For $g\in G$, we write
$g_*\mu$ for the pushforward under $x\mapsto gx$, and put
$$\Prob_G(X):=\{\mu\in\Prob(X):g_*\mu=\mu\text{ for every }g\in G\}.$$
An invariant measure is \emph{ergodic} if it is an extreme point of
$\Prob_G(X)$.
The flow $(X,G)$ is \emph{amenable} if $\Prob_G(X)\neq\varnothing$, and it
is \emph{uniquely ergodic} if $\Prob_G(X)$ is a singleton. It is
\emph{hereditarily amenable} if every subflow of $X$ is amenable. The action
of $G$ on $C(X)$ is given by $(g\cdot f)(x):=f(g^{-1}x)$.

\begin{fact}[Theorem 4.2 in \cite{HoffKru1}]\label{lemma : min tame unique ergodicity}
Every minimal tame amenable flow is uniquely ergodic.
\end{fact}

Recall also the definition of a group extension of a flow.

\begin{definition}
Let $\chi\colon (Y,G) \to (X,G)$ be a homomorphism of minimal flows. We say that $\chi$ is a {\em group extension} if there is a compact Hausdorff group $K$ acting faithfully and continuously on the right on $Y$ in such a way that:
\begin{enumerate} 
\item  for every $y \in Y$, $\chi^{-1}(\chi(y)) = y\cdot K$, and
\item for all $y \in Y$, $g \in G$ and $k \in K$, $(gy)\cdot k=g(y\cdot k)$.
\end{enumerate}
In the above situation, we also say that $(Y,G,K)$ is a {\em group extension} of $(X,G)$. 
\end{definition}

\subsection{Ellis semigroups and tameness}
The \emph{enveloping semigroup} of a flow $(X,G)$ is
$$E(X,G):=\overline{\{g_X:g\in G\}}^{\,X^X},\qquad g_X(x):=gx,$$
where $X^X$ carries the product topology. Multiplication is composition: $(pq)(x):=p(q(x))$.
For every fixed $q\in E(X,G)$, the map $p\mapsto pq$ is continuous. For
$x\in X$, the evaluation at $x$ is the continuous map
$$\ev_x \colon E(X,G)\longrightarrow X,
\qquad
\ev_x(p):=p(x),$$
and $\ev_x[E(X,G)]=\overline{Gx}$.
A factor map $\pi \colon X\to Y$ induces a continuous surjective
$G$-equivariant semigroup homomorphism
$$\pi_E \colon E(X,G)\longrightarrow E(Y,G)$$
satisfying $\pi\circ p=\pi_E(p)\circ\pi$ for every $p\in E(X,G)$.

Put $E:=E(X,G)$. We write $\CM\unlhd_m E$ for a minimal left ideal.
If $u^2=u\in\CM$, then
$K:=u\CM$ is the corresponding \emph{Ellis group}.
For $p\in E$ and $A\subseteq E$, define the Ellis circle operation by
\begin{IEEEeqnarray*}{rCl}
    p\circ A &:=& \bigl\{q\in E:\text{there are nets }(g_i)_i\subseteq G\text{ and }(a_i)_i\subseteq A \\
 & & \text{ such that }g_i\to p\text{ and }g_i a_i\to q\bigr\}.
\end{IEEEeqnarray*}
For $A\subseteq K$, put
$$\cl_\tau(A):=K\cap(u\circ A)=u(u\circ A).$$
The operator $\cl_\tau$ is a closure operator on $K$. The topology
determined by this closure operator is called the
\emph{$\tau$-topology}.
The $\tau$-topology is coarser than the topology on $K$ inherited
from $E$. Moreover, $(K,\tau)$ is always a quasi-compact
$T_1$ semitopological group, that is, multiplication is separately
continuous. No Hausdorffness is asserted here. See
\cite{Glasner1976} for the classical construction in the
$\beta G$ setting, and
\cite[Appendix A]{rzepecki2018}
for the general formulation used here.

Whenever $(K,\tau)$ is a compact Hausdorff topological group,
$\mathfrak h_K$ denotes its normalized Haar measure.

The next fact follows from Facts 2.3 and 2.5 in \cite{KLM}.

\begin{fact}
\label{fact:ellis-group-functoriality}
Let $(X,G)$ and $(Y,G)$ be flows, and let $\Phi\colon E(X,G)\longrightarrow E(Y,G)$ be a continuous surjective $G$-equivariant semigroup homomorphism.
Let $\CM\unlhd_m E(X,G)$, $u^2=u\in\CM$,
and put $\mathcal N:=\Phi[\CM]$, $v:=\Phi(u)$, $K:=u\CM$, $L:=v\mathcal N=\Phi[K]$.
Then $\mathcal N\unlhd_m E(Y,G)$, $v^2=v\in\mathcal N$,
and the restriction
$$
\Phi|_K\colon (K,\tau)\longrightarrow(L,\tau)
$$
is a continuous group epimorphism and a quotient map for the
$\tau$-topologies.
(In particular, the conclusion applies when $\Phi=\pi_E$ is induced
by a factor map $\pi \colon X\to Y$, and when $Y\subseteq X$ is a subflow
and $\Phi \colon E(X,G)\to E(Y,G)$ is the restriction map.)
\end{fact}

There are many equivalent definitions of tame flows (see Theorems 2.4, 3.2 and Definition 3.1 in \cite{GM23}). We give the one which immediately points to a strong connection with the NIP property in model theory. 

A sequence $(f_n)_{n<\omega}$  of real valued functions on a set $X$ is said to be {\em independent}
if there exist real numbers $a < b$ such that
$$\bigcap_{n\in P} f_n^{-1}[(-\infty,a)] \cap \bigcap_{n \in Q} f_n^{-1}[(b,\infty)] \ne \emptyset$$
for all finite disjoint subsets $P, Q$ of $\omega$.

\begin{definition}\label{def: tame}
%A flow $(G,X)$ is {\em tame} if for every $f \in C(X)$ (i.e. a continuous real valued function on $X$) the family of translates $\{gf: g \in G\}$ does not contain an independent sequence.
Let $(X,G)$ be a flow.
A function $f \in C(X,\mathbb{R})$ (i.e.~a continuous real valued function on $X$) is {\em tame} if the  family of translates $\{gf: g \in G\}$ does not contain an infinite independent sequence. The flow $(X,G)$ is {\em tame} if all functions in $C(X,\mathbb{R})$ are tame.
\end{definition}

Let $(X,G)$ be a flow.  
Put $E:=E(X,G)$, choose $\CM\unlhd_m E$ and $u^2=u\in\CM$.
Choose $x\in uX$, so $ux=x$. We set
$$\gal(X,x):=\{\eta\in u\CM\;:\;\eta(x)=x \}.\footnote{We will be also using the same notation for Galois groups computed in Ellis semigroups of some ambits which project onto $X$.}$$

\begin{remark}\label{rem:Gal.closed.0}
    $\gal(X,x)$ is closed in the $\tau$-topology. This is standard,
    for example, we can re-use the short argument of the first paragraph of \cite[Chapter IX.2]{Glasner1976}.
\end{remark}

\begin{fact}[Theorem~8.2 in \cite{GlasnerMegrelishvili2013}]\label{fact:tame-injective}
If $(Y,G)$ is tame, then the canonical restriction homomorphism
$$\operatorname{res}\colon E(\operatorname{Prob}(Y),G)\longrightarrow E(Y,G)$$
is a topological semigroup isomorphism.
\end{fact}

\begin{fact}[{\cite[Corollary 5.6]{HoffKru1}}]\label{Fact: 5.6}
For a tame flow $(X,G)$, hereditary amenability of $(X,G)$ is equivalent to amenability (and also to hereditary amenability) of $(E(X,G),G)$.
\end{fact}

\begin{fact}\label{fact: amenability closure}
Let $(X,G)$ be a tame hereditary amenable flow.  
Put $E:=E(X,G)$, choose $\CM\unlhd_m E$ and $u^2=u\in\CM$. Then the following statements hold:
\begin{enumerate}
    \item 
    $(\CM,G)$ is amenable.

    \item 
    The Ellis group $u\CM$ is dense in $\CM$.
\end{enumerate}
\end{fact}

\begin{proof}
By Fact \ref{Fact: 5.6}, the Ellis flow $(E,G)$ is
hereditarily amenable, and therefore $\CM$ is amenable.
The second point follows from \cite[Lemma~5.7]{HoffKru1}.
\end{proof}

Recall that a flow $(X,G)$ is {\em proximal} if for every $x_1,x_2 \in X$ there exists $\eta \in E(X,G)$ such that $\eta(x_1)=\eta(x_2)$.

The following easy fact is folklore.
\begin{fact}\label{fact: existence of semigroup operation}
Let $(\mathcal{A},a_0,G)$ be an ambit. Then the following conditions are equivalent.
\begin{enumerate}
\item There exists a left continuous semigroup operation $*$ extending the action of $G$ (i.e., $(ga_0) * q=gq$).
\item There exists a left continuous semigroup operation $*$ on $\mathcal{A}$ such that the function $\Phi \colon \mathcal{A}\to E(\mathcal{A},G)$ defined by $\Phi(p):= l_p$ (where $l_p \colon \mathcal{A} \to \mathcal{A}$ is given by $l_p(q):=p*q$) is a semigroup epimorphism such that $\Phi(ga_0)= g_{\mathcal{A}}$ (where $g_{\mathcal{A}}(a):=ga$); then $\Phi$ is also an isomorphism of $G$-flows.
\item  $(\mathcal{A},a_0,G)$ is isomorphic to $(E(X,G),\id, G)$ for some flow $(X,G)$.
\end{enumerate}
\end{fact}

For further background on tame flows and Ellis semigroups, see
\cite{Auslander,Glasner1976,Glasner1,GM23,CodHoff23}.

\subsection{Operator theory}
Let $(Y,G)$ be a flow and let
$\mu\in\operatorname{Prob}_G(Y)$.  For $1\leqslant p\leqslant\infty$, we write
$L^p(Y,\mu)=L^p(Y,\mu;\mathbb C)$.
The inner product on $L^2(Y,\mu)$ is
$$\langle F_1,F_2\rangle_{L^2(Y,\mu)}=\int_Y F_1(y)\,\overline{F_2(y)}\,d\mu(y),$$
 where $F_1,F_2\in L^2(Y,\mu)$.
For $g\in G$ and $F=[f]\in L^2(Y,\mu)$, define
$U_gF=[f\circ g^{-1}]$.
This is well-defined because $\mu$ is $G$-invariant.  As usual, after
choosing representatives, we write $(U_gF)(y)=F(g^{-1}y)$
with equality understood $\mu$-almost everywhere.

The $G$-invariance of $\mu$ implies that
$$\langle U_gF_1,U_gF_2\rangle_{L^2(Y,\mu)}=\langle F_1,F_2\rangle_{L^2(Y,\mu)}$$
for all $F_1,F_2\in L^2(Y,\mu)$.  Since
$U_g^{-1}=U_{g^{-1}}$, each $U_g$ is unitary.  Moreover,
$U_gU_h=U_{gh}$ for every $g,h\in G$,
so $g\mapsto U_g$ is a unitary representation of $G$ on
$L^2(Y,\mu)$.

For a surjective continuous map $\varphi:Y\to Z$, we denote by $\varphi^*:C(Z)\longrightarrow C(Y)$, $\varphi^*(F):=F\circ\varphi$,
the corresponding pullback map.  If $\mu\in\operatorname{Prob}_G(Y)$
and $\nu=\varphi_*\mu$, then
$$\|\varphi^*F\|_{L^2(Y,\mu)}^2=\int_Y|F\circ\varphi|^2\,d\mu=\int_Z|F|^2\,d\nu=\|F\|_{L^2(Z,\nu)}^2$$
for every $F\in C(Z)$.  Since $C(Z)$ is dense in $L^2(Z,\nu)$,
the pullback extends uniquely to a linear isometry
$$\varphi^*:L^2(Z,\nu)\longrightarrow L^2(Y,\mu).$$
Its range satisfies $\operatorname{ran}(\varphi^*)=\overline{\varphi^*[C(Z)]}$. 
Indeed, this follows from the density of $C(Z)$ in $L^2(Z,\nu)$
and the fact that the range of an isometry is closed.
Therefore, the isometry $\varphi^*$ is a unitary isomorphism if and only if
$\overline{\varphi^*[C(Z)]}= L^2(Y,\mu)$.

Let $\mathcal H$ be a complex Hilbert space and let $B(\mathcal H)$
denote the algebra of bounded linear operators on $\mathcal H$.  The
\emph{strong operator topology}\/ (SOT) on $B(\mathcal H)$ is the
topology generated by the seminorms
$$p_\xi(T):=\lVert T\xi\rVert_{\mathcal H}\qquad(\xi\in\mathcal H).$$
Equivalently, a net $(T_i)_i$ in $B(\mathcal H)$ converges to
$T\in B(\mathcal H)$ in SOT if and only if
$$\lVert T_i\xi-T\xi\rVert_{\mathcal H}\longrightarrow0
 \qquad\text{for every }\xi\in\mathcal H.$$
We equip the unitary group $U(\mathcal H)$ and each of its subgroups
with the corresponding subspace topology.

\begin{definition}[See Section~3.1 in \cite{FollandAHA}]
Let $B$ be a topological group and let $\mathcal H$ be a complex
Hilbert space.  A \emph{unitary representation} of $B$ on
$\mathcal H$ is a group homomorphism
$$\rho:B\longrightarrow U(\mathcal H)$$
which is continuous when $U(\mathcal H)$ is equipped with the strong
operator topology.  Equivalently, for every $\xi\in\mathcal H$, the
orbit map
$$B\longrightarrow\mathcal H,\qquad b\longmapsto\rho(b)\xi,$$
is norm-continuous.  
\end{definition}

In the proof of Lemma \ref{lemma: L2 Bohr}, we define a unitary representation in the following way:
let $B$ be a compact Hausdorff group, and let
$K\leqslant U(\mathcal H)$ carry the subspace topology induced by SOT.
If $\rho_0:B\longrightarrow K$
is a continuous group homomorphism, then the composition
$B\xrightarrow{\ \rho_0\ }K\lhook\joinrel\longrightarrow U(\mathcal H)$
is a unitary representation of $B$ on $\mathcal H$.

\begin{fact}[Theorem~5.2 in \cite{FollandAHA}]
\label{fact: peter-weyl for unitary}
Let $B$ be a compact Hausdorff group.
\begin{enumerate}
\item Every irreducible unitary representation of $B$ is
finite-dimensional.

\item Every unitary representation of $B$ is a Hilbert direct sum of
irreducible unitary representations. More precisely, if $\mathcal H$ is
a complex Hilbert space and
$$\rho:B\longrightarrow U(\mathcal H)$$
is a continuous unitary representation, then there is a family
$(\mathcal H_i)_{i\in I}$ of pairwise orthogonal closed
$\rho(B)$-invariant subspaces such that
$$\mathcal H=\bigoplus_{i\in I}\mathcal H_i$$
and each restricted representation $\rho|_{\mathcal H_i}:B\longrightarrow U(\mathcal H_i)$ is irreducible.
\end{enumerate}
\end{fact}

For a compact group $B$, we write $\mathfrak h_B$ for normalized Haar
probability.  If $H\leqslant B$ is closed, then
$\mathfrak h_{B/H}$ is the pushforward of $\mathfrak h_B$ under the quotient
map $B\to B/H$.  
In the fact below, 
the left-translation condition in (i) is the function-algebra form of the condition called right invariance in the terminology of \cite{TakesakiTatsuuma}.

\begin{fact}[Theorem~2 in \cite{TakesakiTatsuuma}]\label{fact: Takesaki-Tatsuuma Galois}
Let $B$ be a compact Hausdorff group, and let $\mathcal N$ be a weak-$*$ closed unital $*$-subalgebra of $L^\infty(B,\mathfrak h_B)$.  For $b,h\in B$, define
$$(L_bF)(c):=F(b^{-1}c),\qquad (R_hF)(c):=F(ch).$$
Then the following are equivalent:
\begin{enumerate}[(i)]
\item $\mathcal N$ is invariant under every left translation, that is,
$$L_b[\mathcal N]=\mathcal N\qquad\text{for every }b\in B;$$
\item there is a unique closed subgroup $H\leqslant B$ such that
$$\mathcal N=L^\infty(B,\mathfrak{h}_B)^H:=\{F\in L^\infty(B,\mathfrak h_B):R_hF=F\text{ for every }h\in H\}.$$
\end{enumerate}
In this correspondence, the subgroup $H$ is recovered from $\mathcal N$ by
$$H=\{h\in B:R_hF=F\text{ for every }F\in\mathcal N\}.$$
Equivalently, if $p_H:B\to B/H$ is the quotient map, then
$$\mathcal N=p_H^*\bigl[L^\infty(B/H,\mathfrak h_{B/H})\bigr].$$
\end{fact}

\section{Bohr compactifications, equicontinuous factors and category of ambits}

Let $(X,x_0,G)$ be an ambit.

\subsection{Bohr compactifications below ambits}\label{subsection: Bohr comp}

The following definition is a natural extension of the notion of the Bohr compactification of the group $G$. It recovers the usual definition when applied to the universal $G$-ambit.

\begin{definition}[Definition (A)]\label{definition: Bohr compactification below ambit A}
By a {\em compactification of $G$ below $(X,x_0,G)$} 
we mean a group homomorphism $c \colon G \to K$
with dense image for some compact group $K$ 
such that $c=f \circ \ev_{x_0}$, where $\ev_{x_0} \colon G \to X$ is given by $\ev_{x_0}(g):=gx_0$ 
and $f \colon X \to K$ is an epimorphism of $G$-ambits 
(mapping $x_0$ to the neutral element $e_K$ of $K$) with the $G$-action on $K$ given by $g\cdot k:=c(g)k$. 

A {\em Bohr compactification below $(X,x_0,G)$} is a universal compactification $\mathfrak{b}$ of $G$ below $(X,x_0,G)$ in the sense that for any compactification $c \colon G \to K$ of $G$ below $(X,x_0,G)$ there is a (unique) continuous homomorphism $h \colon B \to K$ such that $c = h \circ \mathfrak{b}$.
Abusing terminology, the target group $B$ of $\mathfrak{b}$ may be occasionally also called the Bohr compactification below $(X,x_0,G)$.
\end{definition}

\begin{remark}\label{remark: existence of Bohr compactifications}
For any ambit $(X,x_0,G)$ a Bohr compactification of $G$ below $(X,x_0,G)$ exists and is unique up to isomorphism, so will be called {\em the} Bohr compactification of $G$  below $(X,x_0,G)$.
\end{remark}

\begin{proof}
Uniqueness follows from the density of the image. For the existence, take a set $\{c_i\colon G \to K_i\}_{i \in I}$ of all (up to isomorphism) compactifications of $G$ below $(X,x_0,G)$. Put $K:=\overline{\{(c_i(g))_{i \in I}: g \in G\}} \subseteq \prod_{i \in I} K_i$. It is clear that the map $\mathfrak{b} \colon G \to K$ given by $\mathfrak{b}(g):=(c_i(g))_{i \in I}$ is a compactification below $(X,x_0,G)$ which is universal.
\end{proof}

Notice that the epimorphism $f$ in Definition \ref{definition: Bohr compactification below ambit A} is unique (by the density of $Gx_0$ in $X$ and continuity of $f$). 
So with each compactification $c$ of $G$ below $(X,x_0,G)$ we have a uniquely associated triple $(K,c,f)$. This leads to the following equivalent definition of the Bohr compactification of $G$ below $(X,x_0,G)$ which will be particularly useful in this paper.

\begin{definition}[Definition (B)]\label{definition: Bohr compactification below ambit B}
    We say that a triple $(K,c,f)$ is a \emph{compactification of $G$ below $(X,x_0,G)$} if the following diagram commutes
    $$\xymatrix{G \ar[rr]^{c} \ar[dr]_-{\ev_{x_0}:g\mapsto g(x_0)}& & K  \\ & X \ar[ur]_{f} &},$$
where $K$ is a compact group, $c$ is a group homomorphism, and $f$  is an epimorphism of $G$-ambits (mapping $x_0$ to $e_K$) where the $G$-action on $K$ is given by $g\cdot k:=c(g)k$. 

The \emph{Bohr compactification of $G$ below $(X,x_0,G)$} is a triple $(B,\mathfrak{b},w)$ as above which is universal in the sense that for any other triple $(B',\mathfrak{b'},w')$ as above there exists a (unique) continuous group epimorphism $h\colon B\to B'$ such that the following diagram commutes:
        $$\xymatrix{X \ar[rr]^{w} \ar[dr]_{w'} & & B \ar[dl]^{h} \\ & B' &}$$
\end{definition}

Bearing in mind the paragraph preceding Definition \ref{definition: Bohr compactification below ambit B}, we leave as an easy exercise to check the equivalence of both definitions. Note also that the map $h$ in the above definition is an epimorphism of $G$-ambits.

Moreover, we may describe the Bohr compactification of $G$ below $(X,x_0,G)$ by mimicking some techniques from model theory (cf. \cite{GPPb}).
Let us sketch this approach shortly.

\begin{definition}
\begin{enumerate}
    \item 
    Let $Y$ be a compact Hausdorff space. We call a map 
    $f\colon G\to Y$ {\em $X$-definable} if for every disjoint compact sets $C_0,C_1\subseteq Y$ there is a continuous function $\varphi\colon X\to [0,1]$ such that
    $$f^{-1}[C_0]\subseteq \ev_{x_0}^{-1}[\varphi^{-1}[\{0\}]],
    \qquad
    f^{-1}[C_1]\subseteq \ev_{x_0}^{-1}[\varphi^{-1}[\{1\}]].$$
    \item An $X$-definable homomorphism $f \colon G\to H$,
    where $H$ is a compact Hausdorff group, with $f[G]$ being a dense subset of $H$, is called an {\em $X$-definable compactification of $G$}.
\end{enumerate}
\end{definition}

The universal $X$-definable compactification of $G$ is precisely the Bohr compactification of $G$ below $(X,x_0,G)$. This might be proved by adapting arguments from \cite{GPPb}, but as we do not refer to this approach in what follows, we skip the proof.

Working with a $G$-ambit $(X,x_0,G)$ and $E:=E(X,G)$, we have the Bohr compactification of $G$ below $(X,x_0,G)$:
$$\xymatrix{G \ar[rr]^{\fb_1} \ar[dr]_-{g\mapsto g(x_0)}& & B_1  \\ & X \ar[ur]_{w_1} &}$$
which will be called the \emph{$X$-Bohr compactification} of $G$.
Then we also have the Bohr compactification of $G$ below $(E,\id_X,G)$:
$$\xymatrix{G \ar[rr]^{\fb_2} \ar[dr]_-{g\mapsto g\id_X}& & B_2  \\ & E \ar[ur]_{w_2} &}$$
which will be called the \emph{$E$-Bohr compactification} of $G$.
We keep the above notation ($\fb_1$, $\fb_2$, $w_1$, $w_2$, $B_1$ and $B_2$) for the rest of the paper.

Both $(B_1,e_{B_1},G)$ and $(B_2,e_{B_2},G)$ are $G$-ambits with the $G$-action induced by $\fb_1$ and $\fb_2$, respectively. These flows are minimal equicontinuous (also called almost periodic) and satisfy $E(B_1,G)\cong B_1$ and $E(B_2,G)\cong B_2$.
%Therefore both $(B_1,G)$ and $(B_2,G)$ are almost periodic flows, and in consequence WAP. 
In particular, they are WAP.

\begin{remark}\label{rem: easy}
$w_2 \colon E\to B_2$ is a continuous semigroup epimorphism.
\end{remark}

\begin{proof}
Consider $\eta,\xi\in E$ and a net $(g_i)_{i\in I}\subseteq G$ with $\lim_{i\in I}g_i=\eta$ (which formally means that $\lim_{i \in I} g_i \id_X = \eta$) and simply compute, using left-continuity of the product in $E$:
    \begin{IEEEeqnarray*}{rCl}
    w_2(\eta\xi) &=& w_2\big( (\lim_i g_i)\xi\big)=w_2\big(\lim_i(g_i\xi)\big) \\
    &=& \lim_i w_2(g_i\xi) = \lim_i \fb_2(g_i)w_2(\xi) \\
    &=& \big(\lim_i\fb_2(g_i)\big)w_2(\xi) = \big(\lim_iw_2(g_i \id_X)\big)w_2(\xi) \\
    &=& w_2\big(\lim_i(g_i \id_X)\big)w_2(\xi) = w_2(\eta)w_2(\xi)
    \end{IEEEeqnarray*}
\end{proof}

\begin{remark}\label{remark: B2 to B1}
    Let $\pi_0\colon E\to X$ be given by $\pi_0(\eta):=\eta(x_0)$.
    Then, by universality of $B_2$, there is a unique continuous group epimorphism $\ff \colon B_2\to B_1$ completing the following diagram:
    $$
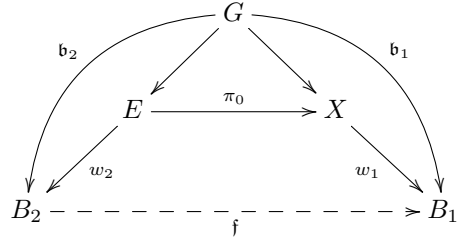
\xymatrix{ & & G \ar[dl] \ar[dr] \ar@/_2pc/[ddll]_-{\fb_2} \ar@/^2pc/[ddrr]^-{\fb_1} & & \\
    & E \ar[rr]^{\pi_0} \ar[dl]^-{w_2}& & X \ar[dr]_-{w_1} & \\
    B_2 \ar@{-->}[rrrr]_{\ff}&&&& B_1}$$
\end{remark}

\begin{remark}
    If $G$ is a topological group and $\mathrm{Sam}(G)$ is its Samuel compactification (\emph{greatest $G$-ambit} in terminology of \cite{Uspenskij}), then the Bohr compactification $B$ of $G$ is the maximal compact Hausdorff group factor of $\mathrm{Sam}(G)$.
    Note also, that $B$ 
    is in fact
    the $\mathrm{Sam}(G)$-Bohr compactification of $G$ and
    also the $E(\mathrm{Sam}(G),G)$-Bohr compactification of $G$.
    %the first and the second Bohr compactifications of $(\mathrm{Sam}(G),G)$, and 
    Moreover, 
    for any ambit $(X,x_0,G)$ we have the following commuting diagram of epimorphisms of $G$-flows (and [semi]groups where it makes sense):
    $$\xymatrix{\mathrm{Sam}(G) \ar@{->>}[r] \ar@{->>}[d]& E \ar@{->>}[r] \ar@{->>}[d]& X\ar@{->>}[d] \\
    B \ar@{->>}[r] & B_2 \ar@{->>}[r] & B_1}$$
\end{remark}

\subsection{Maximal Equicontinuous Factors}\label{subsection: MEFs}

Recall that $(X,x_0,G)$ is a $G$-ambit. Let 
$$(X,G)\xrightarrow[]{\pi}(Z,G)$$
be the Maximal Equicontinuous Factor (MEF) of $(X,G)$, which will be called
the \emph{$X$-MEF} of $(X,G)$ \cite[Chapter 9]{Auslander}. 
Set $K_1:=E(Z,G)\leqslant\homeo(Z)$, which is a compact group \cite[Theorem I.3.3(4)]{Glasner1976}.
Note that $\pi$ lifts up to the Ellis semigroups as
$$\pi_*\colon E\to K_1$$
We need a slight generalization of a classical theorem. The original theorem is stated for minimal flows, but being an ambit is sufficient to repeat the same proof:

\begin{fact}[Theorem I.3.3(5) in \cite{Glasner1976}]\label{fact: equicontinuous ambits}
    Let $(Y,y_0,G)$ be an equicontinuous ambit with $K:=E(Y,G)$.
    Then 
    $$(Y,G)\cong (K/F,G)$$
    for $F:=\{k\in K\;:\;ky_0=y_0\}\leqslant K$.
\end{fact}

By definition, $Z=X/Q^*$, where $Q^*$ denotes the smallest closed $G$-invariant equivalence relation on $X$ such that  the flow $(X/Q^*,G)$ is equicontinuous. 
Recall that $Q^*$ contains the relation $Q$ of regional proximality (cf. Theorem 3 in Chapter 9 of \cite{Auslander}).
By the above fact, $(Z,G)\cong (K_1/F,G)$ for the closed subgroup $F:=\{k\in K_1\;:\;kz_0=z_0\}$, where $z_0:=\pi(x_0)$. 

\begin{remark}\label{rem: KB1}
    By universality of the $X$-MEF, there is a unique epimorphisms of flows $(X/Q^*,G)\to (B_1,G)$, which lifts to the Ellis semigroups as a continuous epimorphism of groups $\mathfrak{t}\colon K_1\to B_1$\footnote{We could try to reverse the last arrow by universality of the Bohr compactification $\mathfrak{b}_1$, but for that we would need a map $X\to K_1$, whereas we only have a map $X\to K_1/F$.}.
\end{remark}

\begin{example}\label{example: MEF1notB1}
[$X$-MEF $\neq$ $X$-Bohr comp.]
 %   This seems to be rather standard, but let us include it anyway.
    Let $X$ be the $2$-dimensional sphere $S^2\cong \mathrm{SO}(3)/
    \mathrm{SO}(2)$. 
    Set $G=\mathrm{SO}(3)$.
    Let $x_0\in X$ be the point $(0,0,1)$.
    We have that $(X,x_0,G)$ is equicontinuous, because $G$ is collection of isometries. 
    Hence, $(X,x_0,G)$ is already its own $X$-MEF.
    Note that $E(X,G)=\mathrm{SO}(3)$, and
    if we apply Fact \ref{fact: equicontinuous ambits}, we indeed obtain that 
    $S^2=X=\mathrm{SO}(3)/\mathrm{SO}(2)$ as $G$-flows.
    Now, let $B_1$ be the Bohr compactification of $G$ below $X$.
    We have that $(B_1,G)$ is an equicontinuous flow and $w_1\colon X\to B_1$
    is an epimorphism of $G$-flows. Therefore, we obtain the following diagram of $G$-epimorphisms:
    $$\xymatrix{\mathrm{SO}(3) \ar[dr]^{f:=w_1\circ\iota} \ar[d]_-{\iota}& \\ \mathrm{SO}(3)/\mathrm{SO}(2)\ar[r]_-{w_1} & B_1}$$
    Note that $f$ is a group epimorphism and $\mathrm{SO}(2)\leqslant\ker f$.
    Because $\mathrm{SO}(3)$ is simple, we have $\ker f=\mathrm{SO}(3)$, and so $B_1=\{e_{B_1}\}$.
\end{example}

By $Q_E^*$ we denote the smallest closed $G$-invariant equivalence relation on $E=E(X,G)$ 
such that the flow $(E/Q^*_E,G)$ is equicontinuous.
(it contains the relation $Q_E$ of regional proximity on $E$).
Then $(E/Q^*_E,G)$ is the MEF of $(E,G)$ which will be called the \emph{$E$-MEF} of $(X,G)$. Set $K_2:=E(E/Q^*_E,G)\leqslant\homeo(E/Q^*_E)$, which is a compact Hausdorff group. 
We fix $u^2=u\in\CM\unlhd_m E$ (an idempotent in a left minimal ideal of $E$). 

\begin{proposition}\label{prop: first iso}
    $(E/Q^*_E,G)\cong(K_2,G)$,
    there exists a unique topological group structure on $E/Q^*_E$ extending the (left) action of $G$ on $E/Q^*_E$,
    and there exists an epimorphism $u\CM/H(u\CM) \twoheadrightarrow K_2$ of compact topological groups.
\end{proposition}

\begin{proof}
    We have $\sigma \colon (E,G)\to (E/Q^*_E,G)$, which lifts up to
    $$\sigma_* \colon E(E,G)\to K_2.$$
    On the other hand, $E\cong E(E,G)$, and the composition of these two maps is a semigroup and $G$-flow epimorphism $\sigma_E\colon E\to K_2$ such that $\sigma_E|_{u\CM}\colon u\CM\to K_2$ is a group epimorphism and topological quotient map with respect to the $\tau$-topologies (using Fact \ref{fact:ellis-group-functoriality}).

    By Fact \ref{fact: equicontinuous ambits}, we have an isomorphism $f\colon (E/Q^*_E,G) \to (K_2/F_E,G)$ given by 
$$f(\xi(\id/Q_E^*)):= \xi/F_E,$$ 
for some closed subgroup $F_E\leqslant K_2$. 
    So $(E,G)\to (K_2/F_E,G)$ is the maximal equicontinuous factor. Note that $E(K_2,G)\cong K_2$ and so $\sigma_E\colon (E,G)\to(K_2,G)$ is also an equicontinuous factor. 
    The following diagram of ambits commutes:
    $$\xymatrix{(E,G) \ar[r]^-{\sigma} \ar[rd]^-{\sigma_E} \ar[d]_-{\eta(\id_X)\mapsto \eta}^-{\cong}& (E/Q_E^*,G) \ar[rrr]^{f,\, \cong} & & & (K_2/F_E,G)\\ E(E,G)\ar[r]_-{\sigma_*}  & (E(E/Q_E^*,G),G)=(K_2,G) \ar[u]_-{\xi\mapsto\xi\big(\sigma(\id_X) \big)} \ar[urrr]_{\xi \mapsto \xi/F_E} }$$
    Hence, by maximality of $(E/Q^*_E,G)\cong(K_2/F_E,G)$, the quotient map $K_2 \to K_2/F_E$ is an isomorphism, so $F_E$ is trivial and $(E/Q^*_E,G)\cong(K_2,G)$. 
    
    So aposteriori there is a unique topological group structure on $E/Q^*_E$ extending the left action of $G$ on $E/Q^*_E$.

    By Corollary A.15 from \cite{rzepecki2018}, we have the following diagram of continuous group epimorphisms:
    $$\xymatrix{u\CM \ar[rr]^{\sigma_E|_{u\CM}} \ar[dr] & & K_2\\ & u\CM/H(u\CM)\ar[ur] &}$$
\end{proof}

\begin{proposition}\label{prop: second iso}
There exists a semigroup and $G$-flow isomorphism  $B_2\cong K_2$.
\end{proposition}

\begin{proof}
    We work with the continuous semigroup epimorphism $\sigma_E$ from the proof of Proposition \ref{prop: first iso}, which is also a $G$-ambit epimorphism.

    The $G$-ambit epimorphism $w_2\colon (E,G)\to (B_2,G)$ is in fact an equicontinuous factor. 
By the proof of Proposition \ref{prop: first iso}, $\sigma_E$ is the MEF of $(E,G)$.
 Thus,   there is a unique epimorphism of $G$-ambits $\ff_1 \colon (K_2,G)\to(B_2,G)$ such that the following diagram commutes:
    $$\xymatrix{(E,G) \ar[d]_-{\sigma_E} \ar[dr]^-{w_2} & \\ (K_2,G)
    \ar@{-->}[r]_-{\ff_1} & (B_2,G).}$$
Since $\sigma_E$ and $w_2$ are semigroup epimorphisms (by Remark \ref{rem: easy}), so is $\ff_1$. 
  
    On the other hand, the existence of $\sigma_E$ allows us to use the universality of the Bohr compactification below $E$ and find a unique continuous epimorphism $\ff_2\colon B_2\to K_2$ making the following diagram commutative (where $\iota \colon G \to E$ is the usual map given by $\iota(g):=g \id_{X}$):
        $$\xymatrix{ & G \ar[d]^-{\iota}  \ar@/_1pc/[ddl]_-{\fb_2} \ar@/^1pc/[ddr]^-{\sigma_E\circ\iota}& \\
        & E \ar[dl]^-{w_2} \ar[dr]_-{\sigma_E}& \\
        B_2 \ar@{-->}[rr]_-{\ff_2} & & K_2}$$
    Note that $\ff_2$ is also an epimorphism of $G$-ambits, $(B_2,G)\to(K_2,G)$.

Combining the universality properties of the $E$-MEF and the Bohr compactification below $E$, gives us
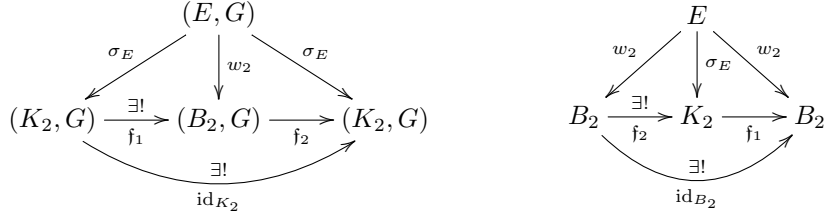
\begin{figure}[h]
\centering
\begin{minipage}{.5\textwidth}
  \centering
  $$\xymatrix{ & (E,G) \ar[dl]_-{\sigma_E} \ar[d]^-{w_2} \ar[dr]^-{\sigma_E} & \\ (K_2,G)\ar[r]^-{\exists!}_-{\ff_1} \ar@/_2pc/[rr]^-{\exists!}_-{\id_{K_2}} & (B_2,G) \ar[r]_-{\ff_2} & (K_2,G)}$$
  \label{fig:test1}
\end{minipage}%
\begin{minipage}{.5\textwidth}
  \centering
  $$\xymatrix{& E \ar[dl]_-{w_2} \ar[d]^-{\sigma_E} \ar[dr]^-{w_2} & \\ 
  B_2 \ar[r]^-{\exists!}_-{\ff_2}  \ar@/_2pc/[rr]^-{\exists!}_-{\id_{B_2}} & K_2 \ar[r]_-{\ff_1} & B_2}$$
  \label{fig:test2}
\end{minipage}
\end{figure}

\noindent
Therefore $\ff_2\ff_1=\id_{K_2}$ and $\ff_1\ff_2=\id_{B_2}$.
\end{proof}

Let us comment and provide some intuition.
On the level of $(X,G)$, by Remark \ref{rem: KB1}, we have the epimorphism $\mathfrak{t}\colon K_1 \to B_1$, but 
we lack a flow epimorphism from $(X,G)$ to $(K_1,G)$ to make $\mathfrak{t}$ an isomorphism, and Example \ref{example: MEF1notB1} shows that $\mathfrak{t}$ need not be an isomorphism. 
However, on the level of $(E,G)$, we actually pass over this issue and obtain an isomorphism $K_2\cong B_2$. 
Let $\ff \colon B_2\to B_1$ be the continuous group epimorphism from Remark \ref{remark: B2 to B1}.
Recall the Ellis semigroup and $G$-flow epimorphism $\pi_* \colon E \to K_1$ from the beginning of Subsection \ref{subsection: MEFs}. By universality of $(B_2,\mathfrak{b}_2,w_2)$, there is a unique continuous group epimorphism $h \colon B_2 \to K_1$ such that $\pi_*=h \circ w_2$. 
The following diagram of continuous group epimorphism commutes:
$$\xymatrix{ B_1 & B_2 \ar[dl]_-{h} \ar[l]_-{\ff} & \\
K_1 \ar[u]^-{\mathfrak{t}} & K_2 \ar[u]_-{\ff_1,\,\cong}& u\CM/H(u\CM) \ar[l]}$$
Note that there is no need to introduce higher Bohr compactifications or higher MEFs of $(X,G)$, simply because $E\cong E(E,G)$.

\begin{cor}\label{cor: all the same}
    If there exists a left-continuous semigroup operation on $X$
    extending the action of $G$ on $X$, then $B_2\cong B_1\cong K_1\cong K_2$.
    Thus, if $K_1\not\cong B_1$, then such a semigroup operation does not exist.
\end{cor}

\begin{proof}
By assumption and Fact \ref{fact: existence of semigroup operation}, the ambits $(X,x_0,G)$ and $(E(X,G),\id_X,G)$ are isomorphic, so  $B_2\cong B_1$ and $K_2 \cong K_1$. The fact that $B_2 \cong K_2$ follows from Proposition \ref{prop: second iso}.
\end{proof}

\subsection{Category of ambits below an Ellis flow}\label{subs:category.below}

Our goal is to develop some topological dynamics ``below'' any given $G$-ambit  $(\mathcal{A},a_0,G)$ on which there is a left continuous semigroup operation extending the action of $G$. This generalizes the classical topological dynamics which takes place below the universal $G$-ambit. Some basic issues are similar, but there are also essential new obstacles.  An important motivation is that one can impose interesting dynamical conditions (e.g., tameness) on $(\mathcal{A},a_0,G)$ which are almost never satisfied by the universal $G$-ambit.

Let $(\mathcal{A},a_0,G)$ be an ambit on which there exists a left continuous semigroup operation $*$ extending the action of $G$ (i.e. $(ga_0) * q=gq$). Recall that, by Fact \ref{fact: existence of semigroup operation}, the existence of a left continuous semigroup operation $*$ on $\mathcal{A}$ extending the action of $G$ is equivalent to saying that $(\mathcal{A},a_0,G)$ is isomorphic to $(E(X,G),\id, G)$ for some flow $(X,G)$.

We will be interested in the category $\mathcal{C}$ of ambits which are $G$-factors of $(\mathcal{A},a_0,G)$. 
Note that the semigroup $(\mathcal{A},*)$ acts naturally on the left on every ambit in $\mathcal{C}$. Indeed, take an ambit $(X,x_0,G)$ from $\mathcal{C}$. By Remark \ref{remark: the action of A on ambits from C}, for any $x \in X$ there exits a flow homomorphism $\pi_x \colon \mathcal{A} \to X$ with $\pi_x(a_0)=x$. We define the action of $(\mathcal{A},*)$ on $X$ by $a  \bullet x:=\pi_x(a)$ (associativity is left as an exercise). This action is left continuous and extends the action of $G$ (i.e, $(ga_0) \bullet x=gx$), because $\pi_x$ is a flow homomorphism. From now on, we will skip $\bullet$.

\begin{remark}\label{remark: the action of A on ambits from C}
If $(X,x_0,G)$ belongs to $\mathcal{C}$, then for every $x \in X$ the ambit $(\overline{Gx},x,G)$ also belongs to $\mathcal{C}$. 
\end{remark}

\begin{proof}
Take an epimorphism $f \colon (\mathcal{A},a_0) \to (X,x_0)$ of $G$-ambits. There is $a_x \in \mathcal{A}$ which is mapped to $x$. We can find a homomorphism (of pointed $G$-flows) $h \colon (\mathcal{A},a_0) \to (\mathcal{A},a_x)$ (just $z \mapsto z *a_x$). Then the composition $f\circ h \colon (\mathcal{A},a_0) \to (\overline{Gx},x)$ witnesses that $(\overline{Gx},x,G)$ belongs to $\mathcal{C}$. 
\end{proof}

By the above remark, abusing terminology, we can say that a minimal $G$-flow $M$ belongs to $\mathcal{C}$ if $(M,m,G)$ belongs to $\mathcal{C}$ for some (equivalently, every) $m \in M$. 
It is very easy to see that any minimal subflow of a product of minimal $G$-flows in the category $\mathcal{C}$ still belongs to $\mathcal{C}$. 
Thus, standard proofs (cf. \cite[Proposition II.4.2]{Glasner1976} and \cite[Remark 1.22]{Krupinski2017}) yield  the following fact (where uniqueness is up to isomorphism). Item (3) is a particular case of Remark \ref{remark: existence of Bohr compactifications}.

\begin{remark}
\begin{enumerate}
\item There exists a unique universal minimal proximal $G$-flow $(\pi_{\mathcal{A}},G)$ in $\mathcal{C}$.
\item For a given minimal $G$-flow $(X,G)$ in $\mathcal{C}$ there exists a unique universal group extension of $X$ in the category $\mathcal{C}$. Denote the universal group extension of  $(\pi_{\mathcal{A}},G)$ in $\mathcal{C}$ by $\psi_{\mathcal{A}} \colon (\pi^{\#}_{\mathcal{A}}, G, \Sigma(G)_{\mathcal{A}}) \to (\pi_{\mathcal{A}},G)$.
\item 
There exists a unique Bohr compactification $\mathfrak{b}_{\mathcal{A}} \colon G \to G^{\textrm{Bohr}}_{\mathcal{A}}$ of $G$ with the induced minimal flow $(G^{\textrm{Bohr}}_{\mathcal{A}},G)$ belonging to $\mathcal{C}$. 

\item
Let $(X,x_0,G)$ be a $G$-ambit and set $(\mathcal{A},a_0,G)=(E(X,G),\id_X,G)$.
We have that $\fb_{\mathcal{A}}\colon G\to G^{\textrm{Bohr}}_{\mathcal{A}}$ coincides with the Bohr compactification of $G$ below $(E(X,G),\id_X,G)$, i.e. with $\fb_2\colon G\to B_2$.
\end{enumerate}
\end{remark}
 
A {\em compactification flow} is a group extension $(X,G,K)$ of a minimal proximal flow. Then the compact group $K$ witnessing the group extension is precisely $\Aut(X,G)$ with the pointwise convergence topology and with the opposite composition. (See the paragraph after Def. 1.23 in \cite{Krupinski2017} and the references in there.)
The flow $(\pi^{\#}_{\mathcal{A}}, G)$ is the {\em universal compactification flow of $G$ in $\mathcal{C}$} (see Proposition \ref{proposition: univ. comp. flow}), and $\Sigma(G)_{\mathcal{A}}$ will be called the {\em generalized Bohr compactification of $G$ in $\mathcal{C}$} (it is uniquely determined as $\Aut(\pi^{\#}_{\mathcal{A}}, G)$ with the opposite composition).

Let us fix a minimal left ideal $\mathcal{M}$ in $\mathcal{A}$ and an idempotent $u \in \mathcal{M}$. Recall that $\mathcal{A}$ acts transitively and left-continuously on any minimal flow from $\mathcal{C}$ by an action which extends the action of $G$ (i.e., $(ga_0)x=gx$). Choose $b_0 \in \pi^{\#}_{\mathcal{A}}$ with $ub_0=b_0$ ($b_0:=ub$ for any choice of $b\in \pi^{\#}_{\mathcal{A}}$ has this property). Let 
$$\gal_{\mathcal{A}}=\gal(\pi^{\#}_{\mathcal{A}},b_0):=\{ \eta \in u\mathcal{M}: \eta b_0 = b_0\}.$$

In the classical case when $\mathcal{A}=\beta G$ (i.e.,  the universal $G$-ambit), Theorem IX.4.2  of \cite{Glasner1976} describes $\pi^{\#}_{\mathcal{A}}$ as a quotient of $\mathcal{M}$ by a certain closed $G$-invariant equivalence relation $\sim$, and $\Sigma(G)_{\mathcal{A}}$ as $u\mathcal{M}/H(u\mathcal{M})$. However, in general, this approach fails, as the relation $\sim$ need not be closed. It is an interesting general question how to describe  both $\pi^{\#}_{\mathcal{A}}$ and $\Sigma(G)_{\mathcal{A}}$. In a model-theoretic category of externally definable flows there was an attempt to do that in Theorem 2.5 of \cite{Krupinski2017}, but recently Andy Zucker found a gap in the proof (there is no reason for the relation $\mathcal{M}$ defined there to be an equivalence relation). Using the ideas around Theorem 2.5 of \cite{Krupinski2017}, the second author has found some description of the universal compactification flow of $G$ in $\mathcal{C}$, which shows that under some special assumptions the generalized Bohr compactification in $\mathcal{C}$ is $u\mathcal{M}/H(u\mathcal{M})$, but it does not prove it in general; we do not present in this paper those partial results. We expect that it may be false in general. However, in the situation in which we are interested in this paper (i.e. for a tame and amenable ambit $\mathcal{A}$), it will be true.

By the proof of \cite[Theorem IX.2.1]{Glasner1976} or just using this theorem and passing to the suitable quotients of Ellis groups (the difference is that in \cite[Theorem IX.2.1]{Glasner1976} we compute all the Ellis groups and Galois groups in the universal $G$-ambit $\beta G$, whereas here in the ambit $\mathcal{A}$ which is a quotient of $\beta G$), we get the following fact.

\begin{fact}\label{fact: identification of Sigma and uI/Gal}
There is a topological group isomorphism $\kappa \colon u \mathcal{M}/\!\gal_{\mathcal{A}} \to \Sigma(G)_{\mathcal{A}}$ determined by $(pb_0)\cdot \kappa(\eta/\!\gal_{\mathcal{A}})=(p\eta) b_0$ for $p \in  \mathcal{M}$. Moreover, $H(u\mathcal{M}) \leqslant \gal_{\mathcal{A}}$.
\end{fact}

For the next fact see \cite[Lemma VIII.1.1]{Glasner1976}.

\begin{fact}\label{fact: homomorphism theta}
Let $\chi \colon (Z,G,K) \to (Y,G)$ and $\phi \colon (X,G,L) \to (Y,G)$ be group extensions and assume that there is a homomorphism $\psi \colon Z \to X$ such that $\chi = \phi \circ \psi$. Then there is a topological group epimorphism $\theta \colon K \to L$ determined by $\psi(z) \cdot \theta(k)= \psi(z\cdot k)$. 
\end{fact}

\begin{proposition}\label{proposition: univ. comp. flow}
\begin{enumerate}
\item 
$(\pi^{\#}_{\mathcal{A}},G)$ is the unique universal compactification $G$-flow in $\mathcal{C}$.

\item 
If $(X,G)$ is a compactification flow in $\mathcal{C}$, then there is a continuous epimorphism $\theta$ from  the compact group $\Sigma(G)_{\mathcal{A}}$ to the compact group $\Aut(X,G)$. 
More precisely, bearing in mind Fact \ref{fact: identification of Sigma and uI/Gal}, if we choose $x_0 \in X$ such that $ux_0=x_0$, then we can find the above $\theta$ determined by $(px_0)\cdot \theta(\eta/\!\gal_{\mathcal{A}}) = (p \eta) x_0$ for $p \in \mathcal{M}$.
\end{enumerate}
\end{proposition}

\begin{proof}
(1) Let $(X,G)$ be a compactification flow in $\mathcal{C}$. So there is a minimal proximal flow $(Y,G)$ and a group extension $\phi \colon (X,G,K) \to (Y,G)$. Since $(Y,G)$ belongs to $\mathcal{C}$, by universality of $\pi_{\mathcal{A}}$, we get an epimorphism $\chi \colon \pi_{\mathcal{A}} \to Y$. Recall also that we have the universal group extension $\psi_{\mathcal{A}} \colon (\pi^{\#}_{\mathcal{A}},G, \Sigma(G)_{\mathcal{A}}) \to (\pi_{\mathcal{A}},G)$.

Set $Z':= \{ (x,p) \in X \times \pi_{\mathcal{A}}: \phi(x) = \chi(p)$\}. 
Pick $x_0 \in X$ and $p_0 \in \pi_{\mathcal{A}}$ so that $ux_0=x_0$ and $up_0=p_0$. 
Then $u(x_0,p_0)=(x_0,p_0)$, so $Z:=\overline{G(x_0,p_0)}$ is a minimal subflow of $Z'$ (to see that $Z$ is a subflow of $Z'$, note that $(x_0,p_0) \in Z'$; indeed, the elements $\phi(x_0), \chi(p_0) \in uY$ are distal by \cite[Proposition I.3.2(3)]{Glasner1976}, and so they are equal by proximality of $(Y,G)$). Since both $(X,G)$ and $(\pi_{\mathcal{A}},G)$ belong to $\mathcal{C}$, so does $(Z,G)$. We have a natural faithful right action of $K$ on $Z$ given by $(x,p) \cdot k:=(x \cdot k,p)$.
This action is jointly continuous, and each $k\in K$ acts by a $G$-flow automorphism.
Let $\pi_1 \colon Z \to X$ and $\pi_2 \colon Z \to \pi_{\mathcal{A}}$ be the projections.

\begin{clm}
$\pi_2$ is a group extension with the above action of $K$ as a witness.
\end{clm}

\begin{clmproof}
Consider any $(x,p) \in Z$. Then $\pi_2^{-1}[\pi_2(x,p)]=\pi_2^{-1}(p)$, so the goal is to show that $\pi_2^{-1}(p) =(x,p) \cdot K$.

($\subseteq$) Take any $(x',p) \in \pi_2^{-1}(p)$. Since $(x,p), (x',p) \in Z$, we have $\phi(x')= \chi(p)=\phi(x)$. Thus, as $\phi$ is a group extension witnesses by $K$, there exists $k \in K$ such that $x'=x \cdot k$. Then $(x',p) = (x,p) \cdot k \in (x,p) \cdot K$.

($\supseteq$)  Take any $(x',p) \in (x,p) \cdot K$, so $x' \in x \cdot K = \phi^{-1}[\phi(x)]$ as $\phi$ is a group extension witnessed by $K$. Choose an idempotent $v \in \mathcal{M}$ such that $v(x,p)=(x,p)$ (such a $v$ exists by minimality of $Z$). Since $(Y,G)$ is proximal, Proposition I.4.1(2) from \cite{Glasner1976} yields $v\mathcal{M}\phi(x) =\{\phi(x)\}$. Hence, by \cite[Proposition I.4.1(3)]{Glasner1976}, we conclude that $\phi^{-1}[\phi(x)] = v\mathcal{M} x$ (we use here that group extensions are distal homomorphisms). All of this implies that $x' \in v\mathcal{M} x$. On the other hand, since $(\pi_{\mathcal{A}},G)$ is proximal, again by  \cite[Proposition I.4.1(2)]{Glasner1976}, we get that $v\mathcal{M} p =\{p\}$. Therefore, $(x',p) \in (v\mathcal{M} x,p)=v\mathcal{M}(x,p) \subseteq Z$ and clearly $\pi_2(x',p)=p$, so $(x',p) \in \pi_2^{-1}(p)$.
\end{clmproof}

By the claim, the fact that $(Z,G)$ is in $\mathcal{C}$, and by universality of  $\psi_{\mathcal{A}}$, there exists an epimorphism $\Theta \colon \pi^{\#}_{\mathcal{A}} \to Z$ with $\pi_2 \circ \Theta = \psi_{\mathcal{A}}$. Then $\pi_1 \circ \Theta \colon \pi^{\#}_{\mathcal{A}} \to X$ is an epimorphism, which witnesses universality of  $(\pi^{\#}_{\mathcal{A}},G)$.

Uniqueness follows from an easy observation that every endomorphism of a compactification flow is an automorphism.

(2)
Consider any $x_0 \in X$ with $ux_0=x_0$.
Let $p_0 = \psi_{\mathcal{A}}(b_0) \in \pi_{\mathcal{A}}$; then $up_0=p_0$. The construction in the proof of (1) yields an epimorphism $\Theta \colon \pi^{\#}_{\mathcal{A}} \to Z$ with $\pi_2 \circ \Theta = \psi_{\mathcal{A}}$. 
We can find $b_0' \in \pi^{\#}_{\mathcal{A}}$ satisfying $ub_0'=b_0'$ and $\Theta(b_0')=(x_0,p_0)$. Since $\psi_{\mathcal{A}}(b_0') = \pi_2(\Theta(b_0')) = \pi_2(x_0,p_0)=p_0 = \psi_{\mathcal{A}}(b_0)$ and $\psi_{\mathcal{A}}$ is a group extension, there exists an automorphism $\sigma$ of the flow $(\pi^{\#}_{\mathcal{A}},G)$ mapping $b_0$ to $b_0'$. Then $\Theta':=\Theta \circ \sigma \colon  \pi^{\#}_{\mathcal{A}} \to Z$ is an epimorphism satisfying $\Theta'(b_0) = (x_0,p_0)$. This implies that $\pi_2(\Theta'(b_0))=p_0 = \psi_{\mathcal{A}}(b_0)$, and so $\pi_2 \circ \Theta' = \psi_{\mathcal{A}}$.

By Fact \ref{fact: identification of Sigma and uI/Gal}, we can identify $\Sigma(G)_{\mathcal{A}}$ with $u\mathcal{M}/\!\gal_{\mathcal{A}}$ and then the right action on $\pi^{\#}_{\mathcal{A}}$ is given by $(pb_0)\cdot (\eta/\!\gal_{\mathcal{A}})=(p\eta) b_0$ for $p \in \mathcal{M}$.

Since $\pi_2$ and $\psi_{\mathcal{A}}$ are group extensions, by Fact \ref{fact: homomorphism theta}, there is a continuous epimorphism $\theta \colon \Sigma(G)_{\mathcal{A}} \to K$ given by $\Theta'(pb_0) \cdot \theta (\eta/\! \gal_{\mathcal{A}}) =\Theta' (pb_0 \cdot (\eta/\! \gal_{\mathcal{A}}))$ which equals $\Theta'((p\eta)b_0)$ by the previous paragraph. But $\Theta'(pb_0)=p\Theta'(b_0)=p(x_0,p_0)=(px_0,pp_0)$ and $\Theta'((p\eta)b_0) = ((p\eta)x_0,(p\eta)p_0)$. So 
\begin{IEEEeqnarray*}{rCl}
((p\eta)x_0,(p\eta)p_0) &=& \Theta'(pb_0) \cdot \theta (\eta/\! \gal_{\mathcal{A}})= (px_0,pp_0)\cdot \theta (\eta/\! \gal_{\mathcal{A}}) \\
&=& ((px_0) \cdot \theta (\eta/\! \gal_{\mathcal{A}}),pp_0)    
\end{IEEEeqnarray*}
which yields $(px_0) \cdot \theta (\eta/\! \gal_{\mathcal{A}}) = (p\eta)x_0$, as required.
\end{proof}

\begin{cor}\label{corollary: Newelski's map in general}
Let $e_B$ be the neutral element in $G^{\textrm{Bohr}}_{\mathcal{A}}$. The map $\theta  \colon u\mathcal{M}/\!\gal_{\mathcal{A}} \to G^{\textrm{Bohr}}_{\mathcal{A}}$ given by $\theta(\eta/\! \gal_{\mathcal{A}}):= \eta e_B$
is a well-defined continuous epimorphism, which satisfies $\theta(ugu/\!\gal_{\mathcal{A}}) = \mathfrak{b}_{\mathcal{A}}(g)$ for all $g \in G$.  Since $H(u\mathcal{M})$ is contained in $\gal_{\mathcal{A}}$, this yields a continuous epimorphism $u\mathcal{M}/H(u\mathcal{M}) \to G^{\textrm{Bohr}}_{\mathcal{A}}$. This also yields an epimorphism and topological quotient map $\widehat{\theta}\colon u\mathcal{M} \to G^{\textrm{Bohr}}_{\mathcal{A}}$ defined as the composition of $\theta$ with the quotient map $u\mathcal{M} \to u\mathcal{M}/\!\gal_{\mathcal{A}}$.
\end{cor}

\begin{proof}
Let $(X,G):= (G^{\textrm{Bohr}}_{\mathcal{A}},G)$, where the action is given by $gx:=\mathfrak{b}_{\mathcal{A}}(g)x$.
Note that $X$ acts transitively, jointly continuously and by automorphisms on the flow $(X,G)$ by multiplying on the right (i.e., for every $h \in X$ the assignment $x \mapsto xh$ is an automorphism of $(X,G)$). By minimality of $(X,G)$, this yields an identification of $X$ with $\Aut(X,G)$. In particular, $(X,G)$ is a group extension of the trivial flow, so it is a compactification flow. 

We check that $uh=h$ for every $h \in X$. There exists $h_0 \in X$ with $uh_0=h_0$. Then $uh=u(h_0(h_0^{-1}h))= (uh_0)(h_0^{-1}h)=h_0h_0^{-1}h=h$, where the second equality follows from the fact that the assignment $x \mapsto x(h_0^{-1}h)$ is an automorphism of $(X,G)$.

By the above two paragraphs, we can apply Proposition \ref{proposition: univ. comp. flow}(2) to the flow $(X,G)$ and $x_0:=e_B$. Bearing in mind the above identification of $\Aut(X,G)$ with $X$, the map $\theta \colon u\mathcal{M}/\! \gal_{\mathcal{A}} \to \Aut(X,G)$ from Proposition \ref{proposition: univ. comp. flow}(2) is in fact a continuous epimorphism onto $X$. By the last formula in Proposition \ref{proposition: univ. comp. flow}(2) and the identification of $\Aut(X,G)$ with $X$, we see that $\theta(\eta/\! \gal_{\mathcal{A}})= x_0\cdot \theta(\eta/\!\gal_{\mathcal{A}})=\eta x_0$. Finally, by the second paragraph, $\theta(ugu/\!\gal_{\mathcal{A}}) = ugu x_0=ug x_0=gx_0= \mathfrak{b}_{\mathcal{A}}(g) x_0 =\mathfrak{b}_{\mathcal{A}}(g)$.

Since $\theta$ is continuous and $u\mathcal{M}/\!\gal_{\mathcal{A}}$ and $G^{\textrm{Bohr}}_{\mathcal{A}}$ are compact Hausdorff, we have that $\theta$ is a topological quotient map, and thus so is $\widehat{\theta}$.
\end{proof}

\begin{question}
Is the set $\{ugu/\!\gal_{\mathcal{A}}: g \in G\}$ dense in $u\mathcal{M}/\!\gal_{\mathcal{A}}$? Whenever this is the case,  the epimorphism $\theta$ from Corollary  \ref{corollary: Newelski's map in general} is the unique continuous epimorphism $u\mathcal{M}/\!\gal_{\mathcal{A}} \to G^{\textrm{Bohr}}_{\mathcal{A}}$ satisfying $\theta(ugu/\!\gal_{\mathcal{A}}) = \mathfrak{b}_{\mathcal{A}}(g)$ for all $g \in G$.
\end{question}

By minimality of $\mathcal{M}$, we get that $Gu$ is dense in $\mathcal{M}$. 
Thus, in the case when $\overline{u\mathcal{M}}=\mathcal{M}$, a positive answer to the last question follows from \cite[Lemma 3.1]{KruPiRze}. However, this argument does not work if $\overline{u\mathcal{M}} \subsetneq \mathcal{M}$.

Note that if the map $\theta \colon u\mathcal{M}/\!\gal_{\mathcal{A}} \to G^{\textrm{Bohr}}_{\mathcal{A}}$ from Corollary  \ref{corollary: Newelski's map in general} is an isomorphism, then the map $\zeta \colon G \to u\mathcal{M}/\! \gal_{\mathcal{A}}$ given by $g \mapsto ugu/\! \gal_{\mathcal{A}}$ is the Bohr compactification of $G$ in $\mathcal{C}$ (in particular, $\zeta$ has dense image). 

Now, we present a generalization to $\mathcal{C}$ of a more precise version of \cite[Corollary IX.4.3]{Glasner1976} (note that \cite[Corollary IX.4.3]{Glasner1976} was left without a proof in \cite{Glasner1976}).

\begin{proposition}\label{proposition: Glasner below A}
Assume that the ambit $\mathcal{A}$ is {\em strongly amenable} in the sense that it has no non-trivial minimal proximal factors. Then the map $\theta \colon u\mathcal{M}/\!\gal_{\mathcal{A}} \to G^{\textrm{Bohr}}_{\mathcal{A}}$ from Corollary  \ref{corollary: Newelski's map in general} is an isomorphism.
\end{proposition}

\begin{proof}
By strong amenability, $\pi_{\mathcal{A}}$ (i.e., the universal minimal proximal ambit in $\mathcal{C}$) is trivial. So $(Y,G,K):=(\pi^{\#}_{\mathcal{A}}, G, \Sigma(G)_{\mathcal{A}})$ is the universal group extension in $\mathcal{C}$ of the trivial flow. So $Y=b_0 \cdot K$ and the map $\Phi \colon K \to Y$ given by $\Phi(k):=b_0 \cdot k$ is a homeomorphism (where recall that $b_0 \in Y$ was chosen so that $ub_0=b_0$). This induces a compact group structure on $Y$, that is $(b_0\cdot k_1) \odot (b_0\cdot k_2):=b_0\cdot (k_1k_2)$. For every $g \in G$ there exists $k_g \in K$ such that $gb_0=b_0 \cdot k_g$. A direct computation shows that for any $g_1,g_2 \in G$, we have $(g_1b_0) \odot (g_2b_0)=(g_1g_2)b_0$, and so $k_{g_1} k_{g_2}=k_{g_1 g_2}$. Thus, $k \colon G \to K$ given by $k(g):=k_g$ is a homomorphism with dense image (as $\Phi[\im(k)]=b_0\cdot\im(k) =Gb_0$ and $Gb_0$ is dense in $Y$), i.e. $k$ is a compactification of $G$. 
The $G$-flow structure on $K$ induced by the homeomorphism $\Phi$ is $g a:=\Phi^{-1}(g(b_0\cdot a))=\Phi^{-1}((gb_0)\cdot a))=\Phi^{-1}(b_0\cdot (k_ga)) = k_g a$, and so we see that it coincides with the $G$-flow structure induced by the compactification $k$. Since $(K,G) \cong (Y,G)$ we have that that $(K,G)$ is in $\mathcal{C}$, so $K$ is a compactification of $G$ in $\mathcal{C}$. By universality of the Bohr compactification, there is a unique topological group epimorphism $h \colon G^{\textrm{Bohr}}_{\mathcal{A}} \to K$ such that $h \circ \mathfrak{b}_{\mathcal{A}} =k$.

On the other hand, by Fact \ref{fact: identification of Sigma and uI/Gal}, we have a topological group isomorphism $\kappa \colon u\mathcal{M}/\!\gal_{\mathcal{A}} \to K$ determined by $(pb_0)\cdot \kappa(\eta/\!\gal_{\mathcal{A}})=(p\eta) b_0$ for $p \in \mathcal{M}$. For $g \in G$ we compute: 
\begin{IEEEeqnarray*}{rCl}
    b_0 \cdot \kappa(ugu/\gal_{\mathcal{A}}) &=& (ub_0) \cdot \kappa(ugu/\gal_{\mathcal{A}}) =ugub_0=ugb_0 \\ &=& u(b_0 \cdot k_g)=(ub_0) \cdot k_g= b_0 \cdot k_g,
\end{IEEEeqnarray*} 
and so $\kappa(ugu/\gal_{\mathcal{A}}) = k_g =k(g)$.

We conclude that for any $g \in G$, $(\kappa^{-1} \circ  h \circ \mathfrak{b}_{\mathcal{A}})(g) = ugu/\gal_{\mathcal{A}}$. Hence, by Corollary  \ref{corollary: Newelski's map in general}, we get $\theta \circ \kappa^{-1} \circ h \circ \mathfrak{b}_{\mathcal{A}} =  \mathfrak{b}_{\mathcal{A}}$, so $\theta \circ \kappa^{-1} \circ h = \id_{G^{\textrm{Bohr}}_{\mathcal{A}}}$, so $\theta$ is injective by surjectivity of $\kappa^{-1} \circ h$.
\end{proof}

Now, we explain that the map $\widehat{\theta}$ from Corollary \ref{corollary: Newelski's map in general} coincides with $w_2$ from Subsection \ref{subsection: Bohr comp}.

\begin{remark}\label{remark: w2=theta}
Let $(\mathcal{A},a_0,G)=(E(X,G),\id_X,G)$ for an ambit $(X,x_0,G)$. Then for
the map $\theta$ from Corollary \ref{corollary: Newelski's map in general} we have
$\theta(\eta/\!\gal_{\mathcal{A}})=w_2(\eta)$, i.e. $\eta e_B=w_2(\eta)$.
Hence, the following diagram commutes
$$\xymatrix{u\CM \ar[r]^-{w_2|_{u\CM}} \ar[d] & B_2 \\ u\CM/\!\gal_{\mathcal{A}} \ar[ur]_-{\theta} &}$$
and $\widehat{\theta}=w_2|_{u\CM}$. Thus, $w_2|_{u\CM}$ is a topological quotient map.
\end{remark}

\begin{proof}
For every $g\in G$ we have $w_2(g\id_X)=\fb_2(g)=\fb_2(g)e_B=g e_B$, and then we use standard continuity and density arguments.
\end{proof}

In Proposition \ref{proposition: Glasner below A}, we obtained that the map $\theta$ is an isomorphism under suitable assumptions on the ambit $\mathcal{A}$. 
In the next section, we give assumptions on the ambit $\mathcal{A}$ guaranteeing that the map $\widehat{\theta}=w_2|_{u\CM}$ is an isomorphism.

\section{The generalized Newelski's conjecture}

\subsection{Formulation}\label{subsection: Formulation of GNC}
We keep the notation from the previous section. In particular, for a flow $(X,G)$ and $u^2=u \in \mathcal{M}\unlhd_m E(X,G)$, we have the Bohr compactification $\mathfrak{b}_2 \colon G \to B_2$ and the group epimorphism $\widehat{\theta}=w_2|_{u\CM} \colon u\CM \to B_2$ which is also a topological quotient map (see Remark \ref{remark: w2=theta}).

\begin{conj}[GNC - Generalized Newelski's Conjecture]\label{conj: GNC}
Let $(X,G)$ be a flow.
Then, under suitable assumptions on the flow $(X,G)$, 
the map $w_2|_{u\CM}\colon u\CM\to B_2$ is a topological group isomorphism.
\end{conj}

As in model theory, we have 

\begin{remark}
The conclusion of Conjecture \ref{conj: GNC} (i.e., $w_2|_{u\CM} \colon u\CM\to B_2$ is a topological group isomorphism) is equivalent to injectivity of $w_2|_{u\CM} \colon u\CM\to B_2$, and trivially implies the conclusion of the GRNC (i.e., $u\mathcal{M}$ is compact Hausdorff in the $\tau$ topology).
\end{remark}

\begin{proof}
It follows from Remark \ref{remark: w2=theta}.
\end{proof}

A slightly different form of GNC is:

\begin{conj}\label{conjecture: main conjecture}
Let $(\mathcal{A},a_0,G)$ be an ambit on which there exists a left continuous semigroup operation $*$ extending the action of $G$. Let $u^2=u \in \mathcal{M}\unlhd_m \mathcal{A}$.
Then, under suitable assumptions on the ambit $\mathcal{A}$, the map $\widehat{\theta}\colon u\CM\to G^{\textrm{Bohr}}_{\mathcal{A}}$ (from Corollary \ref{corollary: Newelski's map in general}) is a topological group isomorphism.
\end{conj}

Of course, the question ``What are the suitable assumptions?'' is the main point. 
Our first, motivated by \cite{ArtemPierre} and \cite{Krupinski2019}, guess was:
if $(X,G)$ is tame and amenable, then $w_2|_{u\CM}\colon u\CM\to B_2$ is an isomorphism of topological groups. However, it might be easily refuted --- by adding a single point to a flow
we obtain an amenable flow, then we do it for a tame flow
for which $w_2|_{u\CM}\colon u\CM\to B_2$ is not injective (e.g., the one from \cite{GismatullinPenazziPillay2015Some}). 
On the other hand, we prove Conjecture \ref{conjecture: main conjecture} for tame and amenable $(\mathcal{A},a_0,G)$:

\begin{theorem}\label{theorem: GNC}
Let $(\mathcal{A},a_0,G)$ be a tame amenable ambit on which there exists a left continuous semigroup operation $*$ extending the action of $G$. Then $\widehat{\theta}\colon u\CM\to G^{\textrm{Bohr}}_{\mathcal{A}}$ is a topological group isomorphism.
\end{theorem}

By the fact that tameness of a flow $(X,G)$ transfers to $(E(X,G),G)$ 
(see \cite[Lemma 5.4]{GlasnerMegrelishvili2012} and the sentence immediately following its proof) and by Fact \ref{Fact: 5.6} and Remark \ref{remark: w2=theta}, Theorem \ref{theorem: GNC} is equivalent to the following theorem which says that Conjecture \ref{conj: GNC} holds for tame hereditarily amenable flows and which is the first main result of this paper.

\begin{theorem}\label{thm:GNC 0}
    If $(X,G)$ is tame and hereditarily amenable, then $w_2|_{u\CM}\colon u\CM\to B_2$ is
    an isomorphism of topological groups. 
\end{theorem}

A topological dynamics proof will be given in Subsection \ref{subsection: dynamical proof of GNC}; an analytic proof will appear in Subsection \ref{subsection: analytic proof of GNC}.

In Corollary \ref{corollary: GNC implies NC for definably amenable NIP groups}, we will see that Theorem \ref{theorem: GNC} implies one of the main results of \cite{ArtemPierre} saying  that the original Newelski's Conjecture holds for definably amenable groups definable in NIP theories (i.e., Fact \ref{fact: Theorem 1.5 from CS in intro}). In Subsection \ref{subsection: local NC}, we will use Theorem \ref{theorem: GNC}  to get a local variant of Newelski's Conjecture.

\subsection{GNC under special assumptions}
We say that GNC holds for a flow $(X,G)$
    if the conclusion of Conjecture \ref{conj: GNC} holds, equivalently if 
    $w_2|_{u\CM} \colon u\CM\to B_2$ is injective.

Note that GNC trivially holds 
for flows with $u\CM=\{u\}$ (e.g., proximal flows).
However, it is not enough to assume only finiteness of $u\CM$, as we have a counterexample for GNC where $|u\CM|=2$ and $|B_1|=|B_2|=1$ (see \cite{GismatullinPenazziPillay2015Some} and e.g. Lemma 3.10 in \cite{GPPb} to note that $B_1=B_2$).
Working in the model-theoretic context of a group $G$ definable in a $p$-adically closed field $M$ acting on the type space $S_G(M)$, Theorem 2 in \cite{YaoZhe26} yields the equivalence between the conclusion of Newelski's conjecture and amenability of the flow $(S_G(M),G)$. In Subsection \ref{subsection: Formulation of GNC}, we gave a simple reason why the right-to-left implication does not hold in general for tame flows (and Theorem   \ref{thm:GNC 0} shows that it holds with hereditary amenability in place of amenability). Below we give an example showing that in general the left-to-right implication does not hold either. 

\begin{example}
    Consider an $\aleph_0$-saturated and strongly $\aleph_0$-homogeneous model $M$ of the theory $T$ of the circular order (cf. \cite{Kim1996Simple}). 
    Let $\bar{m}$ be an enumeration of the model $M$,
    and let $S_{\bar{m}}(M)$ be the space of complete types over $M$ extending $\tp(\bar{m}/\emptyset)$. Since $T$ has NIP, the flow $(X,G):=(S_{\bar{m}}(M),\aut(M))$ is tame by \cite[Corollary 5.8]{KrRz}. The theory $T$ is a classical example for which $\emptyset$ is not an extension base for forking. Then Corollaries 2.7 and 2.24 from \cite{HruKruPi} show that $(X,G)$ is not amenable.
On the other hand, by \cite[Proposition A.1]{KrRz}, the Ellis groups are trivial (and so the compactifications $B_1$ and $B_2$ are trivial as well) and thus GNC holds. Note also that Claim 1 in the proof of \cite[Proposition A.1]{KrRz} tells us that $(X,G)$ is proximal.
\end{example}

\begin{remark}
   GNC holds for WAP flows.
\end{remark}

\begin{proof}
If $(X,G)$ is WAP, then there exists a unique minimal left ideal $\CM$ in $E:=E(X,G)$, which also coincides with its Ellis group $u\CM$, and $(E,G)$ is an amenable flow (cf. \cite{WAP} or Theorem 2.5 in \cite{Conant2025Stable} for short summary). Let $\rho \colon E \to \CM$ be given by $\rho(\eta):=\eta u$ and $\iota \colon G \to E$ by $\iota(g): = g \id_X$. Then $\rho$ is clearly an epimorphism of ambits and $\rho \circ \iota  \colon G \to \mathcal{M}=u\mathcal{M}$ is a group homomorphism, as $(\rho \circ \iota) (g_1g_2)=g_1g_2u=g_1u g_2u = (\rho \circ \iota) (g_1) (\rho \circ \iota) (g_2)$ (we used here that since $\mathcal{M}=u\mathcal{M}$, for every $\eta \in \mathcal{M}$ we have $\eta=u\eta$). 
    We use the universal property of $(B_2,\mathfrak{b}_2,w_2)$ to find the dashed arrow in the following commutative diagram
        $$\xymatrix{ & G \ar[d]^-{\iota}  \ar@/_1pc/[ddl]_-{\fb_2} \ar@/^1pc/[ddr]^-{\rho\circ\iota}& \\
        & E \ar[dl]^-{w_2} \ar[dr]_-{\rho}& \\
        B_2 \ar@{-->}[rr]_-{f} & & \CM}$$
        Then for $\eta\in E$, 
        $$(f w_2)(u\eta u)=\rho(u\eta u)=u\eta u^2=u\eta u,$$
        hence $(fw_2)|_{u\CM}=\id_{u\CM}$ and so $w_2|_{u\CM}$ is injective. Therefore, GNC holds for WAP flows.
\end{proof}

Recall that a factor $\pi \colon (X,G) \to (Z,G)$ is said to be {\em almost 1-1}, if $\{z \in Z: |\pi^{-1}(z)|=1\}$ is a dense subset of $Z$. 

\begin{lemma}\label{lemma: almost to GNC}
    If $(X,G)$ is minimal and the MEF factor $\pi \colon (X,G)\to(Z,G)$ is almost 1-1\footnote{In this case $(Z,G)$ is the maximal Kronecker factor, i.e. a MEF which is a minimal flow.
    Alternatively, in the lemma we may assume that $\pi$ is almost 1-1 and $(Z,G)$ is minimal, as this implies that $(X,G)$ is minimal by Remark 5.6 from \cite{CGK}.},
    then GNC holds and $u\CM\cong K_2 \cong B_2 \cong K_1$. 
\end{lemma}

\begin{proof}
    By Lemma 5.7 from \cite{CGK}, $\pi_*|_{u\CM}\colon u\CM\to\pi_*(u)\pi_*[\CM] = K_1$ is an isomorphism of semitopological groups considered with the $\tau$-topologies (and because $K_1$ is a compact group, we have that also $u\CM$ is a compact group in the $\tau$-topology).
    By universality of $(B_2,\mathfrak{b}_2,w_2)$ for the triple $(K_1, \pi_*\circ(g\mapsto g\id_X), \pi_*)$, there exists a unique continuous group epimorphism $h\colon B_2\to K_1$ such that
$$\xymatrix{G \ar[rr]^{\fb_2} \ar[dr]_-{g\mapsto g \id_X}& & B_2 \ar[dr]^{h}&  \\ & E=E(X,G) \ar[ur]_{w_2} \ar[rr]_{\pi_*}& &  E(Z,G)\cong K_1}$$
    Since $(h\circ w_2)|_{u\CM}=\pi_*|_{u\CM}$,
    we have that $w_2|_{u\CM}$ is injective, so GNC holds.
    
    Then, $h$ is injective on the image of $w_2|_{u\CM}$, which is equal to the entire group $B_2$, so $h$ is a continuous isomorphism of groups.
    Both are compact groups, hence $h$ is an isomorphism of topological groups.

Since $w_2|_{u\CM}$ is an isomorphism, so is $\mathfrak{f}_1 \colon K_2 \to B_2$ in the first diagram in the proof of Proposition \ref{prop: second iso}. As $\mathfrak{f}_1$ is continuous, we get that $K_2 \cong B_2$ as topological groups.
\end{proof}

The notation in the next result is taken from Remarks \ref{remark: B2 to B1} and \ref{rem: KB1}.

\begin{remark}\label{cor: GNCtoISO}
Assume that  GNC holds, i.e. $w_2|_{u\CM}$ is an isomorphism. Then $u\CM\cong K_2\cong B_2$.
    If additionally the unique map $\ff\colon B_2\to B_1$ from Remark \ref{remark: B2 to B1} is an isomorphism,
    then $u\mathcal{M} \cong K_2 \cong B_2 \cong K_1 \cong B_1$.
\end{remark}

\begin{proof}
The first conclusion follows as in the last paragraph of the proof of Lemma \ref{lemma: almost to GNC}.

    Now, assume that $w_2|_{u\CM}$ is an isomorphism of groups (i.e. GNC holds), and $\ff$ is an isomorphism of groups as well.
By an easy diagram chase (recall that the map $\mathfrak{t} \colon K_1\to B_1$ was given in Remark \ref{rem: KB1}) on the following commutative diagram
$$\xymatrix{B_2 \ar[r]^-{\cong}_-{\ff} \ar@{-->}@/^2pc/[rr]^-{\exists!}  & B_1 & K_1 \ar@{->>}[l]^-{\mathfrak{t}}\\ 
& E \ar@{->>}[ul]^-{w_2} \ar@{->>}[u]_-{w_1\pi_0} \ar@{->>}@/_1pc/[ur]_-{\pi_*} &}$$
we conclude that $\mathfrak{t}$ is an isomorphism as well.
\end{proof}

One could ask whether $B_1 \cong B_2$ under the assumptions of Lemma \ref{lemma: almost to GNC}. Example \ref{example: MEF1notB1} gives us a negative answer.
    Indeed, the flow in Example \ref{example: MEF1notB1} satisfies the assumptions of Lemma \ref{lemma: almost to GNC}, hence $B_2\cong K_1\cong \mathrm{SO}(3)$, but $B_1=\{e_{B_1}\}$ in that case.

We now give a general source for examples where the compactifications $\mathfrak{b}_1$ and $\mathfrak{b}_2$ are not isomorphic despite of the very strong properties of the underlying flow. This will reveal an intriguing difference between the abstract and the model theory situations discussed in Section \ref{section: model theory}. 

\begin{example}\label{example: B_1 not equal  to B_2}
If $(X,G)$ is an equicontinuous flow, then the Ellis semigroup $E(X)$ is a compact group and the map $\iota \colon G \to E(X,G)$, $g \mapsto g \id_X$, is precisely the Bohr compactification of $G$ below $E(X)$, so we can assume that $\mathfrak{b}_2 = \iota$. On the other hand, if $(X,G)$ is additionally minimal, then, after distinguishing any point $x_0 \in X$, we have that $(X,G) \cong (E(X,G)/H,G)$ for the compact subgroup 
$$H:=\ker (\ev_{x_0}):=\{ \eta \in E(X,G): \eta(x_0)=x_0\}$$
of $E(X)$. Thus, whenever $H$ is a non-trivial subgroup of $E(X)$, then the Bohr compactification below $E(X)$ (i.e., $\mathfrak{b}_2$) is not isomorphic to $\mathfrak{b}_1$. Moreover, each minimal equicontinuous flow is WAP (hence tame) and hereditarily amenable.

Example  \ref{example: MEF1notB1} is an instance in which even the target groups of the above two compactifications are not isomorphic. For a simpler ``finite'' example consider $G:=S_3$, $H :=\langle(1,2)\rangle \leqslant G$, and the action of $G$ on $G/H$ by left translations. Then $(G/H,G)$ is equicontinuous, $E(G/H,G)=G$, but $|G/H|=3<6=|G|$, so the Bohr compactification below $G/H$ is of smaller cardinality (in fact, it is trivial) than the one below $E(G/H,G)$.
\end{example}

\begin{proof}
For the fact that equicontinuity of $(X,G)$ implies that $E(X,G)$ is a compact group, and equicontinuity together with minimality imply that $(X,G) \cong (E(X,G)/H,G)$  for $H=\ker (\ev_{x_0})$ see \cite[Theorem I.3.3]{Glasner1976}. This yields amenability of $(X,G)$ (and so hereditary amenability by minimality of $(X,G)$) with a witnessing $G$-invariant measure being the pushforward of the Haar measure on $E(X,G)$ by the quotient map $E(X,G) \to E(X,G)/H$ (in fact, this is a unique $G$-invariant measure on $(X,G)$ by \cite[Proposition II.10]{WAP}).  By \cite[Theorem I.3.3]{Glasner1976}, if $(X,G)$ is equicontinuous, then each member of $E(X,G)$ is a continuous function, so $(X,G)$ is WAP (hence tame) by \cite[Proposition II.2]{WAP}.

The fact that $\iota \colon G \to E(X,G)$ is the Bohr compactification of $G$ below $E(X,G)$ follows from the fact that $\iota$ is a group homomorphism with dense image and $w_2$ is a continuous group epimorphism by Remark \ref{rem: easy}.

To see that if $H$ is nontrivial then $\mathfrak{b}_1$ is not isomorphic to $\mathfrak{b}_2$, note that the map $\mathfrak{f}$ in Remark \ref{remark: B2 to B1} is a unique continuous group homomorphism such that $f \circ \mathfrak{b}_2=\mathfrak{b}_1$ and it is not an isomorphism because $f = f \circ w_2 = w_1 \circ \ev_{x_0}$ and $\ker(\ev_{x_0})=H$ is nontrivial and contained in $\ker(f)$.
\end{proof}

\begin{cor}\label{cor: GNC1}
    If $(X,G)$ is tame, amenable, minimal and metric, then GNC holds.
\end{cor}

\begin{proof}
    By Corollary 5.4.2 from \cite{Glasner2}, we have that $(X,G)$ is almost automorphic, i.e. the $X$-MEF $\pi$ is almost 1-1. Then we apply Lemma \ref{lemma: almost to GNC}.
\end{proof}

In Theorem \ref{thm:GNC 0} (which will be proven in the next subsection) we drop the minimality and metrizability assumption. This requires a different method of a proof.

\subsection{Dynamical proof of GNC}\label{subsection: dynamical proof of GNC}
In this subsection, we give a dynamical proof of Theorem \ref{theorem: GNC}, which is equivalent to Theorem \ref{thm:GNC 0} as explained in Subsection \ref{subsection: Formulation of GNC}.

From now on, let $(\mathcal{A},a_0,G)$ be an ambit with a left continuous semigroup operation extending the action of $G$, $\mathcal{M}$ a minimal left ideal, and $u \in \mathcal{M}$ an idempotent.

A very important fact for the proof of Theorem \ref{theorem: GNC} is Lemma 5.7 in \cite{HoffKru1}, i.e. 
if $\mathcal{A}$ is tame and amenable, then  $\overline{u\mathcal{M}}=\mathcal{M}$ (see Fact \ref{fact: amenability closure}(2)).

\begin{lemma}\label{lemma: cl(uM)=M implies strong amenability}
Assume that $\overline{u\mathcal{M}}=\mathcal{M}$. 
Then $\mathcal{A}$ is strongly amenable (i.e., has no non-trivial minimal proximal factors).
\end{lemma}

\begin{proof}
Let $Y$ be a minimal proximal factor of $\mathcal{A}$. Pick $y_0 \in Y$ such that $uy_0=y_0$. By minimality of $Y$, we have $\mathcal{M}y_0=Y$. Since $Y$ is proximal, we have that $u\mathcal{M} y_0=\{y_0\}$ (by Prop. I.3.2(3) or I.4.1(2) in \cite{Glasner1976}). As by assumption $\overline{u\mathcal{M}}=\mathcal{M}$ 
and the evaluation map  $\ev_{y_0} \colon \mathcal{M} \to Y$ is continuous, we conclude that $Y=\{y_0\}$.
\end{proof}

By virtue of Proposition \ref{proposition: Glasner below A}, Fact \ref{fact: amenability closure}(2), and Lemma \ref{lemma: cl(uM)=M implies strong amenability}, in order to prove Theorem \ref{theorem: GNC}, it is enough to show that $\gal_{\mathcal{A}} =\{u\}$.
The final highly non-trivial reduction is the Generalized Revised Newelski's Conjecture for tame flows, i.e. Fact \ref{fact:RNC}:
the Ellis group $u\mathcal{M}$ of a tame flow is always Hausdorff, equivalently, $H(u\mathcal{M})=\{u\}$.
Therefore, in order to prove Theorem \ref{theorem: GNC}, it is enough to show that $\gal_{\mathcal{A}} =H(u\mathcal{M})$.

The desired equality $\gal_{\mathcal{A}} =H(u\mathcal{M})$ follows by an obvious adaptation (and simplification due to the fact that the universal minimal proximal factor of $\mathcal{A}$ is trivial here) of either the proofs of Propositions IX.4.1 and IX.4.2 from \cite{Glasner1976} or Theorem 2.5 from \cite{Krupinski2017}, because under the assumption $\overline{u\mathcal{M}}=\mathcal{M}$ the relations on $\mathcal{M}$ in both these sources are indeed closed, invariant, equivalence relations. Just for the reader's convenience we briefly repeat the argument from \cite{Glasner1976}, which completes the proof of Theorems \ref{theorem: GNC}. As in Subsection \ref{subs:category.below}, $\mathcal{C}$ denotes the category of minimal flows which are factors of $\mathcal{A}$.

\begin{lemma}\label{lemma: Gal= H(uM)}
Assume that $\overline{u\mathcal{M}}=\mathcal{M}$. 
Then the universal minimal proximal factor of $\mathcal{A}$ is trivial.

Let $\sim$ be a relation on $\mathcal{M}$ defined by
$$p \sim q \iff (up)^{-1}q \in H(u\mathcal{M}).$$
Then $\sim$ is a closed, $G$-invariant, equivalence relation on $\mathcal{M}$. Moreover, $Z:=\mathcal{M}/\!\!\sim$ is the universal group extension in $\mathcal{C}$ of the trivial flow and 
$$\gal(Z, u/\!\! \sim)=\{\eta\in u\CM\;:\;\eta(u/\!\!\sim)=u/\!\!\sim \}=H(u\mathcal{M})$$ 
so $u\mathcal{M}/H(u\mathcal{M})$ is the generalized Bohr compactification of $G$ in $\mathcal{C}$.
\end{lemma}

Before the proof let us comment that although (in contrast to the case of $\mathcal{A}=\beta G$) in general we do not have an action of $\mathcal{A}$ on the flow $2^\mathcal{A}$ (because the subambits of this flow need not be in the category $\mathcal{C}$), 
using the definition of the circle operation $\circ$ for the $\tau$-topology the proofs of 1.2-1.12 (except 1.12(2)) from \cite[Chapter IX]{Glasner1976} go through (with some slight modifications) in our context (see \cite[Appendix A]{rzepecki2018} for the proofs). Observe that Lemma 3.4 from \cite[Chapter IX]{Glasner1976} also goes through in our context.

Note also that the additional condition $p \circ u\mathcal{M} = q \circ u\mathcal{M}$ which was needed in the definition of $\sim$ in \cite[Chapter IX.4]{Glasner1976} is automatically satisfied under the assumption $\overline{u\mathcal{M}}=u\mathcal{M}$. Indeed, then $p \circ u\mathcal{M}= p\circ \overline{u\mathcal{M}}=p \circ \mathcal{M}=\mathcal{M}$ for every $p \in \mathcal{M}$.

\begin{proof}[Proof of Lemma \ref{lemma: Gal= H(uM)}]
Triviality of the universal minimal proximal factor of $\mathcal{A}$ is Lemma \ref{lemma: cl(uM)=M implies strong amenability}.

It is clear that $p \sim q$ iff $upH(u\mathcal{M})=uqH(u\mathcal{M})$, so $\sim$ is an equivalence relation on $\mathcal{M}$. 
To see invariance, consider any $p \sim q$ and $g \in G$. 
Since $q \in \mathcal{M} = p \circ u\mathcal{M}$, by \cite[Lemma 3.4]{Glasner1976}, $q \in p \circ N$ for any $\tau$-neighborhood $N$ of $(up)^{-1}q$ in $u\mathcal{M}$. 
Hence, $gq \in (gp) \circ N$, and so $(u(gp))^{-1}(gq) \in u \circ N$ by \cite[Lemma 3.4]{Glasner1976}. Thus, $(u(gp))^{-1}(gq) \in \cl_{\tau}(N)$. Since by assumption $(up)^{-1}q \in H(u\mathcal{M})$, the above conclusion holds for all $\tau$-neighborhoods $N$ of $H(u\mathcal{M})$. As $u\mathcal{M}/H(u\mathcal{M})$ is Hausdorff, the intersection of all those $\cl_\tau(N)$ equals $H(u\mathcal{M})$, so we get $(u(gp))^{-1}(gq) \in H(u\mathcal{M})$, i.e. $gp \sim gq$. 

Finally, we show that $\sim$ is closed. So take any net $(p_i,q_i) \in \mathcal{M} \times \mathcal{M}$ with $p_i \sim q_i$, which converges to $(p,q)$. Let $N$ be any $\tau$-neighborhood of $H(u\mathcal{M})$ so also of
each $(up_i)^{-1}q_i$. Since $q_i \in \mathcal{M} = p_i \circ u\mathcal{M}$, by \cite[Lemma 3.4]{Glasner1976}, $q_i \in p_i \circ N$. Using limits, we conclude that $q \in p \circ N$ and so $(up)^{-1}q \in \cl_{\tau}(N)$ by \cite[Lemma 3.4]{Glasner1976}. Thus, as in the previous paragraph, we get $(up)^{-1}q \in H(u\mathcal{M})$, i.e. $p \sim q$.

In this way, $Z$ becomes a minimal $G$-flow which is a factor of $(\mathcal{M},G)$ so also of $\mathcal{A}$. It is clear that $u (u/\!\!\sim)=u/\!\!\sim$. By definition, 
\begin{IEEEeqnarray*}{rCl}
\gal(Z,u/\!\!\sim) &=& \{ \eta \in u\mathcal{M}: \eta (u/\!\! \sim)=u/\!\! \sim\}= \{ \eta \in u\mathcal{M}: \eta \sim u\} \\ &=& [u]_{\sim} \cap u\mathcal{M}=H(u\mathcal{M}).
\end{IEEEeqnarray*}

Applying the obvious counterpart of Theorem IX.2.1 from \cite{Glasner1976} (which can be shown either by the same proof or by an application of that theorem in the original context of $\beta G$ to the quotient objects obtained from the epimorphism $\beta G \to \mathcal{A}$), in order to show that $(Z,G,u\mathcal{M}/H(u\mathcal{M}))$ is the universal group extension of the trivial flow (and, in consequence, $u\mathcal{M}/H(u\mathcal{M})$ is the generalized Bohr compactification of $G$ in $\mathcal{C}$), it remains to prove that $Z$ is distal. By \cite[Proposition I.3.2(3)]{Glasner1976}, it is enough to show that $uz=z$ for every $z=p/\!\!\sim \;\in Z$; equivalently, $p \sim up$. The last thing holds as $(up)^{-1}up=u \in H(uM)$. 
\end{proof}

Although Theorem  \ref{theorem: GNC} follows by the sequence of reductions described in this subsection, we summarize it below.

\begin{proof}[Proof of Theorem \ref{theorem: GNC}]
By assumption, $\mathcal{A}$ is tame and amenable.

By Fact \ref{fact: amenability closure}(2) and Lemma \ref{lemma: Gal= H(uM)}, working with the notation from Lemma \ref{lemma: Gal= H(uM)} and from Subsection  \ref{subs:category.below}, we have $\CM/\!\!\sim \; = \pi_{\mathcal{A}}^{\#}$, and we may choose $b_0:=u/\!\!\sim$ in the definition of $\gal_{\mathcal{A}}$ (see Subsection \ref{subs:category.below}); then $\gal_{\mathcal{A}}=\gal(Z,u/\!\!\sim)$. So, by Lemma \ref{lemma: Gal= H(uM)}, $\gal_{\mathcal{A}} = H(u\mathcal{M})$. This together with Fact \ref{fact:RNC} implies that $\gal_{\mathcal{A}}=\{u\}$.

By Fact \ref{fact: amenability closure}(2) and Lemma \ref{lemma: cl(uM)=M implies strong amenability}, $\mathcal{A}$ is strongly amenable, so the conclusion follows from Proposition \ref{proposition: Glasner below A}.
\end{proof}

In fact, one can prove Theorem \ref{theorem: GNC} a bit faster, i.e. without most of the general considerations in Subsection \ref{subs:category.below} involved above, which we now describe. But first of all, if we want to avoid using Proposition \ref{proposition: univ. comp. flow} and Corollary \ref{corollary: Newelski's map in general}, we have to explain what the map $\widehat{\theta}$ in the statement of Theorem \ref{theorem: GNC} is. Bearing in mind the natural identification of $\mathcal{A}$ with $E(\mathcal{A},G)$ the map $w_2$ from Subsection \ref{subsection: Bohr comp} considered for the ambit $\mathcal{A}$ becomes a semigroup epimorphism from $\mathcal{A}$ to $G^{\textrm{Bohr}}_{\mathcal{A}}$, and we take $\widehat{\theta}:=w_2|_{u\mathcal{M}} \colon u\mathcal{M} \to G^{\textrm{Bohr}}_{\mathcal{A}}$ which is a continuous group epimorphism by Remark \ref{rem: easy}. As in the short proof of Remark \ref{remark: w2=theta}, we get that $\widehat{\theta}(\eta)=\eta e_B$.

\begin{proof}[Faster proof of Theorem \ref{theorem: GNC}]
Let $\mathcal{A}$ be a $G$-ambit with a left continuous semigroup operation extending the action of $G$.
Define $\widehat{\theta} \colon u\mathcal{M} \to  G^{\textrm{Bohr}}_{\mathcal{A}}$ as in the above paragraph. By the second paragraph of the proof of Corollary \ref{corollary: Newelski's map in general}, we know that $uh=h$ for all $h \in G^{\textrm{Bohr}}_{\mathcal{A}}$. Thus, for every $g \in G$, $\widehat{\theta}(ugu)=ugu e_B =ge_B = \mathfrak{b}_{\mathcal{A}}(g)e_B=\mathfrak{b}_{\mathcal{A}}(g)$. 

From now on,  assume that $\mathcal{A}$ is additionally tame and amenable. By Fact \ref{fact: amenability closure}(2) and Lemma \ref{lemma: cl(uM)=M implies strong amenability}, $\mathcal{A}$ is strongly amenable. By Facts \ref{fact: amenability closure}(2) and \ref{fact:RNC} and Lemma \ref{lemma: Gal= H(uM)}, the universal compactification flow below $\mathcal{A}$ is of the form $(Z,G,u\mathcal{M})$, $b_0:=u/\!\! \sim$ satisfies $ub_0=b_0$ (actually any point in $Z$ has this property), and $\gal(Z,b_0)= \{u\}$. As in Fact \ref{fact: identification of Sigma and uI/Gal}, the proof of Lemma \ref{lemma: Gal= H(uM)} (more precisely, the application of the explicit action in the proof of Theorem IX.2.1 from \cite{Glasner1976}) shows that the witnessing right action of $u\mathcal{M}$ on $Z$ is given by $(pb_0) \cdot \eta :=(p\eta)b_0$ for every $p \in \mathcal{M}$. 

Having the above two paragraphs, we apply the proof of Proposition \ref{proposition: Glasner below A} to $u\mathcal{M}$ in place of $\gal_{\mathcal{A}}$ (now $\gal_{\mathcal{A}} = u\mathcal{M}$) and to $\widehat{\theta}$ defined above in place of $\theta$ defined in Corollary \ref{corollary: Newelski's map in general} (in fact, these functions coincide, but we do not use it) and we get that $\widehat{\theta}$ is a topological group isomorphism (continuity of $\widehat{\theta}$ follows from the proof of Proposition \ref{proposition: Glasner below A}, namely $\widehat{\theta} \circ \kappa^{-1} \circ h = \id_{G^{\textrm{Bohr}}_{\mathcal{A}}}$ where $\kappa^{-1} \circ h$ is a continuous epimorphism between compact Hausdorff groups), which finishes the proof of Theorem \ref{theorem: GNC}.
\end{proof}

\section{More on the structure of amenable and tame flows}
\label{sec: more on amenable tame structure}

\subsection{Analytic proof of GNC}\label{subsection: analytic proof of GNC}
This part is dedicated to the proof of GNC via methods from operator theory.
The main engine of the analytic proof of GNC is the following lemma which is interesting by its own rights and useful for the structural results in Subsections \ref{subsection: structure of amenable tame flows} and \ref{subsec: X-MEF 2}.

Here, $E=E(X,G)$, $\CM\unlhd_m E$, $u^2=u\in\CM$, and
$(B_2,\fb_2,w_2)$ is 
the Bohr compactification of $G$ below $(E,\id_X,G)$.

\begin{lemma}\label{lemma: L2 Bohr}
Let $(X,G)$ be tame, $\CM\unlhd_mE(X,G)$, and let
$\mu\in\operatorname{Prob}_G(\CM)$.  Then:
\begin{enumerate}
\item
$\mu$ is the unique $G$-invariant probability measure on $\CM$;

\item 
$(w_2|_\CM)_*\mu=\mathfrak h_{B_2}$;

\item
the pullback along $w_2|_\CM$ extends uniquely to a unitary
isomorphism
$$(w_2|_\CM)^*:L^2(B_2,\mathfrak h_{B_2}) \longrightarrow L^2(\CM,\mu).$$
\end{enumerate}
\end{lemma}

\begin{proof}
By \cite[Lemma~2.2(3)]{HoffKru1}, the Ellis flow $(E(X,G),G)$ is
tame; hence its minimal subflow $(\CM,G)$ is minimal and tame.
Fact~\ref{lemma : min tame unique ergodicity} makes $\mu$ unique, and
therefore ergodic.  Since $\CM$ is minimal, $\supp(\mu)=\CM$.

The measure
$(w_2|_\CM)_*\mu$ is invariant under the dense subgroup $\fb_2[G]$,
and therefore $(w_2|_\CM)_*\mu=\mathfrak h_{B_2}$ (e.g. by Fact \ref{lemma : min tame unique ergodicity}).
Thus pullback initially gives an isometry
\begin{equation}\label{eq:HN001}
 (w_2|_\CM)^*:L^2(B_2,\mathfrak h_{B_2})\longrightarrow L^2(\CM,\mu),
 \qquad (w_2|_\CM)^*(F)=F\circ(w_2|_\CM).
\end{equation}
It remains to prove that $(w_2|_\CM)^*$ is onto, 
i.e. that $(w_2|_{\CM})^*[C(B_2)]$ is dense.
For this task, we will find an auxiliary unitary representation $B_2\to U(L^2(\CM,\mu))$. Let us start with ``target'' subgroup of $U(L^2(\CM,\mu))$ and, working with the strong operator topology, set
$$K:=\overline{\{U_g:g\in G\}}
   \subseteq B(L^2(\CM,\mu)).$$
%$$K:=\overline{\{U_g\;:\;g\in G\}}\subseteq \mathscr L(L^2(\CM,\mu)).$$

\begin{clm}
    $K$ is a compact Hausdorff group.
\end{clm}

\begin{clmproof}
Let $\Gamma:=\{U_g:g\in G\}$.
If $F\in C(\CM)$ and $(g_n)_{n<\omega}$ is a sequence in $G$, then
tameness and \cite[Definition~3.1(1), Theorem~2.4]{GM23}, applied
successively to the real and imaginary parts of $F$, yield a
subsequence $(g_{n_k})_k$ for which $(U_{g_{n_k}}F)_k$ converges
pointwise on $\CM$, say, to a function $F_\infty:\CM\to\mathbb C$. 
Since $|U_{g_{n_k}}F|\leqslant\lVert F\rVert_\infty$ and
$|F_\infty|\leqslant\lVert F\rVert_\infty$ pointwise on $\CM$,
the dominated convergence theorem, applied to the sequence of functions
$\bigl(\lvert U_{g_{n_k}}F-F_\infty\rvert^2\bigr)_k$ dominated by
the integrable constant function $4\lVert F\rVert_\infty^2$, shows
that $(U_{g_{n_k}}F)_k$
converges in $L^2(\CM,\mu)$.  Thus $\Gamma F$ is relatively norm
compact, i.e.
for every $\epsilon>0$ there are $U_{g_1}F,\ldots,U_{g_n}F\in \Gamma F$
such that $\Gamma F$ is covered by the open balls $B(U_{g_i}F,\epsilon)$.

Let $\xi\in L^2(\CM,\mu)$ and let $\epsilon>0$.
Density of $C(\CM)$ in $L^2(\CM,\mu)$ gives us $F\in C(\CM)$ such that
$\lVert \xi- F\rVert_{L^2(\CM,\mu)}\leqslant\frac{\epsilon}{3}$.
For $F$, there are 
$U_{g_1}F,\ldots,U_{g_n}F\in \Gamma F$ as in the first part of the proof of the claim for $\frac{\epsilon}{3}$.
Using the identities
$$\lVert U_g\xi-U_gF\rVert_{L^2(\CM,\mu)}=\lVert\xi-F\rVert_{L^2(\CM,\mu)}$$
holding for every $g\in G$,
we obtain that for every $g\in G$ there is some $i\leqslant n$ such that
\begin{IEEEeqnarray*}{rCl}
    \lVert U_g\xi - U_{g_i}\xi\rVert_{L^2(\CM,\mu)} &\leqslant &
    \lVert U_g\xi - U_gF\rVert_{L^2(\CM,\mu)} 
    + \lVert U_g F - U_{g_i}F \rVert_{L^2(\CM,\mu)}  \\ &&
    + \Vert U_{g_i}F - U_{g_i}\xi\rVert_{L^2(\CM,\mu)} \\
    &\leqslant & \frac{\epsilon}{3} + \frac{\epsilon}{3} + \frac{\epsilon}{3}
    =\epsilon.
\end{IEEEeqnarray*}
Hence
$$C_\xi:=\overline{\{U_g\xi:g\in G\}}^{\,\lVert\cdot\rVert}$$
is compact for every $\xi\in L^2(\CM,\mu)$.

Regard $\Gamma$ as a discrete space and let $\beta\Gamma$ be its
Stone--\v Cech compactification, identified with the space of
ultrafilters on $\Gamma$. Let $p\in\beta\Gamma$. Compactness of the
spaces $C_\xi$ gives limits
$$T_p(\xi):=\lim_{g\to p}U_g(\xi),
\qquad
S_p(\xi):=\lim_{g\to p}U_g^{-1}(\xi),$$
where $\xi\in L^2(\CM,\mu)$. Taking ultrafilter limits in the
linearity and norm identities shows that $T_p$ and $S_p$ are linear
isometries. Clearly,
$\lim_{g\to p}(U_g,U_g^{-1})=(T_p,S_p)$ in the product topology on $B(L^2(\CM,\mu))\times B(L^2(\CM,\mu))$.
Then, joint continuity of multiplication on norm-bounded sets gives $T_pS_p=S_pT_p=I$.
Thus $T_p\in U(L^2(\CM,\mu))$, $S_p=T_p^{-1}$, and $p$ converges to
$T_p$ (in the strong operator topology).

The map
$$\beta\Gamma\longrightarrow U(L^2(\CM,\mu)),
\qquad p\longmapsto T_p,$$
is continuous, because for every $\xi\in L^2(\CM,\mu)$ the coordinate map
$p\longmapsto T_p\xi$ is continuous. Its image is therefore compact. It contains $\Gamma$,
and every $T_p$ is a limit of elements of $\Gamma$, so its image
is exactly $K$. Thus $K$ is compact.

Since $U(L^2(\CM,\mu))$ is a Hausdorff topological group, the
closure of its subgroup $\Gamma$ is a subgroup. Hence $K$ is a
compact Hausdorff group.
\end{clmproof}

\begin{clm}
    There is a continuous semigroup epimorphism $w:E(X,G)\to K$
    such that
    $$\xymatrix{ G \ar[rr]^-{g\mapsto U_g} \ar[dr]_-{g\mapsto g\circ\id_X}& & K \\
    & E(X,G) \ar[ur]_-{w} & }$$
    commutes.
\end{clm}

\begin{clmproof}
The map
$G\longrightarrow E(X,G)$, given by $g\longmapsto g\circ\id_X$,
and the map $G\longrightarrow K$, given by $g\longmapsto U_g$,
extend uniquely to continuous surjections
$$\pi_E:\beta G\longrightarrow E(X,G)
\qquad\text{and}\qquad\pi_K:\beta G\longrightarrow K.$$
Define $\lambda:E(X,G)\longrightarrow E(\CM,G)$ by
$\lambda(\eta)(\xi):=\eta\xi$.
For $\eta\in E(X,G)$, let
$$\widehat{\eta}:=\operatorname{res}^{-1}(\lambda(\eta))
\in E(\operatorname{Prob}(\CM),G),$$
where $\operatorname{res}$ is the isomorphism from
Fact~\ref{fact:tame-injective}.

We claim that $\pi_K$ is constant on every fibre of $\pi_E$.
Let $\mathfrak p,\mathfrak q\in\beta G$ satisfy $\pi_E(\mathfrak p)=\pi_E(\mathfrak q)=\eta$.
Fix a nonzero $F\in C(\CM)$ with $F\geqslant0$, and put
$$c_F:=\int_\CM F\,d\mu,\qquad\nu_F:=\frac{F}{c_F}\mu.$$
Since $\supp(\mu)=\CM$, we have $c_F>0$, and therefore
$\nu_F\in\operatorname{Prob}(\CM)$.

For every $g\in G$, invariance of $\mu$ gives
$$
\begin{aligned}
\langle U_gF,H\rangle_{L^2(\CM,\mu)}
&=\int_\CM F(g^{-1}\xi)\overline{H(\xi)}\,d\mu(\xi)=\int_\CM F(\xi)\overline{H(g\xi)}\,d\mu(\xi)\\
&=c_F\int_\CM\overline{H(g\xi)}\,d\nu_F(\xi)=c_F\int_\CM\overline H\,d(g_*\nu_F)
\end{aligned}
$$
for every $H\in C(\CM)$.

Since $\pi_E(\mathfrak p)=\eta$, continuity of $\lambda$ and of
$\operatorname{res}^{-1}$ implies that
$g_*\nu_F\longrightarrow\widehat{\eta}(\nu_F)$ 
as $g\to\mathfrak p$.
Consequently,
$$\langle\pi_K(\mathfrak p)F,H\rangle_{L^2(\CM,\mu)}=\lim_{g\to\mathfrak p}\langle U_gF,H\rangle_{L^2(\CM,\mu)}
=c_F\int_\CM\overline H\,d\widehat{\eta}(\nu_F).$$
The same argument applied to $\mathfrak q$ gives
$\langle\pi_K(\mathfrak q)F,H\rangle_{L^2(\CM,\mu)}=c_F\int_\CM\overline H\,d\widehat{\eta}(\nu_F)$.
Hence
$$\langle\pi_K(\mathfrak p)F,H\rangle_{L^2(\CM,\mu)}=\langle\pi_K(\mathfrak q)F,H\rangle_{L^2(\CM,\mu)}$$
for every nonnegative $F\in C(\CM)$ and every $H\in C(\CM)$.
Note that the same equality holds for all $F,H\in C(\CM)$.
Density of $C(\CM)$ in
$L^2(\CM,\mu)$ now yields
$$\pi_K(\mathfrak p)=\pi_K(\mathfrak q).$$

Thus $\pi_K$ is constant on the fibres of $\pi_E$.  Since $\pi_E$ is
a continuous surjection from the compact space $\beta G$ onto $E(X,G)$, it is a quotient map.  Therefore there is a unique continuous map
$$w:E(X,G)\longrightarrow K$$
such that $\pi_K=w\circ\pi_E$.
The surjectivity of $\pi_K$ implies that $w$ is surjective.  Moreover,
for every $g\in G$, $w(g\circ\id_X)=U_g$.

It remains to verify that $w$ is a semigroup homomorphism.  Let
$g\in G$ and $\xi\in E(X,G)$, and choose a net $(h_i)_i$ in $G$ converging to $\xi$.
Continuity of left multiplication by $g$ on $E(X,G)$ gives
$(gh_i)_i\to g\xi$.  Hence, by continuity of $w$,
$$w(g\xi)=\lim_i w(gh_i)=\lim_i U_{gh_i}=\lim_i U_gU_{h_i}=U_gw(\xi).$$
Indeed, $(U_{h_i})_i=(w(h_i))_i\to w(\xi)$ in the strong operator topology, and left multiplication
by $U_g$ is continuous (in the strong operator topology): for every
$F\in L^2(\CM,\mu)$,
$$\lVert U_gU_{h_i}F-U_gw(\xi)F\rVert_{L^2(\CM,\mu)}=\lVert U_{h_i}F-w(\xi)F\rVert_{L^2(\CM,\mu)}\longrightarrow0.$$
Now let $\eta,\xi\in E(X,G)$ and choose a net $(g_i)_i$ in $G$ converging to $\eta$.
Since right multiplication by $\xi$ is continuous in $E(X,G)$,
$$w(\eta\xi)=\lim_iw(g_i\xi)=\lim_iU_{g_i}w(\xi)=w(\eta)w(\xi).$$
\end{clmproof}

Combining the above two claims and the universality of the Bohr compactification $(B_2, \fb_2, w_2)$, we obtain a continuous epimorphism $\rho: B_2\to K$ such that the diagram
$$\xymatrix{ B_2 \ar[rr]^-{\rho} & & K \\
& G \ar[ul]^-{\fb_2} \ar[ur]_-{g\mapsto U_g} &}$$
commutes. Thus, $\rho:B_2\to U(L^2(\CM,\mu))$ is the auxiliary unitary representation we are looking for.

We now apply Peter--Weyl, Fact \ref{fact: peter-weyl for unitary}, to the unitary representation $\rho$.
Since $B_2$ is compact, that theorem says that $L^2(\CM,\mu)$ 
is the Hilbert direct sum of irreducible $\rho(B_2)$-invariant subspaces and
that each of these subspaces is finite-dimensional. In particular,
the finite-dimensional $\rho(B_2)$-invariant subspaces have dense
linear span in $L^2(\CM,\mu)$. 
It is therefore enough to prove that every nonzero finite-dimensional $\rho(B_2)$-invariant subspace $W$ of $L^2(\CM,\mu)$ is contained in $(w_2|_\CM)^*[C(B_2)]$.

Let $W\subseteq L^2(\CM,\mu)$ be an arbitrary nonzero
finite-dimensional $\rho(B_2)$-invariant subspace, and set
$n:=\dim W$.  Choose an orthonormal basis $f_1,\ldots,f_n$ of $W$
and write $\mathbf{f}:=(f_1,\ldots,f_n)$.
For every $b\in B_2$, the operator $\rho(b)$ preserves $W$.  Define
the matrix $\pi(b)=(\pi_{ij}(b))_{i,j=1}^n$ by
$$\pi_{ij}(b):=\langle \rho(b)f_j,f_i\rangle_{L^2(\CM,\mu)}.$$
Equivalently, $\rho(b)f_j=\sum_{i=1}^n f_i\,\pi_{ij}(b)$ for $j\leqslant n$.
Thus, with $\rho(b)$ acting coordinatewise on $\mathbf f$, these
identities can be written as
$$\rho(b)\mathbf f=\mathbf f\,\pi(b).$$
Since $\rho(b)|_W$ is unitary, $\pi(b)\in U(n)$.  Moreover, because
$\rho:B_2\to K$ is a continuous group epimorphism,
we have $\pi(bc)=\pi(b)\pi(c)$ for every $b,c\in B_2$,
and every matrix coefficient mapping
$$b\longmapsto\pi_{ij}(b)=\langle\rho(b)f_j,f_i\rangle_{L^2(\CM,\mu)}$$
is continuous.  Hence $\pi:B_2\to U(n)$ is a continuous unitary
representation.  Finally, since $\rho(\fb_2(g))=U_g$ for all $g\in G$, we obtain
$$U_g\mathbf f=\mathbf f\,\pi(\fb_2(g)).$$

For every $g\in G$ and $\mu$-almost every $\eta\in\CM$, we have
$$\mathbf f(g^{-1}\eta)\pi(w_2(g^{-1}\eta))=\mathbf f(\eta)\pi(\fb_2(g))
 \pi\bigl(\fb_2(g)^{-1}w_2(\eta)\bigr)=\mathbf f(\eta)\pi(w_2(\eta)).$$
Thus, every coordinate of
$\CM\ni \eta\mapsto\mathbf f(\eta)\pi(w_2(\eta))$ belongs to
$L^2(\CM,\mu)^G$.  Since $\mu$ is ergodic, 
$L^2(\CM,\mu)^G=\mathbb C\mathbf 1$.
Hence there is a constant row $c\in\mathbb C^n$ such that
$\mathbf f(\eta)\pi(w_2(\eta))=c$ for $\mu$-almost every $\eta\in\CM$.
Consequently,
$$\mathbf f(\eta)=c\,\pi(w_2(\eta))^{-1}$$
for $\mu$-almost every $\eta\in\CM$.
The map $B_2\ni b\mapsto c\,\pi(b)^{-1}\in \mathbb{C}^n$ is continuous.  
Hence each $f_j$ is the pullback under $w_2|_\CM$ of a member of
$C(B_2)$.  It follows that $W\subseteq(w_2|_\CM)^*[C(B_2)]$.
\end{proof}

The proof of Lemma~\ref{lemma: L2 Bohr} also yields the following corollary:

\begin{cor}\label{cor: tame finite-dimensional L2}
Let $(X,G)$ be a tame flow and let $\mu\in\operatorname{Prob}_G(X)$. Then
$$L^2(X,\mu)=\overline{\operatorname{span}}\{V:V\text{ is a finite-dimensional $G$-invariant subspace}\}.$$
\end{cor}

\begin{theorem}\label{thm: GNC for amenable minimal ideal}
Let $(X,G)$ be tame, let $\CM\unlhd_mE(X,G)$, assume that
$(\CM,G)$ is amenable, and let $u^2=u\in\CM$.  Then
$$w_2|_{u\CM}:u\CM\longrightarrow B_2$$
is an isomorphism of compact Hausdorff topological groups (with the $\tau$-topology on $u\CM$).
\end{theorem}

\begin{proof}
Let $\mu$ be a $G$-invariant probability on $\CM$.  By
Lemma~\ref{lemma: L2 Bohr}, it is unique and minimality gives
$\supp(\mu)=\CM$. Put $E:=E(X,G)$.

We need to show that $\ker(w_2|_{u\CM})=\{u\}$.
Take $\eta\in u\CM$ such that
$w_2(\eta)=1_{B_2}$.  
Right multiplication
$$R_\eta:\CM\longrightarrow E,
 \qquad R_\eta(\xi)=\xi\eta,$$
is continuous and so an automorphism of the flow $(\CM,G)$. 
Indeed, because $\eta\in\CM$ we have that $\xi\eta\in\CM$ for every $\xi\in\CM$.
Then, every $\xi\in\CM=Eu$ satisfies $\xi u=\xi$, 
and right multiplication by the inverse of
$\eta$ in $u\CM$ is the inverse of $R_\eta$.  Moreover,
$w_2(\eta)=1_{B_2}$ implies $w_2|_{\CM} \circ R_\eta=w_2|_\CM$.

The measure $(R_\eta)_*\mu$ is $G$-invariant, so unique
ergodicity gives $(R_\eta)_*\mu=\mu$.  
Since $(R_\eta)_*\mu=\mu$ and $R_\eta$ is a homeomorphism,
pullback by $R_\eta$ defines a unitary operator
$$(R_\eta)^*:L^2(\CM,\mu)\longrightarrow L^2(\CM,\mu),
 \qquad (R_\eta)^*(F):=F\circ R_\eta.$$
Furthermore, for every $h\in C(B_2)$,
and its class $H=[h]\in L^2(B_2,\mathfrak{h}_{B_2})$,
$$(R_\eta)^*\bigl((w_2|_\CM)^*(H)\bigr)
 =\bigl((w_2|_\CM)\circ R_\eta\bigr)^*(H)
 =(w_2|_\CM)^*(H).$$
Thus $(R_\eta)^*$ fixes $(w_2|_\CM)^*[C(B_2)]$ pointwise.
Since this subspace is dense in $L^2(\CM,\mu)$ by
Lemma~\ref{lemma: L2 Bohr}, we obtain $(R_\eta)^*=\id_{L^2(\CM,\mu)}$.
Consequently, for every $F\in C(\CM)$,
$F\circ R_\eta=F$ happens
$\mu$-almost everywhere.  Since $\supp(\mu)=\CM$ and both sides are
continuous, this equality holds everywhere.  Continuous functions
separate points, so $R_\eta=\id_\CM$.  Therefore $\eta=u\eta=R_\eta(u)=u$.
\end{proof}

\begin{proof}[Proof of Theorem \ref{thm:GNC 0}]
By \cite[Corollary~5.6]{HoffKru1}, the Ellis flow
$(E(X,G),G)$ is hereditarily amenable. 
Therefore, Theorem~\ref{thm: GNC for amenable minimal ideal} applies to every
$\CM\unlhd_m E(X,G)$ and every idempotent $u\in\CM$, and yields GNC
for $(X,G)$.
\end{proof}

\begin{cor}
\label{cor: tame amenable finite flows}
If a flow $(X,G)$ is finite then $B_2$ is finite as well.
Let $(X,G)$ be minimal, tame and amenable.
If $B_2$ is finite, then $|X|\leqslant|B_2|$. 
In particular, 
for a  minimal, tame and amenable $(X,G)$, $X$ is finite if and only if $B_2$ is finite.
\end{cor}

\begin{proof}
If $X$ is finite then also $E(X,G)$ is finite and so $B_2$ as a surjective image of $E(X,G)$.

Now assume that $B_2$ is finite. Theorem~\ref{thm: GNC for amenable minimal ideal} gives a topological
group isomorphism $u\CM\cong B_2$ and so $u\CM$ is finite.
Then, by Fact~\ref{fact: amenability closure},
$\CM=u\CM$. Since $X=\CM x$, it follows that $X$ is finite and bounded by the size of $B_2$.
\end{proof}

\noindent
See Proposition \ref{prop: sizes of X and B2} for a more general variant of the above corollary.

\subsection{Measured structure of amenable tame flows}\label{subsection: structure of amenable tame flows}
Now, we combine Lemma \ref{lemma: L2 Bohr} with a Galois-correspondence result from \cite{TakesakiTatsuuma} for von Neumann algebras, to characterize the measured structure of minimal, tame and amenable flows (in Theorem \ref{thm: measured vNeumann} and corresponding translation in Corollary \ref{rem: homogeneous MAlg}).

Let $(X,G)$ be a minimal
flow.  
Put $E:=E(X,G)$, choose $\CM\unlhd_m E$ and $u^2=u\in\CM$.
Choose $x\in uX$, so $ux=x$.  Moreover, set
$$\gal_X:=\gal(X,x)\quad\text{ and }\quad\gal_B:=w_2[\gal_X].$$

\begin{remark}\label{rem:NEW.closed.Gal}
$\gal_X$ and $\gal_B$ are closed (in the $\tau$-topology and in the topology of $B_2$, respectively).
\end{remark}

\begin{proof}
    By Remark \ref{rem:Gal.closed.0}, we have that $\gal_X$ is $\tau$-closed.
    Then $\gal_B$ is closed in $B_2$ as the
    image of a closed subset in quasi-compact $u\CM$ via a continuous map to a Hausdorff space.
    Here we use the $\tau$-continuity of $w_2|_{u\CM}$
from Remark~\ref{remark: w2=theta}.
\end{proof}

\begin{remark}\label{remark: rho X map} 
The formula
$$\rho_x:\CM x\longrightarrow B_2/\gal_B,\qquad \rho_x(\xi x):=w_2(\xi)\gal_B,$$
where $\xi\in\CM$, defines a continuous surjective $G$-equivariant map, where $B_2/\gal_B$ carries the quotient topology and the natural left $G$-action 
(i.e. $g\cdot b\gal_B:=\fb_2(g)b\gal_B$ for $g\in G$). 
Moreover, as $X$ is minimal, we have $\CM x=X$
and $\rho_x$ is a factor map of compact (Hausdorff) flows.
\end{remark}

\begin{proof}
If $\xi x=\zeta x$, then $(u\zeta)^{-1}(u\xi)\in\gal_X$,
and, since $w_2(u)=1_{B_2}$, this gives $w_2(\xi)\gal_B=w_2(\zeta)\gal_B$. Thus $\rho_x$ is well-defined. If $\ev_x:\CM\to\CM x$ is evaluation at $x$ and $q:B_2\to B_2/\gal_B$ is the quotient map, then $q\circ w_2|_{\CM}=\rho_x\circ \ev_x$ and so $\rho_x\circ \ev_x$ is continuous. 
Since $\ev_x$ is a quotient map, $\rho_x$ is continuous.
An equivariant continuous surjective map between compact (Hausdorff) spaces is a factor map.
\end{proof}

\begin{theorem}\label{thm: measured vNeumann}
Let $(X,G)$ be minimal, tame and amenable.
Then:
\begin{enumerate}[(A)]
\item The flows $X$ and $\CM$ carry unique $G$-invariant regular Borel probability measures, denoted by $\mu_X$ and $\mu_{\CM}$, 
and we have $\supp(\mu_X)=X$ and $\supp(\mu_{\CM})=\CM$. Thus, $(\ev_x)_*\mu_{\CM}=\mu_X$.

\item Pullback along $w_2|_{\CM}$ is a normal trace-preserving $*$-isomorphism
\begin{equation}
\label{eq: pullback-q-Linfty}
 (w_2|_{\CM})^*:L^\infty(B_2,\mathfrak h_{B_2})\longrightarrow L^\infty(\CM,\mu_{\CM}).
\end{equation}
On both $L^2$-level and $L^\infty$-level the map
$\Phi_x:=\big((w_2|_{\CM})^*\big)^{-1}\ev_x^*$ is $G$-equivariant and isometric, and on $L^\infty$-level it is a normal trace-preserving $*$-monomorphism.

\item
There exists a unique closed subgroup $H_x$ of $B_2$ such that
\begin{IEEEeqnarray*}{rCl}
\Phi_x[L^\infty(X,\mu_X)] &=& L^\infty(B_2,\mathfrak{h}_{B_2})^{H_x} \\
 &:=&\{F\in L^\infty(B_2,\mathfrak{h}_{B_2}):R_hF=F\text{ for every }h\in H_x\},
\end{IEEEeqnarray*}
where $(R_hF)(b):=F(bh)$.  
In fact, $H_x=\gal_B$.
Consequently,
\begin{equation}
\label{eq: measured homogeneous}
L^\infty(X,\mu_X) \cong_G L^\infty(B_2/\gal_B,\mathfrak h_{B_2/\gal_B}).
\end{equation}
(Here $G$ acts on $B_2/\gal_B$ by $g\cdot b\gal_B=\fb_2(g)b\gal_B$
and $\cong_G$ denotes a normal trace-preserving
$G$-equivariant $*$-isomorphism of commutative von Neumann algebras.)
\end{enumerate}
\end{theorem}

\begin{proof}
By Fact~\ref{fact: amenability closure} and 
\cite[Lemma~2.2(3)]{HoffKru1}, $(\CM,G)$ is amenable, minimal, and tame.
Fact~\ref{lemma : min tame unique ergodicity}, applied to $(X,G)$ and to $(\CM,G)$,
gives uniqueness of the $G$-invariant measures.  
The support of a $G$-invariant measure is a nonempty subflow, so minimality gives the assertions about support in (A).

We move on to the proof of (B). Lemma~\ref{lemma: L2 Bohr} gives
\begin{equation}
\label{eq:q push}
 (w_2|_{\CM})_*\mu_{\CM}=\mathfrak h_{B_2}
\end{equation}
and a unitary isomorphism
\begin{equation}
\label{eq:measured-q-L2}
 (w_2|_{\CM})^*:L^2(B_2,\mathfrak h_{B_2})
 \longrightarrow L^2(\CM,\mu_{\CM}).
\end{equation}
If $F\in L^\infty(\CM,\mu_{\CM})$, surjectivity in
\eqref{eq:measured-q-L2} gives $f\in L^2(B_2,\mathfrak h_{B_2})$ with
$F=f\circ w_2|_{\CM}$ almost everywhere.  For $C>\lVert F\rVert_\infty$,
\eqref{eq:q push} yields
$$\mathfrak h_{B_2}(|f|>C)=\mu_{\CM}(|f\circ w_2|_{\CM}|>C)=0.$$
Thus, $f\in L^\infty(B_2,\mathfrak{h}_{B_2})$ and
$\lVert f\rVert_\infty=\lVert f\circ w_2|_{\CM}\rVert_\infty$.  It follows that
$(w_2|_{\CM})^*$ restricted to $L^\infty(B_2,\mathfrak{h}_{B_2})$
is onto on $L^\infty(\CM,\mu_{\CM})$.  
Literal pullback and
\eqref{eq:q push} make it a trace-preserving $*$-isomorphism.
Every $*$-isomorphism between von Neumann algebras is an order
isomorphism, hence preserves suprema of bounded increasing nets of
positive elements and is therefore normal
\cite[Corollary~7.1.16]{KadisonRingroseII}.
This proves (B).

Now, we move to the proof of (C). Set $\mathcal A_x:=\Phi_x[L^\infty(X,\mu_X)]$.
We will prove that
\begin{equation}
\label{eq:measured-range-H}
\mathcal A_x
=L^\infty(B_2,\mathfrak h_{B_2})^{\gal_B}.
\end{equation}

We start with showing that
$\mathcal A_x\subseteq
L^\infty(B_2,\mathfrak h_{B_2})^{\gal_B}$.
For $\eta\in u\CM$, let $R_\eta:\CM\to\CM$ be given by
$R_\eta(\xi):=\xi\eta$.
This is a $G$-automorphism with inverse $R_{\eta^{-1}}$ (see also the proof of Theorem~\ref{thm: GNC for amenable minimal ideal}).
Hence $(R_\eta)_*\mu_{\CM}$ is $G$-invariant, and uniqueness in~\textup{(A)} gives $(R_\eta)_*\mu_{\CM}=\mu_{\CM}$.
Moreover, since $w_2$ is a semigroup homomorphism, $w_2|_{\CM}(R_\eta(\xi))=w_2(\xi)w_2(\eta)$ for every $\xi\in\CM$, and therefore $(w_2|_{\CM})^*\circ R_{w_2(\eta)}=(R_\eta)^*\circ (w_2|_{\CM})^*$.

\begin{clm}
For every $\eta\in\gal_X$, for every $F\in\mathcal A_x$,
we have $R_{w_2(\eta)}(F)=F$.
\end{clm}

\begin{clmproof}
Suppose that $\eta (x)=x$, then $\ev_x\circ R_\eta=\ev_x$.  
Hence, for every $F\in L^\infty(X,\mu_X)$,
\begin{IEEEeqnarray*}{rCl}
    (w_2|_{\CM})^*(R_{w_2(\eta)}\Phi_x(F))&=&
(R_\eta)^*((w_2|_{\CM})^*\Phi_x(F))=(R_\eta)^*(\ev_x^*(F))
\\ &=& \ev_x^*(F) = (w_2|_{\CM})^*\Phi_x(F).
\end{IEEEeqnarray*}
Since $(w_2|_{\CM})^*$ is injective, $R_{w_2(\eta)}$ fixes $\mathcal A_x$ pointwise.
\end{clmproof}

Since $\gal_B=w_2[\gal_X]$, the claim gives
$\mathcal A_x\subseteq L^\infty(B_2,\mathfrak h_{B_2})^{\gal_B}$.
For the reverse inclusion, let $p:B_2\to B_2/\gal_B$
be the quotient map. By definition of $\rho_x$, $p\circ(w_2|_{\CM})=\rho_x\circ \ev_x$.
Together with \textup{(A)} and \eqref{eq:q push}, this gives
$(\rho_x)_*\mu_X=p_*\mathfrak h_{B_2}
=\mathfrak h_{B_2/\gal_B}$.
Taking pullbacks, we obtain
$$(w_2|_{\CM})^*\circ p^*=\ev_x^*\circ\rho_x^*,
\qquad\text{hence}\qquad
p^*=\Phi_x\circ\rho_x^*.$$
Since $\gal_B$ is closed (Remark~\ref{rem:NEW.closed.Gal}),
the final assertion of
Fact~\ref{fact: Takesaki-Tatsuuma Galois} yields
$$
L^\infty(B_2,\mathfrak h_{B_2})^{\gal_B}
=p^*\bigl[L^\infty(B_2/\gal_B,\mathfrak h_{B_2/\gal_B})\bigr]
\subseteq\Phi_x[L^\infty(X,\mu_X)]
=\mathcal A_x.
$$
This proves \eqref{eq:measured-range-H}.
The subgroup-recovery formula in
Fact~\ref{fact: Takesaki-Tatsuuma Galois} shows that
$\gal_B$ is the unique closed subgroup of $B_2$
whose fixed-point algebra is $\mathcal A_x$.
%Thus $H_x=\gal_B$.
Taking $H_x:=\gal_B$, we obtain the existence and uniqueness
asserted in (C).

By the final assertion of Fact~\ref{fact: Takesaki-Tatsuuma Galois} and the identity $p_*\mathfrak h_{B_2}=\mathfrak h_{B_2/H_x}$, pullback gives a normal trace-preserving $*$-isomorphism
$$p^*:L^\infty(B_2/H_x,\mathfrak h_{B_2/H_x})\xrightarrow{\ \cong\ }\mathcal A_x=L^\infty(B_2,\mathfrak h_{B_2})^{H_x}.$$
Since $p$ intertwines left translation by $\fb_2(g)$ with the stated $G$-action on $B_2/H_x$, the map
$$(p^*)^{-1}\circ\Phi_x:L^\infty(X,\mu_X)\longrightarrow L^\infty(B_2/\gal_B,\mathfrak h_{B_2/\gal_B})$$
is a normal trace-preserving $G$-equivariant $*$-isomorphism.  This proves (C).
\end{proof}

\begin{cor}\label{rem: homogeneous MAlg}
For a probability space $(Y,\nu)$, let $\operatorname{MAlg}(Y,\nu)$ denote its complete Boolean algebra of measurable sets modulo null sets.  Assertion~\textup{(C)} equivalently gives a measure-preserving $G$-isomorphism
$$\operatorname{MAlg}(X,\mu_X)\cong_G\operatorname{MAlg}(B_2/\gal_B,\mathfrak h_{B_2/\gal_B})$$
for a minimal, tame and amenable flow $(X,G)$.
\end{cor}

The next lemma is a consequence of Fact~\ref{fact: amenability closure}
and GNC (i.e., Theorem \ref{thm:GNC 0}).

\begin{lemma}
\label{lemma: dense uX}
Let $(X,G)$ be minimal, tame and amenable.
The set $uX$ is dense in $X$ and the restriction of the map from Remark \ref{remark: rho X map}, 
$$\rho_x|_{uX}:uX\longrightarrow B_2/\gal_B$$
is a bijection. 
\end{lemma}

\begin{proof}
First, we prove density. By minimality of $X$, $X=\CM x$ and so $uX=(u\CM)x$.
By Fact~\ref{fact: amenability closure}, $\overline{uX}=\overline{(u\CM)x}=\overline{u\CM}x=\CM x=X$.

The restriction $\rho_x|_{uX}$ is onto because $w_2|_{u\CM}$ is onto. To prove injectivity, suppose that $\eta,\kappa\in u\CM$ and $\rho_x(\eta x)=\rho_x(\kappa x)$. Then $w_2(\kappa^{-1}\eta)\in\gal_B=w_2[\gal_X]$. 
Since $w_2|_{u\CM}$ is injective by
Theorem~\ref{thm: GNC for amenable minimal ideal},
$\kappa^{-1}\eta\in\gal_X$, and consequently $\eta x=\kappa x$.
Hence $\rho_x|_{uX}$ is a bijection.
\end{proof}

There is a natural structure of a homogeneous $u\CM$-space
on $u\CM/\gal_X$, and similarly a structure of 
a homogeneous $B_2$-space on $B_2/\gal_B$.
If GNC holds, i.e. if $w_2|_{u\CM}:u\CM\to B_2$ is an isomorphism, then, 
using the definition of $\gal_B$,
we conclude that $u\CM/\gal_X$ and $B_2/\gal_B$ are isomorphic as homogenoues spaces. Below, we provide an alternative proof which uses Lemma \ref{lemma: dense uX}.

\begin{cor}
    Let $(X,G)$ be minimal, tame and amenable.
    Then
    $$B_2/\gal_B \cong u\CM/\gal_X$$
    as homogeneous spaces.
\end{cor}

\begin{proof}
    Consider the composition of maps $\ev_x$ and $\rho_x|_{uX}$ from Lemma \ref{lemma: dense uX}, which is on the other hand equal to the composition of $w_2|_{u\CM}$ with the quotient map $B_2\to B_2/\gal_B$. 
    It factors through $u\CM/\gal_X$, giving 
    a homeomorphism of homogeneous spaces
    $B_2/\gal_B \cong u\CM/\gal_X$. 
\end{proof}

\begin{proposition}\label{prop: sizes of X and B2}
Let $(X,G)$ be minimal, tame and amenable.
If $X$ is first countable (e.g. if $X$ is metrizable) then $|X|= |B_2/\gal_B|$.
\end{proposition}

\begin{proof}
By Lemma~\ref{lemma: dense uX}, $|B_2/\gal_B|=|uX|\leqslant |X|$.
We need to argue for other inequality. In the case of finite $|B_2/\gal_B|=|uX|$, by density of $uX$ in $X$, we obtain also $|uX|=|X|$ and so $|B_2/\gal_B|=|X|$. Let us assume that $B_2/\gal_B$ is infinite.

By Lemma \ref{lemma: dense uX}, we have that $uX$ is dense in $X$.
Thus, every point of $X=\overline{uX}$ is then the limit of a sequence in $uX$, so
$$|X|\leqslant |uX|^{\aleph_0}=|B_2/\gal_B|^{\aleph_0}.$$

By Remark~\ref{rem:NEW.closed.Gal}, $\gal_B$ is closed, so $B_2/\gal_B$ is a compact Hausdorff homogeneous space.

By homogeneity, the minimum cardinality $\chi(b\gal_B,B_2/\gal_B)$ of a local base does not depend on the choice of point $b\gal_B$ in $B_2/\gal_B$.
Thus, for the cardinal 
$w(B_2/\gal_B)=\chi(B_2/\gal_B)=\sup\{\chi(b\gal_B,B_2/\gal_B) : b\gal_B \in B_2/\gal_B\}+\omega$, Theorem 7.19 from \cite{HodelCardinalFunctionsI}
gives us
$$2^{w(B_2/\gal_B)}=2^{\chi(B_2/\gal_B)}\leqslant |B_2/\gal_B|.$$
As $B_2/\gal_B$ is Hausdorff, we have $|B_2/\gal_B|\leqslant 2^{w(B_2/\gal_B)}$. Hence $2^{w(B_2/\gal_B)}=|B_2/\gal_B|$, and so
$|B_2/\gal_B|^{\aleph_0}=|B_2/\gal_B|$.
\end{proof}

Corollary~\ref{cor: tame finite-dimensional L2} might be used in the weakly-mixing set-up.
Recall that a probability-measure-preserving action $G\curvearrowright(X,\mu)$ is \emph{weakly mixing} if the diagonal action of $G$ on $(X\times X,\mu\otimes\mu)$ is ergodic. 
As the proof below recalls, weak mixing excludes nonzero finite-dimensional $G$-invariant subspaces of $(\mathbb C\mathbf 1)^{\perp}\subseteq L^2(X,\mu)$, whereas Corollary~\ref{cor: tame finite-dimensional L2} says that, for tame flows, finite-dimensional invariant subspaces generate $L^2(X,\mu)$. 
Thus the only (non-zero) finite-dimensional $G$-invariant subspaces must be
$L^2(X,\mu)^G=\mathbb{C}\mathbf{1}$ 
and so $L^2(X,\mu)=\mathbb{C}\mathbf{1}$:

\begin{cor}
\label{cor: amenable no mixing}
Let $(X,G)$ be minimal, tame and amenable.
If the probability-measure-preserving action of $G$ on
$(X,\mu)$ is weakly mixing, then $X$ is a singleton.
\end{cor}

\begin{proof}
Let $W\subseteq (\mathbb C\mathbf 1)^\perp\subseteq L^2(X,\mu)$ be a nonzero finite-dimensional $G$-invariant subspace, and let $f_1,\ldots,f_n$ be an orthonormal basis of $W$. Then
$$F(x,y):=\sum_{i=1}^n f_i(x)\overline{f_i(y)}$$
defines an element of $L^2(X\times X,\mu\otimes\mu)$. 
For every $g\in G$, the functions $U_gf_1,\ldots,U_gf_n$ form another orthonormal basis of $W$. 
Thus $F$ is invariant under the diagonal action. 
Since every $f_i\in (\mathbb C\mathbf 1)^\perp$,
every $f_i$ has integral zero, and so we have
$$\int_{X\times X}F\,d(\mu\otimes\mu)=0.$$
On the other hand, orthonormality gives $\lVert F\rVert_2^2=n>0$.
Thus $F$ is a nonzero invariant function orthogonal to the constants, contradicting weak mixing. Therefore $(\mathbb C\mathbf 1)^\perp\subseteq L^2(X,\mu)$ contains no nonzero finite-dimensional $G$-invariant subspace.

Now, Corollary~\ref{cor: tame finite-dimensional L2} yields $L^2(X,\mu)=\mathbb C\mathbf 1$. By Theorem~\ref{thm: measured vNeumann}\textup{(A)}, $\supp(\mu)=X$, so the canonical map $C(X)\to L^2(X,\mu)$ is injective. Consequently $C(X)=\mathbb C\mathbf 1$, and $X$ is a singleton.
\end{proof}

\subsection{MEF for tame, minimal and amenable flows}
\label{subsec: X-MEF 2}
In this subsection, we apply the results of the previous subsection
to the maximal equicontinuous factor of a minimal, tame and amenable
flow. We show that the factor map is proximal and induces
isomorphisms on the von Neumann algebras level.
We begin with a general observation for minimal flows.

We work with the set-up and notation of the previous subsection, 
in particular $(X,G)$ is a minimal flow.
Let $\pi:X\longrightarrow Z$ be MEF of $(X,G)$, and let
$$\widehat\pi:E(X,G)\longrightarrow E(Z,G)$$
be induced by $\pi$.  
Then set $z:=\pi(x)$ and $K:=u\CM$.

\begin{lemma}\label{lemma: MEF B/G}
There is a unique $G$-homeomorphism
$$\gamma:B_2/\gal_B\longrightarrow Z$$
such that $\pi=\gamma\circ\rho_x$ ($\rho_x$ as in Remark \ref{remark: rho X map}). It is given by
$\gamma\bigl(w_2(\eta)\gal_B\bigr)=\widehat\pi(\eta)(z)$,
where $\eta\in K$.
Moreover, $\widehat\pi[K]=E(Z,G)$.
\end{lemma}

\begin{proof}
By Remark~\ref{remark: rho X map}, the map $\rho_x:X\longrightarrow B_2/\gal_B$
is a factor map onto a minimal equicontinuous flow.
The universal property of the Bohr compactification gives a
continuous homomorphism
$$\beta:B_2\longrightarrow E(Z,G)$$
such that $\beta\circ w_2=\widehat\pi$.
Its image is compact and contains the dense image of $G$ in $E(Z,G)$,
so $\beta$ is onto. Since $w_2[K]=B_2$, it follows that $\widehat\pi[K]=E(Z,G)$.

Let $p:B_2\to B_2/\gal_B$ be the quotient map.
Consider the following diagram:
$$\xymatrix{\CM x \ar[dr]_-{\rho_x} & B_2 \ar[r]^-{\beta} \ar[d]^-{p} & E(Z,G) \ar[d]^-{\ev_z}  \\
& B_2/\gal_B & Z}$$
Note that if $b \in\gal_B$ then $\beta(b)(z)=z$.
Indeed, there is some $\eta\in\gal_X$ with $b=w_2(\eta)$, and so
$$\beta(b)(z)=(\beta w_2)(\eta)\big(\pi(x)\big)=\widehat{\pi}(\eta)\big(\pi(x)\big)=\pi\big(\eta(x)\big)=\pi(x)=z.$$
The formula $\gamma(b\gal_B):=\beta(b)(z)$ defines a continuous $G$-map from $B_2/\gal_B$ to $Z$. 
By decoding the definitions, we obtain easily $\gamma\circ\rho_x=\pi$.

Since $\rho_x:X\to B_2/\gal_B$ is an equicontinuous factor of $X$
and $\pi:X\to Z$ is MEF, there exists a factor map
$$\alpha:Z\longrightarrow B_2/\gal_B$$
such that $\rho_x=\alpha\circ\pi$.
We have $\gamma\circ\alpha\circ\pi=\gamma\circ\rho_x=\pi$ and so,
surjectivity of $\pi$ gives $\gamma\circ\alpha=\id_Z$. 
Since $\alpha$
is onto, it follows that $\alpha$ and $\gamma$ are inverse
$G$-homeomorphisms. Uniqueness of $\gamma$ follows from the
surjectivity of $\rho_x$.
\end{proof}

From now on, assume additionally that $(X,G)$ is tame and amenable,
and let $\mu_X$ be its unique $G$-invariant measure.
Let $\lambda_Z$ be the unique $G$-invariant measure
on the flow $Z$. 
We have $ \pi_*\mu_X=\lambda_Z$.

\begin{theorem}
\label{thm: amenable X-MEF 2}
\begin{enumerate}
\item The pullback of $\pi$ 
is a normal trace-preserving $G$-equivariant \\ $\ast$-isomorphism of commutative von Neumann algebras:
\begin{equation}
\label{eq:MEF-Linfty-isomorphism}
\pi^*:L^\infty(Z,\lambda_Z)\xrightarrow{\ \cong\ }L^\infty(X,\mu_X)
\end{equation}
and a $G$-equivariant unitary isomorphism
$$\pi^*:L^2(Z,\lambda_Z)\xrightarrow{\ \cong\ }L^2(X,\mu_X).$$

\item 
For every $y\in Z$, $u[\pi^{-1}(y)]$ consists of the unique point of $uX$ lying over $y$. In particular, $\pi$ is proximal.
\end{enumerate}
\end{theorem}

\begin{proof}
We work in the notation of 
Theorem \ref{thm: measured vNeumann} and its proof,
Remark \ref{remark: rho X map} and Lemma \ref{lemma: MEF B/G}.
\begin{clm}
    $\ev_x^*\circ\, \pi^* = (w_2|_{\CM})^*\circ p^* \circ \gamma^*$.
\end{clm}

\begin{clmproof}
    Let $F\in L^{\infty}(Z,\lambda_Z)$ and $F=[f]$ for some $f\in C(Z)$.
    Then for every $\eta\in\CM$
    \begin{IEEEeqnarray*}{rCl}
        (w_2|_{\CM}^*\circ p^*\circ \gamma^*)(f)(\eta) &=& (f\circ \gamma\circ p \circ w_2|_{\CM})(\eta) \\
        &=& (f\circ \ev_z\circ\beta\circ w_2|_{\CM})(\eta) \\
        &=& (f\circ\ev_z\circ\,\widehat{\pi})(\eta)=f\big(\widehat{\pi}(\eta)(z)\big) \\
        &=& f\big(\pi(\eta(x))\big)=(f\circ\pi\circ\ev_x)(\eta) \\
        &=& (\ev_x^*\circ \pi^*)(f)(\eta)
    \end{IEEEeqnarray*}
    Thus $(\ev_x^*\circ\, \pi^*)(F) = \big((w_2|_{\CM})^*\circ p^* \circ \gamma^*\big)(F)$. By the weak-$\ast$-density of $C(Z)$ in $L^{\infty}(Z,\lambda_Z)$,
    we obtain $\ev_x^*\circ\, \pi^* = (w_2|_{\CM})^*\circ p^* \circ \gamma^*$
    on the entire $L^{\infty}(Z,\lambda_Z)$.
\end{clmproof}
The map
$$\Psi_x:=\Big((p^*)^{-1}\circ (w_2|_{\CM}^*)^{-1}\circ \ev_x^*\Big):L^\infty(X,\mu_X)\longrightarrow L^\infty(B_2/\gal_B,\mathfrak h_{B_2/\gal_B})$$
is the normal trace-preserving $G$-equivariant $*$-isomorphism
from Theorem~\ref{thm: measured vNeumann}\textup{(C)}.
Now, by applying $(p^*)^{-1}\circ (w_2|_{\CM}^*)^{-1}$ to both sides of the equality from the above claim, we obtain $\Psi_x\circ\pi^*=\gamma^*$.
Because $\gamma$ is a $G$-homeomorphism and $\Psi_x$ is 
a normal trace-preserving $G$-equivariant $*$-isomorphism, 
we conclude that the map $\pi^*:L^\infty(Z,\lambda_Z)\xrightarrow{\ \cong\ }L^\infty(X,\mu_X)$
is a normal trace-preserving $G$-equivariant $\ast$-isomorphism of commutative von Neumann algebras.

Pullback $\pi^*$ is an isometry on the $L^2$-level. 
By the just obtained $L^\infty$-level part, its range contains $L^\infty(X,\mu_X)$ and is therefore dense in $L^2(X,\mu_X)$. The range of an isometry is closed, so the extension is onto. This proves the first point.

We now prove the second point. 
By Lemmas~\ref{lemma: MEF B/G} and~\ref{lemma: dense uX},
the map $\pi|_{uX}=\gamma\circ(\rho_x|_{uX})$ is bijective.
Thus, $(\pi|_{uX})^{-1}(y)$ consists of a single point.
Let $x'\in u[\pi^{-1}(y)]$, i.e. $x'=ux''$ and $\pi(x'')=y$ for some $x''\in X$.
Then
$$\pi(ux'')=\widehat{\pi}(u)\pi(x'')=\widehat{\pi}(u)y=y,$$
since $\widehat{\pi}(u)$ is the only idempotent in $E(Z,G)$.
Thus, $ux''\in \pi|_{uX}^{-1}(y)$. 

Now, if $x'\in \pi|_{uX}^{-1}(y)$,
then $x'=ux''$ for some $x''\in X$ and $\pi(x')=y$.
We have $y=\pi(ux'')$ and so $ux''\in\pi^{-1}(y)$.
Hence, $x'=ux''=u(ux'')\in u[\pi^{-1}(y)]$.

Consequently, $u$ equalizes every pair of points in a common fibre of $\pi$, so $\pi$ is proximal.
\end{proof}

The above result will be used in a future paper by the first author
and Andy Zucker to extend the main results of \cite{Glasner2}
to nonmetrizable flows.

\section{Specializations to model-theoretic contexts}\label{section: model theory}
The goal of this section is to explain that the notions and results from this paper generalize the appropriate notions and results in the model-theoretic contexts of definable groups and of first order theories. We recover some known and obtain some new results in these contexts.

In Subsection \ref{subsection: definable groups}, we explain that the GNC specializes in the context of definable NIP groups to the original Newelski's Conjecture (which justifies the name Generalized Newelski's Conjecture), and Theorem \ref{theorem: GNC} implies Fact \ref{fact: Theorem 1.5 from CS in intro}. In Subsection \ref{subsection: first order theories}, we analyze the specialization of the GNC to first order theories and discuss the connection to Newelski's Conjecture for first order theories and Fact \ref{fact: NC for theories in intro}. In particular, we obtain a new result that the Bohr compactification of the group $\aut(\FC)$ of automorphisms of the monster model $\FC$ below the ambit $(S_{\bar c}(\FC),\tp(\bar c/\FC),\aut(\FC))$ (which is the core ambit for this context) is precisely the Kim-Pillay Galois group of the theory in question.

We will also point out some phenomena in the context of definable groups (and some also for first order theories) which fail drastically in the general abstract context, which may lead to some new (or existing) dynamical properties that should be assumed to get similar results in general. In Subsection \ref{subsection: two problems}, we state two natural problems that arise from these observations.

One of these phenomena is that for a group $G$ definable in an NIP structure $M$, the Bohr compactification of $G$ below the ambit $(S_G(M), \tp(e/M),G)$ coincides with the Bohr compactification of $G$ below the Elllis flow $(E(S_G(M),G),\id, G)$. We leave as a question whether an analog holds for first order theories. This fails in general for tame ambits as shown by Example \ref{example: B_1 not equal  to B_2}. Another phenomenon is the transfer of amenability from an ambit to its Ellis flow, which holds both for definable groups and for first order theories (in the latter context, this is a new result proved in Subsection  \ref{subsection: first order theories}) under the NIP assumption, but fails for general tame ambits (we skip an example as it is pretty long).

In Subsection \ref{subsection: local NC}, using Theorem \ref{theorem: GNC}, we obtain a local version of Newelski's conjecture in the context of definable groups. The main point is a description of the Bohr compactification of $G$ below an ambit of the form $(S_{G, \Delta}(M),\tp_{\Delta}(e),G)$ for suitable families of formulas $\Delta$, which in particular yields a description of the externally definable Bohr compactification of $G$.

\subsection{Definable groups}\label{subsection: definable groups}
For the notions and basic facts recalled below, see \cite{newelski09,GPPb,Krupinski2017}.

Let $G$ be a group definable in a structure $M$. Let $N \succ M$ be $|M|^+$-saturated, and $\mathfrak{C}\succ N$ a monster model (i.e. a $\kappa$-saturated and strongly $\kappa$-homogeneous model for a sufficiently large cardinal $\kappa$). Let $\bar{G}:=G(\mathfrak{C})$, i.e. the interpretation of $G$ in $\mathfrak{C}$. Now, we recall the main two flows considered in the context of definable groups.

By $S_G(M)$ we denote the space of complete types over $M$ concentrated on $G$ (i.e., containing the formula defining $G$). Then $(S_G(M),\tp(e/M),G)$ is an ambit with the action defined by $g \cdot \tp(a/M) :=\tp(ga/M)$ (where $a \in \bar{G}$). Note that the elements of $S_G(M)$ can be also treated as ultrafilters in the Boolean algebra of definable subsets of $G$; then the above action is just the obvious action by left translation.

By $S_{G,M}(N)$ we denote the space of complete types over $N$ concentrated on $G$ and finitely satisfiable in $M$. Then $(S_{G,M}(N),\tp(e/N),G)$ is an ambit with the action defined by $g \cdot \tp(a/N) :=\tp(ga/N)$ (where $a \in G(\mathfrak{C})$ is such that $\tp(a/N) \in S_{G,M}(N)$ ). This ambit is isomorphic to the ambit $(S_{G,\textrm{ext}}(M), \tp_\textrm{ext}(e/M), G)$ of complete external types on $G$ over $M$ with the action given by left translation, where a complete external type on $G$ over $M$ is an ultrafilter in the Boolean algebra of externally definable subsets of $G$. We will be identifying these two ambits.

The advantage of $S_{G,M}(N)$ over $S_G(M)$ is that  $S_{G,M}(N)$ carries a left continuous semigroup operation $*$ extending the action of $G$: $$p*q =\tp(ab/N),$$
where $b \models q$ and $a \models p$ is such that $\tp(a/N,b)$ is finitely satisfiable in $M$. So $S_{G,M}(N)$ can and will be identified with $E(S_{G,M}(N),G)$ as a semigroup and $G$-ambit (see Fact \ref{fact: existence of semigroup operation}).

The {\em definable Bohr compactification of $G$} is the universal group compactification of $G$ such that the preimages of any two disjoint closed subsets of the target compact group can be separated by a definable subset of $G$. This is precisely the Bohr compactification of $G$ below the ambit $(S_G(M),\tp(e/M),G)$. Similarly, {\em the externally definable Bohr compactification of $G$} is the universal group compactification of $G$ such that the preimages of any two disjoint closed subsets of the target compact group can be separated by an externally definable subset of $G$. This is precisely the Bohr compactification of $G$ below the ambit $(S_{G,\textrm{ext}}(M), \tp_\textrm{ext}(e/M), G)$.

So we have three natural Bohr compactifications of $G$: 
\begin{enumerate}
\item $\mathfrak{b}_1 \colon G \to B_1$ --- the one below $(S_G(M),\tp(e/M),G)$,
\item $\mathfrak{b}_2 \colon G \to B_2$ --- the one below $E(S_G(M),G)$, 
\item $\mathfrak{b}_3 \colon G \to B_3$ --- the one below  $(S_{G,\textrm{ext}}(M), \tp_\textrm{ext}(e/M), G)$. 
\end{enumerate}

We have the obvious epimorphisms of $G$-ambits 
$$\xymatrix{S_{G,\textrm{ext}}(M) \ar[r] & E(S_G(M),G) \ar[rr]^{\ev_{\tp(e/M)}} & & S_G(M)},$$
which yield epimorphisms between the corresponding Bohr compactifications with the properties as in Remark \ref{remark: B2 to B1}
$$\xymatrix{ B_3 \ar[r]^{\mathfrak{g}} & B_2 \ar[r]^{\mathfrak{f}} & B_1}.$$

Proposition 3.4 in \cite{GPPb} yields $B_1=\bar{G}/\bar{G}^{00}_M$ and $\mathfrak{b}_1(g)=g/\bar{G}^{00}_M$, where $\bar{G}^{00}_M$ is the smallest $M$-type-definable subgroup of $\bar{G}$ of bounded index (i.e., less than $\kappa$). 

\medskip

Let $u^2=u \in \mathcal{M} \unlhd_m S_{G,M}(N)$. 

\begin{definition}\label{definition: Newelski's map}
{\em Newelski's map} is the group epimorphism $\rho \colon u\CM \to \bar{G}/\bar{G}^{00}_M$ given by $\rho(\tp(a/N)):=a/\bar{G}^{00}_M$.
\end{definition}

{\em Newelski's conjecture} predicted that under suitable model-theoretic assumptions, the epimorphism $\rho$ is an isomorphism. It was refuted in \cite{GismatullinPenazziPillay2015Some} for groups definable in NIP theories, concretely for $G:=\SL_2(\mathbb{R})$ treated as a group definable in $M:=(\mathbb{R},+,\cdot)$. The main positive results is the following, where definable amenability of $G$ means that the flow $(S_G(M),G)$ is amenable.

\begin{fact}[Theorem 1.5 in \cite{ArtemPierre}]\label{fact: Theorem 1.5 from CS}
Newelski's conjecture holds for definably amenably groups definable in NIP theories.
\end{fact}

As discussed in the introduction, the various refinements and variants of Newelski's map \cite{Krupinski2017}, \cite{KruPiRze}, \cite{KrRz}, \cite{Krupinski2023a} led to important contributions in model theory (e.g. to understand complexity of strong types) and to additive combinatorics (to construct generalized locally compact models for arbitrary approximate subgroups in a simpler and ``more classical'' way than originally done by Hrushovski).

Recall the map $\widehat{\theta} \colon u\CM \to B_3$ from Corollary \ref{corollary: Newelski's map in general} and Theorem \ref{theorem: GNC}.

\begin{lemma}\label{lemma: Newelski's map and theta for definable groups}
The Newelski's map $\rho \colon u\CM \to \bar{G}/\bar{G}^{00}_M$ coincides with the composition $\mathfrak{f} \circ \mathfrak{g} \circ \widehat{\theta}$.
\end{lemma}

\begin{proof}
Since $\widehat{\theta}(p) := p e_{B_3}$ and both $\mathfrak{g}$ and $\mathfrak{f}$ are epimorphisms of $G$-ambits (see the comment after Definition \ref{definition: Bohr compactification below ambit B}), we have $(\mathfrak{f} \circ \mathfrak{g} \circ \widehat{\theta})(p) = pe^{B_1} =p (e/\bar{G}^{00}_M)$. In order to finish the proof, it is enough to check that the unique left continuous action of $S_{G,M}(N)$ on $\bar{G}/\bar{G}^{00}_M$ extending the action of $G$ is given by $p (b/\bar{G}^{00}_M) =ab/\bar{G}^{00}_M$ for $a \models p$ such that $\tp(a/N,b)$ is finitely satisfiable in $M$. The easiest way to see it is by showing that the formula $p \cdot (b/\bar{G}^{00}_M): =ab/\bar{G}^{00}_M$ for $a \models p$ such that $\tp(a/N,b)$ is finitely satisfiable in $M$ yields a well-defined left continuous function extending the action of $G$. It is well-defined, because whenever $a\equiv_M a'$, then $a$ and $a'$ have the same Lascar strong type over $M$ (see Subsection \ref{subsubsection: Bohr below S_c(C)}) so they are in the same coset modulo $\bar{G}^{00}_M$ (as lying in the same coset is a bounded $M$-invariant equivalence relation on $\bar{G}$). The fact that it extends the action of $G$ is obvious. To show continuity, for a basic open set $[\varphi(x,c)]$ in $S_{G,M}(N)$ (so $c \in N$) and an element $b/\bar{G}^{00}_M \in \bar{G}/\bar{G}^{00}_M$, we need to show that the set $U:=\{p \in S_{G,M}(N): p \cdot (b/\bar{G}^{00}_M) \in [\varphi(x,c)]\}$ is open. Since the restriction map $r \colon S_{G,M}(N,b) \to S_{G,M}(N)$ is a homeomorphism (by $|M|^+$-saturation of $N$) , it is enough to see that $r^{-1}[U]$ is open. And this is true, because for $a \in \bar{G}$ such that $\tp(a/N,b) \in S_{G,M}(N,b)$ we have
$$\tp(a/N,b) \in r^{-1}[U] \iff \tp(ab/N) \in [\varphi(x,c)] \iff \tp(a/N,b) \in [\varphi(xb,c)].$$
\end{proof}

Thus, we see the GNC (i.e., Conjecture \ref{conj: GNC}, equivalently \ref{conjecture: main conjecture}) applied to the ambit $S_{G,M}(N)$ is weaker than the original Newelski's conjecture. Note that Example \ref{example: B_1 not equal  to B_2} shows that working in the general abstract context of an ambit $(X,x_0,G)$, even under the very strong assumptions of minimality and equicontinuity, it is not true that $\mathfrak{f} \circ \widehat{\theta} \colon u\CM \to B_1$ is an isomorphism, because the canonical epimorphism $\mathfrak{f} \colon B_2 \to B_1$ need not be an isomorphism. 
%That is way it would be hard to conjecture that $\mathfrak{f} \circ \widehat{\theta}$ is an isomorphism under some reasonable assumptions.
Thus, it seems hard to find reasonable assumptions which would imply that $\mathfrak{f} \circ \widehat{\theta}$ is an isomorphism (cf. Problem \ref{problem: 2}).

On the other hand, in the NIP context (i.e., when the theory of $M$ has NIP) Theorem 4.15 of \cite{CPS14} implies that $\mathfrak{f} \circ \mathfrak{g}$ is an isomorphism, so 

\begin{cor}
In the NIP context, the conclusion of the GNC applied to $S_{G,M}(N)$ implies the conclusion of Newelski's conjecture.
\end{cor}

Note that Example 4.17 in \cite{KLM2} shows that without NIP (still in the model-theoretic context of the ambit $(S_G(M),\tp(e/M),G)$), $\mathfrak{f}$ need not be an isomorphism.

By \cite[Theorem 3.17]{CPS14}, under NIP, definable amenability of $G$ implies amenability of the flow $(S_{G,M}(N),G)$. Moreover NIP together with Shelah's theorem from \cite{She2} (saying that if $M$ has NIP, then the Shelah expansion $M^{\textrm{Sh}}$ of $M$ has quantifier elimination and NIP) implies that $(S_{G,M}(N),G)$ is tame (see Remark 1.6 in \cite{ArtemPierre}). Therefore, 

\begin{cor}\label{corollary: GNC implies NC for definably amenable NIP groups}
Theorem \ref{theorem: GNC} implies Fact \ref{fact: Theorem 1.5 from CS}.
\end{cor}

Fact \ref{fact: amenability closure} and Lemma \ref{lemma: cl(uM)=M implies strong amenability} show that whenever $\mathcal{A}$ is a tame amenable ambit admitting a left continuous semigroup operation extending the action of $G$, then amenability of $\mathcal{A}$ implies strong amenability of $\mathcal{A}$. This yields the left-to-right  implication in the following

\begin{fact}[Theorem 0.6 in \cite{Krupinski2017}]
If $G$ is a group definable in a NIP theory, then G is definably amenable (equivalently externally definable amenable) if and only if  it is externally definably strongly amenable.
\end{fact}

\subsection{First order theories}\label{subsection: first order theories}
Let $\mathfrak{C}$ be a monster model of a first order theory $T$. Let $\bar c$ be an enumeration of $\mathfrak{C}$ and $S_{\bar c}(\mathfrak{C})$ the space of complete types over $\mathfrak{C}$ (in a very long tuple of variables corresponding to $\bar c$) extending $\tp(\bar c/\emptyset)$. Then $(S_{\bar c}(\mathfrak{C}),\tp(\bar c/\mathfrak{C}), \aut(\mathfrak{C}))$ is an ambit with respect to the obvious action. As mentioned in the introduction, this ambit plays a central role in the topological dynamics of first order theories, as it captures a lot of information about the underlying theory $T$ (e.g., see \cite{KruPiRze,Krupinski2019,Hrushovski2019,KrRz,HruKruPi,GHK}).

\subsubsection{Bohr compactification of $\aut(\mathfrak{C})$ below  $(S_{\bar c}(\mathfrak{C}),\tp(\bar c/\mathfrak{C}), \aut(\mathfrak{C}))$}\label{subsubsection: Bohr below S_c(C)}

The goal is to prove that the Bohr compactification  of $\aut(\mathfrak{C})$ below  $(S_{\bar c}(\mathfrak{C}),\tp(\bar c/\mathfrak{C}), \aut(\mathfrak{C}))$ coincides with the quotient map from $\aut(\mathfrak{C})$ to the Kim-Pillay Galois group $\gal_{\textrm{KP}}(T)$ of the theory $T$. This is counterpart for first order theories of Proposition 3.4 in \cite{GPPb} which says that for a  group $G$ definable in a structure $M$ the quotient map $G \to \bar{G}/\bar{G}^{00}_M$ is the definable Bohr compactification of $G$.

First, we briefly recall basic knowledge on Lascar strong types and Galois groups. A longer exposition can be found for instance in Subsection 2.6 of \cite{GHK}. For details the reader may consult \cite{Lascar2001,Casanovas2001}.

The relation of having the same {\em Lascar strong type}, denoted by $E_\textrm{L}$ or by $\equiv_{\textrm{Ls}}$, is the finest bounded (i.e. with less than $\kappa$ classes) invariant equivalence relation on a given product of sorts. The classes of $E_\textrm{L}$ are called Lascar strong types. It turns out that $E_L$ is the transitive closure of the relation of having the same type over a model (where the model is varying). The group of {\em Lascar strong automorphisms}, denoted by $\autf_\textrm{L}(\mathfrak{C})$,  is the subgroup of $\aut(\mathfrak{C})$ consisting of the automorphisms preserving all Lascar strong types in all (possibly infinite) numbers of variables. The {\em Lascar Galois group} of $T$, denoted by $\gal_{\textrm{L}}(T)$, is the quotient $\aut(\mathfrak{C})/\autf_\textrm{L}(\mathfrak{C})$.

The relation of having the same {\em Kim-Pillay strong type}, denoted by $E_\textrm{KP}$ or by $\equiv_{\textrm{KP}}$, is the finest bounded $\emptyset$-type-definable equivalence relation (on a given product of sorts). The classes of $E_\textrm{KP}$ are called {\em Kim-Pillay strong types}. The group of {\em Kim-Pillay strong automorphisms}, denoted by $\autf_\textrm{KP}(\mathfrak{C})$,  is the subgroup of $\aut(\mathfrak{C})$ consisting of the automorphisms preserving all Kim-Pillay strong types in all (possibly infinite) numbers of variables. The {\em Kim-Pillay Galois group} of $T$, denoted by $\gal_{\textrm{KP}}(T)$, is the quotient $\aut(\mathfrak{C})/\autf_\textrm{KP}(\mathfrak{C})$. We have that $\autf_\textrm{L}(\mathfrak{C}) \leq \autf_\textrm{KP}(\mathfrak{C})$, so we have the obvious epimorphism $\gal_{\textrm{L}}(T) \to \gal_{\textrm{KP}}(T)$. The theory $T$ is called {\em G-compact} if $\autf_\textrm{L}(\mathfrak{C}) = \autf_\textrm{KP}(\mathfrak{C})$.

Both groups $\gal_{\textrm{L}}(T)$ and $\gal_{\textrm{KP}}(T)$ do not depend on the choice of $\FC$, up to a natural isomorphism.
 
Both groups $\gal_{\textrm{L}}(T)$ and $\gal_{\textrm{KP}}(T)$ are equipped with the logic topology which make these groups respectively a quasi-compact and compact Hausdorff topological group. To define the logic topologies, choose any small $M \prec \mathfrak{C}$ and its enumeration $\bar m$.  Let $\pi_{\textrm{L}} \colon \aut(\mathfrak{C}) \to \aut(\mathfrak{C})/\autf_\textrm{L}(\mathfrak{C})$ and $\pi_{\textrm{KP}} \colon \aut(\mathfrak{C}) \to \aut(\mathfrak{C})/\autf_\textrm{KP}(\mathfrak{C})$ be the quotient maps, and $t \colon \aut(\mathfrak{C}) \to S_{\bar m}(M)$ be given by $t(\sigma):=\tp(\sigma(\bar m)/M)$. We have the following commutative diagram:
$$\xymatrix{\aut(\mathfrak{C}) \ar[rr]^{\pi_{\textrm{L}}} \ar[dr]^{t}& & \gal_{\textrm{L}}(T) \ar[r] & \gal_{\textrm{KP}}(T)  \\ & S_{\bar{m}}(M) \ar[ur]^{\pi_{\textrm{L}},\bar{m}} \ar[urr]_{\pi_{\textrm{KP}},\bar{m}} & &}$$
The {\em logic topologies} on $\gal_{\textrm{L}}(T)$ and $\gal_{\textrm{KP}}(T)$ are defined as the quotient topologies induced by the maps  $\pi_{\textrm{L},\bar{m}}$ and  $\pi_{\textrm{KP},\bar{m}}$, respectively. This definition does not depend on the choice of $\FC$ and $M$. It follows from the definition that for $* \in \{\textrm{L},\textrm{KP}\}$, a subset $C$ of $\gal_{*}(T)$ is closed if and only if $\pi_{*}^{-1}[C] = \{ \sigma \in \aut(\mathfrak{C}): \mathfrak{C} \models \pi(\sigma(\bar a))\}$ for some small (i.e., the number of formulas in $\pi(\bar x)$ is bounded) partial type $\pi(\bar x)$ and a tuple $\bar a$ from $\mathfrak{C}$ (without loss of generality $|\bar a|=|\bar x|$ is small here). 

\medskip

Now, we turn to Bohr compactifications. As in Section \ref{subsection: Bohr comp}, $\mathfrak{b}_1 \colon \aut(\mathfrak{C}) \to B_1$ denotes the Bohr compactification of $\aut(\mathfrak{C})$ below $(S_{\bar c}(\mathfrak{C}),\tp(\bar c/\mathfrak{C}), \aut(\mathfrak{C}))$, and $\mathfrak{b}_2 \colon \aut(\mathfrak{C}) \to B_2$ the Bohr compactification of $\aut(\mathfrak{C})$ below the Ellis flow $(E(S_{\bar c}(\mathfrak{C}),\aut(\mathfrak{C})),\id, \aut(\mathfrak{C}))$. Let $\mathfrak{f} \colon B_2 \to B_1$ be the epimorphism from Remark \ref{remark: B2 to B1}, so $\mathfrak{f}  \circ \mathfrak{b}_2 =\mathfrak{b}_1$. Analogously, since $\pi_{\textrm{KP}} \colon \aut(\FC) \to \gal_{\textrm{KP}}(T)$ is a group compactification of $\aut(\FC)$ below $S_{\bar c}(\FC)$, there exists a unique continuous group epimorphism $\mathfrak{g} \colon B_1 \to \gal_{\textrm{KP}}(T)$ such that $\mathfrak{g} \circ \mathfrak{b}_1 = \pi_{\textrm{KP}}$. Therefore,

\begin{remark}\label{remark: Autf_KP bigger than other kernels}
$\aut(\FC) \geq \autf_{\textrm{KP}}(\FC) \geq \ker(\mathfrak{b}_1) \geq \ker(\mathfrak{b}_2)$.
\end{remark}

Let $u^2=u\in\CM\unlhd_m E:=E(S_{\bar c}(\mathfrak{C}),\aut(\mathfrak{C}))$.

\begin{lemma}\label{lemma: bounded quotients}
$|B_1| \leq |B_2| \leq \beth_3(|T|)$. Thus, $|\aut(\FC)/\ker(\mathfrak{b}_1)| \leq |\aut(\FC)/\ker(\mathfrak{b}_2)| \leq \beth_3(|T|)$. 
\end{lemma} 

\begin{proof}
The first inequality follows from the existence of the epimorphism $\mathfrak{f} \colon B_2 \to B_1$. By Corollary \ref{corollary: Newelski's map in general} there is an epimorphism $\widehat{\theta} \colon  u\mathcal{M} \to B_2$, so $|B_2| \leq |u\CM|$. On the other hand, by \cite[Corollary A.8]{Hrushovski2019},  $|u\CM| \leq \beth_3(|T|)$ (it was first proved in \cite{Krupinski2019} with the bigger bound equal to $\beth_5(|T|)$ which is sufficient to realize the goal of this subsection).
\end{proof}

\begin{lemma}\label{lemma: comparison of kernels}
$\ker(\mathfrak{b}_1) = \autf_{\textrm{KP}}(\FC)$ and $\autf_{\textrm{L}}(\FC) \leq \ker(\mathfrak{b}_2)$.
\end{lemma}

\begin{proof}
Let us first prove that $\ker(\mathfrak{b}_1) = \autf_{\textrm{KP}}(\FC)$.
By Remark \ref{remark: Autf_KP bigger than other kernels}, $\autf_{\textrm{KP}}(\FC) \geq \ker(\mathfrak{b}_1)$, hence it remains to prove $\autf_{\textrm{KP}}(\FC) \leq \ker(\mathfrak{b}_1)$. By Lemma \ref{lemma: bounded quotients}, $|\aut(\FC)/\ker(\mathfrak{b}_1)|$ is bounded, so, by virtue of \cite[Proposition 4.5]{HruKruPi}, it is enough to show that $\ker(\mathfrak{b}_1)$ is relatively type-definable in the sense of \cite{HruKruPi}. 

Since $\mathfrak{b}_1$ is a compactification of $\aut(\FC)$ below $S_{\bar c}(\FC)$, there is a unique $\aut(\FC)$-flow epimorphism $w_1 \colon S_{\bar c}(\FC) \to B_1$ such that $\mathfrak{b}_1= w_1 \circ t_1$, where $t_1 \colon \Aut(\FC) \to S_{\bar c}(\FC)$ is given by $t_1(\sigma):=\tp(\sigma(\bar c)/\FC)$. By boundedness of $B_1$ and continuity of $w_1$, we get that $w_1^{-1}(e_{B_1}) = [\pi]$ for some small partial type $\pi$ (i.e. the number of formulas in $\pi$ is bounded). Thus, $\ker(\mathfrak{b}_1) = t_1^{-1}[w_1^{-1}(e_{B_1})] = \{ \sigma \in \aut(\FC): \FC\models \pi(\sigma(\bar c))\}$, which is a relatively type-definable subset (by smallness of $\pi$).

Now, we show that $\autf_{\textrm{L}}(\FC) \leq \ker(\mathfrak{b}_2)$. By Lemma \ref{lemma: bounded quotients}, $|\aut(\FC)/\ker(\mathfrak{b}_2)|$ is bounded, so, by virtue of \cite[Proposition 5.2]{Krupinski2019}, it is enough to show that $\ker(\mathfrak{b}_2)$ is closed in the pointwise convergence topology.

As above, we have a (continuous) epimorphism $w_2 \colon E \to B_2$ such that $\mathfrak{b}_2= w_2 \circ t_2$, where $t_2 \colon \aut(\FC) \to E$ is given by $t_2(\sigma):=\sigma \id_{S_{\bar{c}}(\FC)}$. So closedness of $\ker(\mathfrak{b}_2)$ follows from continuity of $t_2$ and $w_2$.
\end{proof}

\begin{question}\label{question: ker(b_2)}
Is it true that $\ker(\mathfrak{b}_2) = \autf_{\textrm{KP}}(\FC)$?
\end{question}

We expect a negative answer, but by Lemma \ref{lemma: comparison of kernels} and Remark \ref{remark: Autf_KP bigger than other kernels}, we get

\begin{cor}\label{corollary: G-compactness implies equality of kernels}
If $T$ is G-compact, then $\ker(\mathfrak{b}_1) =  \ker(\mathfrak{b}_2) = \autf_{\textrm{KP}}(\FC)$. Thus, $\mathfrak{b}_1$ and $\mathfrak{b}_2$ coincides (by which we mean that they are isomorphic).
\end{cor}

\begin{lemma}\label{lemma: induced top weaker than logic top}
The topology on $\aut(\FC)/\ker(\mathfrak{b}_1)$ inherited from $B_1$ is weaker than the logic topology.
\end{lemma}

\begin{proof}
We have the map $h_1:=\pi_{\textrm{KP},\bar{c}} \colon S_{\bar c}(\FC) \to \gal_{KP}(T) =\aut(\FC)/\ker(\mathfrak{b}_1)$ which restricted to $S:=\aut(\FC) \cdot \tp(\bar{c}/\FC)$ is given by $h_1(\tp(\sigma(\bar c)/\FC))=\sigma/\ker(\mathfrak{b}_1)$. We also have the factorization $i_1 \colon \aut(\FC)/\ker(\mathfrak{b}_1) \to B_1$ of $\mathfrak{b}_1$, which is a group embedding. Let $t_1 \colon \aut(\FC) \to S$ be given by $\sigma \mapsto \tp(\sigma(\bar c)/\FC)$, and $q_1 =\pi_{\textrm{KP}}\colon \aut(\FC) \to \aut(\FC)/\ker(\mathfrak{b}_1)$ be the quotient map. Let  $w_1 \colon S_{\bar c}(\FC) \to B_1$ be the unique $\aut(\FC)$-flow epimorphism with $\mathfrak{b}_1= w_1 \circ t_1$. It is clear that the following diagram commutes:

\begin{center}
			\begin{tikzcd}
			S \ar[r, two heads, "h_1|_S"] \ar[rr,bend left=30,"w_1|_S"]& \aut(\FC)/\ker(\mathfrak{b}_1) \ar[r, "i_1"] & B_1 \\
			\aut(\FC) \ar[u, two heads, "t_1"] \ar[ur, two heads,"q_1"] \ar[urr,"\mathfrak{b}_1"] & &
			\end{tikzcd}
		\end{center}

Consider any closed subset $C$ of $ \aut(\FC)/\ker(\mathfrak{b}_1)$ in the topology inherited from $B_1$ via the embedding $i_1$, i.e. $i_1[C]= C'\cap \im(i_1)$ for some closed $C'\subseteq B_1$. By boundedness of $B_1$ and continuity of $w_1$, we get that $w_1^{-1}[C'] = [\pi]$ for some small partial type $\pi$. So $(h_1|_S)^{-1}[C] =(w_1|_S)^{-1}[C'\cap \im(i_1)]= [\pi] \cap S$. Therefore, $q_1^{-1}[C]= \{ \sigma \in \aut(\FC): \FC \models \pi(\sigma(\bar c))\}$ which is relatively type-definable (by smallness of $\pi$), and so $C$ is closed in the logic topology. 
\end{proof}

\begin{proposition}\label{proposition: Bohr below S(C) agrees with GAL_KP}
The map $\mathfrak{b}_1$ is surjective and the induced $i_1 \colon \aut(\FC)/\ker(\mathfrak{b}_1) \to B_1$ is a group isomorphism. Moreover, the topology on $\aut(\FC)/\ker(\mathfrak{b}_1)$ inherited from $B_1$ coincides with the logic topology. In consequence, the quotient map $q_1 \colon \aut(\FC) \to \gal_{\textrm{KP}}(T)$ is the Bohr compactification of $\aut(\FC)$ below the ambit $(S_{\bar c}(\FC),\tp(\bar{c}/\FC),\aut(\FC))$.
\end{proposition}

\begin{proof}
By Lemma \ref{lemma: induced top weaker than logic top}, compactness of $\aut(\FC)/\ker(\mathfrak{b}_1) = \gal_{KP}(T)$, and Hausdorffness of $B_1$, we get that the topology on $\aut(\FC)/\ker(\mathfrak{b}_1)$ inherited from $B_1$ is compact (Hausdorff) and coincides with the logic topology. This implies that $\mathfrak{b}_1[\aut(\FC)]=i_1[\aut(\FC)/\ker(\mathfrak{b}_1)]$ is closed in $B_1$, but it is also dense, and so it is equal to $B_1$.
\end{proof}

\begin{question}\label{question: b_1=b_2in general}
Do we have a counterpart of Proposition \ref{proposition: Bohr below S(C) agrees with GAL_KP} for $\mathfrak{b}_2$, i.e.: Is $\mathfrak{b}_2$ surjective (equivalently, is the induced $i_2 \colon \aut(\FC)/\ker(\mathfrak{b}_2) \to B_2$ a group isomorphism)? Does the topology on $\aut(\FC)/\ker(\mathfrak{b}_2)$ inherited from $B_2$ coincide with the logic topology? A positive answer to the last question would imply that $\aut(\FC)/\ker(\mathfrak{b}_2)$ is compact (Hausdorff), hence $\ker(\mathfrak{b}_2) =\autf_{KP}(\FC)$ and $\mathfrak{b}_2$ would be surjective.
Thus, the quotient map $q_2 \colon \aut(\FC) \to \gal_{KP}(T)$ would be the Bohr compactification of $\aut(\FC)$ below $E$ (equivalently, $\mathfrak{b}_1$ and $\mathfrak{b}_2$ would coincide).
\end{question}

Note that the only difficulty in adapting the above proofs to $\mathfrak{b}_2$ is in Lemma \ref{lemma: induced top weaker than logic top}. Most likely it is false. In the context of definable groups, Example 4.17 in \cite{KLM2} yields a definable group $G=M$ (even $\omega$-categorical) such that the Bohr compactification of $G$ below $S_G(M)$ is trivial, whereas the Bohr compactification below $E(S_G(M))$ is non-trivial. Maybe by adding an affine sort we could produce a counter-example to the above question, but it does not seem to be straightforward, as we consider compactifications of a much bigger group after adding the affine sort.

\begin{question}\label{question: b_1=b_2 under NIP}
Do $\mathfrak{b}_1$ and $\mathfrak{b}_2$ coincide under NIP?
\end{question}

By \cite[Theorem 4.15]{CPS14}, we know that a counterpart of this statement for groups definable in NIP theories holds.

\subsubsection{Transfer of amenability}
The goal is to show that if $T$ is NIP and amenable in the sense of \cite{HruKruPi}, i.e. the flow $(S_{\bar c}(\FC), \aut(\FC))$ is amenable, then so is the Ellis flow $(E(S_{\bar c}(\FC),\aut(\FC)),\aut(\FC))$. This is a variant for first order theories of Theorem 3.17 in \cite{CPS14} which says that if $G$ is a group definable in a structure $M$, then definable amenability of $G$ implies definable amenability of $G$ in the Shelah expansion of $M$ which in particular implies that the flow $(E(S_G(M),G),G)$ is amenable. 

As before, let $E:=E(S_{\bar c}(\mathfrak{C}),\aut(\mathfrak{C}))$.

\begin{lemma}\label{lemma: Autf_KP-inv eta}
Assume that $T$ is NIP and amenable. Then there exists $\eta \in E$ which is $\autf_{\textrm{KP}}(\FC)$-invariant. 
\end{lemma}

\begin{proof}
By assumption and \cite[Corollary 2.24]{HruKruPi}, the empty set is an extension base for forking. 
So, by \cite[Theorem 7.7]{Krupinski2019}, the minimal left ideals
in $E$ are of bounded size. By \cite[Corollary 7.2]{Krupinski2019}, this implies that there exists $\eta \in E$ whose image is contained in the set of types which do not fork over $\emptyset$.
On the other hand, by \cite[Proposition 2.11]{Anand_Udi2011}, a global type $p$ does not fork over $\emptyset$ if
and only if $p$ is invariant under $\autf_{\textrm{KP}}(\FC)$. Thus, $\eta$ is $\autf_{\textrm{KP}}(\FC)$-invariant.
\end{proof}

\begin{proposition}\label{proposition: transfer of amenability for theories}
Assume that $T$ is NIP and amenable. Then the flow $(E,\aut(\FC))$ is also amenable.
\end{proposition}

\begin{proof}
By Lemma \ref{lemma: Autf_KP-inv eta}, we can choose $\eta \in E$ which is $\autf_{\textrm{KP}}(\FC)$-invariant.
    The map $\aut(\FC)\ni\sigma\mapsto\sigma\eta\in\overline{\aut(\FC)\eta}$ factors through $\aut(\FC)/\autf_{\KP}(\FC)$:
    $$\xymatrix{\aut(\FC) \ar[rr]^-{\sigma\mapsto\sigma\eta} 
    \ar[dr]_-{\sigma\mapsto\sigma/\autf_{\KP}(\FC)\qquad}& & \overline{\aut(\FC)\eta} \\
    & \aut(\FC)/\autf_{\KP}(\FC)\ar@{-->}[ur]_-{f}}$$
    Note that $f$ is an $\aut(\FC)$-equivariant bijection on $\aut(\FC)\eta$,
    i.e. $f(\tau \sigma/\autf_{\KP}(\FC))=\tau\cdot f(\sigma/\autf_{\KP}(\FC))$ for every $\tau,\sigma\in\aut(\FC)$.
    Moreover, $f$ is Borel due to Lemma 7.15 from \cite{Krupinski2019}.
    Indeed, consider a basic open set
    $$U:=\{\xi\in E\;:\;\bigwedge\limits_{k\leqslant n}\xi(p_k)\in[\theta_k(\bar{x};\bar{c})]\},$$
    where $p_1,\ldots,p_n\in S_{\bar{c}}(\FC)$ and $\theta_1(\bar x,\bar y), \dots, \theta_n(\bar x,\bar y)$ are formulas without parameters.
    Then $\sigma/\autf_{\KP}(\FC)\in f^{-1}[U]$ if and only if
    for every $k\leqslant n$ we have
    $\sigma\eta(p_k)\in[\theta_k(\bar{x};\bar{c})]$.
    By the choice of $\eta$, each $\eta(p_k)$ is $\autf_{\textrm{KP}}(\FC)$-invariant, so its $\aut(\FC)$-orbit is of bounded size.
    Thus, Lemma 7.15 from \cite{Krupinski2019} implies that 
    $$U_k:=\{\sigma/\autf_{\KP}(\FC):\sigma\eta(p_k)\in[\theta_k(\bar{x};\bar{c})]\}$$ 
is constructible. Hence, $f^{-1}[U]=\bigcap_{k\leqslant n}\, U_k$ is constructible, too.

    Since $\aut(\FC)/\autf_{\KP}(\FC)$ is a compact Hausdorff group, it has a normalized Haar measure, say $\mathfrak{h}$. Then the pushforward $f_*\mathfrak{h}$ is an $\aut(\FC)$-invariant, regular, Borel probability measure on $E$.
\end{proof}

\subsubsection{Newelski's map for first order theories}
Let $u^2=u\in\CM\unlhd_m E:=E(S_{\bar c}(\FC),\aut(\FC))$. Let $\FC' \succ \FC$ be a monster model in which $\FC$ is small.
A counterpart of Newelski's map in the context of first order theories was introduced in \cite[Section 2]{KruPiRze}. This is the continuous group epimorphism $\rho \colon u\mathcal{M} \to \gal_{\textrm{KP}}(T) =\aut(\FC')/\autf_{\textrm{KP}}(\FC')$ given by $\rho(\eta):=\sigma'/ \autf_{\textrm{KP}}(\FC')$ where $\sigma'(\bar c) \models \eta(\tp(\bar c/\FC))$. Theorem 0.7 from \cite{Krupinski2019} is a counterpart of  Fact \ref{fact: Theorem 1.5 from CS}; it is a counterpart of Newelski's conjecture for NIP amenable theories.

\begin{fact}[Theorem 0.7 in \cite{Krupinski2019}]\label{fact: NC for theories}
If $T$ is NIP and amenable, then $\rho$ is an isomorphism.
\end{fact}

Let us use the notation from Subsection \ref{subsubsection: Bohr below S_c(C)}. Namely, using Proposition \ref{proposition: Bohr below S(C) agrees with GAL_KP}, we have Bohr compactifications $\mathfrak{b}_1 \colon \aut(\mathfrak{C}) \to B_1=\gal_{\textrm{KP}}(T)$ and $\mathfrak{b}_2 \colon \aut(\mathfrak{C}) \to B_2$, the group epimorphism  $\mathfrak{f} \colon B_2 \to B_1$  satisfying $\mathfrak{f}  \circ \mathfrak{b}_2 =\mathfrak{b}_1$, and the epimorphism $\widehat{\theta} \colon u\CM \to B_2$.

\begin{lemma}\label{lemma: Newelski's map and theta for theories}
The Newelski's map $\rho$ recalled above coincides with $\mathfrak{f} \circ \widehat{\theta}$.
\end{lemma}

\begin{proof}
Since $\widehat{\theta}(\eta): = \eta e_{B_2}$ and $\mathfrak{f}$ is an epimorphism of $\aut(\FC)$-ambits, we have $(\mathfrak{f} \circ \widehat{\theta})(\eta) = \eta (\id_{\FC'}/\autf_{\textrm{KP}}(\FC'))$. So it is enough to show that the unique left continuous action of $E$ on $\gal_{\textrm{KP}}(T)$ extending the action of $\aut(\FC)$ is given by $\eta (\tau'/\autf_{\textrm{KP}}(\FC'))=(\sigma' \tau')/\autf_{\textrm{KP}}(\FC')$ where $\sigma'(\bar c) \models \eta(\tp(\bar c/\FC))$ (and $\sigma',\tau' \in \aut(\FC')$). For that it remains to show that the last formula defines a left continuous function extending the action of $\aut(\FC)$. It is well-defined as $\tp(\sigma_1'(\bar c)/\FC) = \tp(\sigma_2'(\bar c)/\FC)$ implies $\sigma_1'/\autf_{\textrm{KP}}(\FC')  = \sigma_2'/\autf_{\textrm{KP}}(\FC')$. The fact that it extends the action of $\aut(\FC)$ is obvious. For continuity consider any $\tau'/\autf_{\textrm{KP}}(\FC') \in \gal_{\textrm{KP}}(T)$ and any closed $C\subseteq \gal_{\textrm{KP}}(T)$. Then, for the quotient map $\pi'_{\textrm{KP}} \colon \aut(\FC') \to \gal_{\textrm{KP}}(T)$, by the definition of the logic topology, there exists a partial type $\pi(\bar x)$ over $\FC$ such that $\pi'^{-1}_{\textrm{KP}}[C] = \{ \sigma' \in \aut(\FC'): \FC' \models \pi(\sigma'(\bar c))\}$. So $\{\eta \in E: \eta(\tau'/\autf_{\textrm{KP}}(\FC')) \in C\} = \{\eta \in E : \eta(\tp(\tau'(\bar c)/\FC)) \in [\pi]\}$ which is closed (to see the last equality, one should note that for any $\eta \in E$ there is $\sigma' \in \aut(\FC')$ with  $\sigma'(\bar c) \models \eta(\tp(\bar c/\FC))$ and  $\sigma'(\tau'(\bar c)) \models \eta(\tp(\tau'(\bar c)/\FC))$; see \cite[Proposition 2.2]{KruPiRze}).
\end{proof}

So the conclusion of GNC applied to the ambit $(S_{\bar c}(\FC),\tp(\bar{c}/\FC),\aut(\FC))$ is weaker than the conclusion of Fact \ref{fact: NC for theories}, similarly to the case of definable groups (see the comments after Lemma \ref{lemma: Newelski's map and theta for definable groups}). However, for definable groups, assuming NIP, both conclusions are equivalent by Theorem 4.15 from \cite{CPS14}, whereas for theories a counterpart of \cite[Theorem 4.15]{CPS14} is left as Question \ref{question: b_1=b_2 under NIP}. Note that assuming NIP and amenability of $T$, Fact \ref{fact: NC for theories} together with Lemma \ref{lemma: Newelski's map and theta for theories} imply that $\mathfrak{b}_1$ and $\mathfrak{b}_2$ coincide, but it would be very interesting to find a simple direct proof. Then we would get that Theorem \ref{theorem: GNC} implies Fact \ref{fact: NC for theories}. Recall that amenability of $T$ implies G-compactness by \cite[Theorem 4.1]{HruKruPi} (in the NIP case, this is a much easier \cite[Corollary 2.26]{HruKruPi}), so by Corollary \ref{corollary: G-compactness implies equality of kernels}, we get $\ker(\mathfrak{b}_2)=\ker(\mathfrak{b}_1)$, but it is still weaker than the equality (or isomorphism) of $\mathfrak{b}_1$ and $\mathfrak{b}_2$.

\subsubsection{Two problems}\label{subsection: two problems}

\begin{enumerate}
\item By \cite[Theorem 3.17]{CPS14}, for a group $G$ definable in a NIP structure $M$ we have that $(S_G(M), G)$ is amenable if and only if $(E(S_G(M),G),G)$ is amenable. 
\item By Proposition \ref{proposition: transfer of amenability for theories}, for a first order theory $T$ with NIP we have that $(S_{\bar c}(\FC),\aut(\FC))$ is amenable if and only if $(E(S_{\bar c}(\FC),\aut(\FC)),\aut(\FC))$ is amenable.
\end{enumerate}

In a general context of a tame flow $(X,G)$, Fact \ref{Fact: 5.6} tells us that amenability of $(E(X,G),G)$ is equivalent to hereditary amenability of $(X,G)$. A simple example discussed after Conjecture \ref{conjecture: main conjecture} shows that hereditary amenability cannot be weakened to amenability. It is also possible to give an example of a tame ambit (not just a flow) where amenability is not sufficient for the above equivalence, but it is much more complicated so we skip it. However, in the model-theoretic situations (1) and (2), we do not need hereditary amenability. A posteriori (i.e. using (1) and (2)), hereditary amenability follows from amenability in those two contexts. This leads to

\begin{problem}\label{problem: 1}
Find a dynamical property of a tame ambit $(X,x_0,G)$  which holds for ambits in (1) and (2) and which implies that amenability of $(X,G)$ is equivalent to hereditary amenability of $(X,G)$ (and so to amenability of $(E(X,G),G)$)?
\end{problem}

By \cite{CPS14}, if $G$ is a group definable in a NIP structure $M$, then the Bohr compactifications of $G$ below $(S_G(M),\tp(e/M),G)$ and $(E(S_G(M),G),\id_{S_G(M)}, G)$ coincide. Question \ref{question: b_1=b_2 under NIP} asks whether for a NIP theory $T$ the Bohr compactifications of $\aut(\FC)$ below  $(S_{\bar c}(\FC),\tp(\bar{c}/\FC),\aut(\FC))$ and $(E(S_{\bar c}(\FC),\aut(\FC)),\id_{S_{\bar c}(\FC)},\aut(\FC))$ coincide, with expected positive answer. This leads to

\begin{problem}\label{problem: 2}
Find a dynamical property of a tame ambit $(X,x_0,G)$ which holds for the above two model-theoretic ambits and which implies that the Bohr compactifications of $G$ below $(X,x_0,G)$ and $(E(X,G),\id_X,G)$ coincide\footnote{The first author discussed a similar problem with Eli Glasner, who provided a solution. Let us comment that here we would like to recognize the general mechanism from topological dynamics which is behind the model-theoretic techniques.}.
\end{problem}

\subsection{Local Newelski's Conjecture}\label{subsection: local NC}
Let $G$ be a group definable in a structure $M$, $\FC \succ M$ a monster model, and $\bar{G}:=G(\FC)$. Let $\Delta$ be a (finite or infinite) collection of formulas $\varphi(x,y)$, each of which implies $G(x)$, and such that $\{ \varphi(G,b): b \in M\}$ is closed under the left translations by the elements of $G$. By a $\Delta$-formula we mean a Boolean combination of instances of formulas from $\Delta$. Finally, $S_{G,\Delta}(M)$ denotes the space of complete $\Delta$-types over $M$  (i.e., complete consistent collections of $\Delta$-formulas over $M$) concentrated on $G$. Then $G$ acts on the left on $S_{G,\Delta}(M)$ via $g\cdot \tp_{\Delta}(a/M):=\tp_{\Delta}(ga/M)$, and $(S_{G,\Delta}(M), \tp_{\Delta}(e/M), G)$ is an ambit. Let $E:=E(S_{G,\Delta}(M),G)$ and $u^2=u \in \CM \unlhd_m E$. 

We consider our usual map $\widehat{\theta} = w_2|_{u\CM} \colon u\CM \to B_2$, where $\mathfrak{b}_2 \colon G \to B_2$ is the Bohr compactification of $G$ below $(E,\id_{S_{G,\Delta}(M)},G)$.

\begin{theorem}\label{theorem: local NC}
If each formula in $\Delta$ has NIP and $(S_{G,\Delta}(M),G)$ is hereditarily amenable, then $\widehat{\theta}$ is an isomorphism of topological groups.
\end{theorem}

\begin{proof}
Since each formula in $\Delta$ has NIP, the characteristic functions of clopens in $S_{G,\Delta}(M)$ are all tame which implies that the flow $(S_{G,\Delta}(M), G)$ is tame. Thus, the result follows from Theorem \ref{thm:GNC 0}.
\end{proof}

The key thing is to describe $B_2$ in a more explicit way. For that we first recall a piece of basic general theory developed by Newelski.

Let $\mathcal{A}$ be a {\em Boolean $G$-algebra} of subsets of $G$, i.e.\ a Boolean algebra closed under the left translations by the elements of $G$. 
By $S(\mathcal{A})$ we denote the Stone space of $\mathcal{A}$. The group $G$ acts naturally on $S(\mathcal{A})$ by left translations, turning $S(\mathcal{A})$ into a $G$-ambit with the principal ultrafilter $p_e:=\{A \in \mathcal{A}: e \in A\}$ as the distinguished point. 
For $p \in S(\mathcal{A})$ the operator $d_p \colon \mathcal{A} \to \mathcal{P}(G)$ is defined by: 
$$d_p(A):=\{ g\in G: g^{-1}A \in p\}.$$ 
Each $d_p$ is a homomorphism of Boolean $G$-algebras. The algebra $\mathcal{A}$ is called {\em d-closed} if $d_p[\mathcal{A}]=\mathcal{A}$
for every $p\in S(\mathcal{A})$.
The next fact was observed by Newelski, but appeared in Adam Malinowski's thesis \cite{Mal}.

\begin{fact}\cite[Fact 2.22 and Theorem 2.24]{Mal}\label{fact: A^d}
Let $\mathcal{A}$ be a Boolean $G$-algebra of subsets of $G$. There exists a smallest d-closed $G$-algebra of subsets of $G$ containing $\mathcal{A}$ which is denoted by $\mathcal{A}^d$. The d-closure $\mathcal{A}^d$ is realized in one step, i.e.\ $\mathcal{A}^d$ is precisely the Boolean algebra generated by $\{d_p(A): A \in \mathcal{A}, p \in S(\mathcal{A})\}$. Moreover, $E(S(\mathcal{A})) \cong S(\mathcal{A}^d)$ as semigroups and as $G$-ambits (with respect to a unique left-continuous semigroup operation $*$ on $S(\mathcal{A}^d)$ extending the action of $G$), with $\id_{S(\mathcal{A})}$ and $p_e$ as the distinguished points, respectively (where $p_e$ is the principal ultrafilter at $e$).
\end{fact}

Applying this fact in our situation to $\mathcal{A}$ defined as the Boolean algebra generated by $\{ \varphi(G,b): b \in M\}$, we produce $\mathcal{A}^d$ as above, and we can replace $(E,\id_{S_{G,\Delta}(M)},G)$ by $(S(\mathcal{A}^d),p_e,G)$. Now, let $M'$ be the structure with the universe $G$ equipped with the group operation and all members of $\mathcal{A}^d$ as predicates. Let $\Delta'$ consist of these predicates. Then $(S(\mathcal{A}^d),p_e,G)$ is naturally identified with $(S_{\Delta'}(M'), \tp_{\Delta'}(e/M'),G)$ and $\Delta'$ is easily seen to be closed under the left and right translations by the elements of $G$ (i.e., $\mathcal{A}^d$ has this property). We describe the Bohr compactification of $G$ below  $(S_{\Delta'}(M'), \tp_{\Delta'}(e/M'),G)$ under the additional assumptions that $\Delta'$ is closed under group inversion.
More generally, we describe the Bohr compactification of $G$ below the ambit of all complete $\Sigma$-types over $M$ for suitable collections of formulas $\Sigma$. The proof follows the standard method as in \cite{GPPb}, but there are some new issues to take care of in the local version.

\begin{proposition}\label{proposition: Bohr below local types}
Let $\Sigma$ be a (finite or infinite) collection of formulas $\varphi(x,y)$, each of which implies $G(x)$, and such that $\{ \varphi(G,b): b \in M\}$ is closed under the left translations by the elements of $G$ and under group inversion. Let $\bar{G}^{00}_{\Sigma,M}$ be the smallest bounded index $\Sigma$-type-definable over $M$ subgroup of $\bar{G}$ (where $\Sigma$-type-definable over $M$ means that $\bar{G}^{00}_{\Sigma,M} = \pi(\bar G)$ for some partial $\Sigma$-type $\pi$ over $M$). Then the quotient map $G \to \bar{G}/\bar{G}^{00}_{\Sigma,M}$ is the Bohr compactification of $G$ below the ambit  $(S_{G,\Sigma}(M), \tp_{\Sigma}(e/M), G)$.
\end{proposition}

\begin{proof}
Let $\mathfrak{b} \colon G \to B$ be the Bohr compactification of $G$ below $(S_{G,\Sigma}(M), \tp_{\Sigma}(e/M), G)$. Note that $\mathfrak{b}$ is {\em $\Sigma$-definable} in the sense that the preimages of any two disjoint closed subsets of $B$ can be separated by a $\Sigma$-definable subset of $G$. Adapting the proof of Lemma 3.2 from \cite{GPPb}, we get that the unique $G$-flow epimorphism $w \colon S_{G,\Sigma}(M) \to B$ such that $w \circ t=\mathfrak{b}$ for $t \colon G \to S_{G,\Sigma}(M)$ given by $t(g):=\tp_{\Sigma}(g/M)$ satisfies the formula 
$$w(p)= \bigcap_{\varphi(x) \in p} \overline{\mathfrak{b}[\varphi(G)]}.$$
We also have the function $\mathfrak{b}^* \colon \bar{G} \to B$ given by $\mathfrak{b}^*(a) := w(\tp_{\Sigma}(a/M))$. A slight elaboration of the proof of \cite[Proposition 3.4]{GPPb} shows that $\mathfrak{b}^*$ is a group homomorphism. An elaboration, which we give below, is needed, because  $\Sigma$-definable sets are not closed under the group products.

\begin{clm}
$\mathfrak{b}^*$  is a group homomorphism.
\end{clm}

\begin{clmproof}
Suppose for a contradiction that $\mathfrak{b}^*(ab) \ne \mathfrak{b}^*(a)\mathfrak{b}^*(b)$ for some $a,b \in \bar G$. Then $\mathfrak{b}^*(ab) \in U$ and $ \mathfrak{b}^*(a)\mathfrak{b}^*(b) \in V$ for some open, disjoint subsets $U$ and $V$ of $B$. So there is a $\Sigma$-formula  $\varphi(x) \in \tp_{\Sigma}(ab/M)$ such that $\overline{\mathfrak{b}[\varphi(G)]} \subseteq U$. Set $p:=\tp_{\Sigma}(a/M)$ and  $q:=\tp_{\Sigma}(b/M)$. Then 
\begin{align*}
\bigcap_{\psi_1\in p, \psi_2 \in q} \overline{\mathfrak{b}[\psi_1(G) \cdot \psi_2(G)]} &= \bigcap_{\psi_1 \in p, \psi_2 \in q} \overline{\mathfrak{b}[\psi_1(G)]} \cdot  \overline{\mathfrak{b}[\psi_2(G)]}\\ &=\bigcap_{\psi_1 \in p}  \overline{\mathfrak{b}[\psi_1(G)]} \cdot \bigcap_{\psi_2 \in q}  \overline{\mathfrak{b}[\psi_2(G)]}\\ 
&=\{\mathfrak{b}^*(a) \cdot \mathfrak{b}^*(b)\}.
\end{align*}
Hence, there are $\psi_1(x) \in p$ and $\psi_2(x) \in q$ such that $\overline{\mathfrak{b}[\psi_1(G) \cdot \psi_2(G)]} \subseteq V$. Since $U \cap V = \emptyset$, we conclude that $\varphi(G) \cap (\psi_1(G) \cdot \psi_2(G)) = \emptyset$ which contradicts the fact that $\varphi(\bar G) \cap (\psi_1(\bar G) \cdot \psi_2(\bar G)) \ne \emptyset$ witnessed by the element $ab$. 
\end{clmproof}

To finish the proof, it is enough to show that $\ker \mathfrak{b}^*= \bar{G}^{00}_{\Sigma,M}$. By the claim and the fact that $\mathfrak{b}^*$ naturally factors through $S_{G,\Sigma}(M)$, we get that $\ker \mathfrak{b}^*$ is a $\Sigma$-type-definable over $M$ subgroup of $\bar G$ which is of bounded index (as $|B| \leq 2^{2^{|G|}}$). Thus, $\ker \mathfrak{b}^* \supseteq  \bar{G}^{00}_{\Sigma,M}$. The opposite inclusion is very easy in the global case, but in our local case it requires a proof in which we will use the additional assumption that 
%the Boolean algebra of $\Sigma$-definable sets is closed under inversion.
the family of the instances of the formulas from $\Sigma$ is closed under inversion.

First, observe that it is enough to show that the quotient map $\mathfrak{b}'\colon G \to \bar{G}/\bar{G}^{00}_{\Sigma,M}$ is a compactification of $G$ below $(S_{G,\Sigma}(M), \tp_{\Sigma}(e/M), G)$. Indeed, by the inclusion $\ker \mathfrak{b}^* \supseteq  \bar{G}^{00}_{\Sigma,M}$, we have the obvious epimorphism $h \colon \bar{G}/\bar{G}^{00}_{\Sigma,M} \to \bar{G}/\ker \mathfrak{b}^*$ commuting with the quotient maps from $G$. Composing $h$ with the isomorphism $f \colon \bar{G}/\ker \mathfrak{b}^* \to B$ induced by $\mathfrak{b}^*$, we obtain an epimorphism $ f \circ h \colon  \bar{G}/\bar{G}^{00}_{\Sigma,M} \to B$ commuting with $\mathfrak{b}'$ and $\mathfrak{b}$. Since $\mathfrak{b}'$  is a compactification of $G$ below the ambit  $(S_{G,\Sigma}(M), \tp_{\Sigma}(e/M), G)$ and $\mathfrak{b}$ is the universal such compactification, we conclude that $h$ is an isomorphism, so $\ker \mathfrak{b}^* =  \bar{G}^{00}_{\Sigma,M}$.

To see that $\mathfrak{b}'\colon G \to \bar{G}/\bar{G}^{00}_{\Sigma,M}$ is a compactification of $G$ below the ambit $(S_{G,\Sigma}(M), \tp_{\Sigma}(e/M), G)$, it remains to show that for $a,b \in \bar{G}$:
\begin{equation}
\textrm{if $\tp_{\Sigma}(a/M)=\tp_{\Sigma}(b/M)$, then $a/\bar{G}^{00}_{\Sigma,M}= b/\bar{G}^{00}_{\Sigma,M}$.}\tag{*}\label{condition *}
\end{equation}

Let $\Psi: =\{ \psi_{\varphi}(x,z;y): \varphi(x,y) \in \Sigma\}$, where $\psi_{\varphi}(x,z;y):=\varphi(x^{-1}z,y)$. 

\begin{clm}
If $\tp_{\Sigma}(a/M)=\tp_{\Sigma}(b/M)$, then there exists a sequence $(c_1,c_2,\dots)$ such that both sequences $a^{\frown}(c_i)$ and $b^{\frown}(c_i)$ are $\Psi$-indiscernible over $M$ (meaning that for every $i<j$ and $k<l$, $\tp_{\Psi}(a_i,a_j/M) =\tp_{\Psi}(a_k,a_l/M)$) and $\Sigma$-indiscernible over $M$ (meaning that for every $i,k$, $\tp_{\Sigma}(a_i/M) =\tp_{\Sigma}(a_k/M)$).
\end{clm}

\begin{clmproof}
As in the usual construction of coheirs, choose an ultrafilter $\mathcal{U}$ on the Boolean algebra of all subsets of $G$ which extends the type $p:=\tp_{\Sigma}(a/M)$ (treated as an ultrafilter on the Boolean algebra of $\Sigma$-definable subsets of $G$). Set 
$$p_{\mathcal{U}}:=\{ \varphi(x): \varphi(x) \; \text{ is a } \Sigma\textrm{-formula such that } \varphi(\bar{G}) \cap G \in \mathcal{U}\}.$$ 
Then $p \subseteq p_{\mathcal{U}}$ and $p_{\mathcal{U}}$ is finitely satisfiable in $G$.

We claim that for every $\Psi$-formula $\delta(x,z)$
\begin{equation}
\textrm{if $\tp_{\Sigma}(a/M)=\tp_{\Sigma}(b/M)$, then $\delta(a,z) \in p_{\mathcal{U}}$ iff $\delta(b,z) \in p_{\mathcal{U}}$.}\tag{$\dagger$}\label{condition dag}
\end{equation}
It is clear that for any $c \in M$, $\delta(c,z)$ is a $\Sigma$-formula.
As the set of instances of the formulas from $\Sigma$ is closed under the left translations and under {\em inversion}, we see that for any $c \in M$, $\delta(x,c)$ is a $\Sigma$-formula, and so the condition  $\tp_{\Sigma}(a/M)=\tp_{\Sigma}(b/M)$ implies that for any $c \in M$, $\models \delta(a,c)$ iff $\models \delta(b,c)$. Therefore, we have the following sequence of equivalences:
$$\delta(a,z) \in p_{\mathcal{U}} \iff \delta(a,\bar G) \cap G \in \mathcal{U} \iff  \delta(b,\bar G) \cap G \in \mathcal{U} \iff \delta(b,z) \in p_{\mathcal{U}},$$
so (\ref{condition dag}) is proved.

Now, let $(c_1,c_2,\dots)$ be a Morley sequence in $p_{\mathcal{U}}$ over $M,a,b$, that is $c_i \models p_{\mathcal{U}}|_{M,a,b,c_{<i}}$ for every $i$. The fact that $a^{\frown}(c_i)$ and $b^{\frown}(c_i)$ are both $\Sigma$-indiscernible over $M$ is obvious. Thus, $\Psi$-indiscernibility over $M$ follows from  (\ref{condition dag}).
\end{clmproof}

The next claim follows by compactness.

\begin{clm}
If $E$ is a $\Psi$-type-definable over $M$ bounded equivalence relation on $\bar G$ and a sequence $(c_i)_{i<\omega}$ is $\Psi$-indiscernible over $M$, then $E(c_i,c_j)$ for all $i<j<\omega$.
\end{clm}

The next claim is a direct consequence of Claims  2 and 3.

\begin{clm}
If $\tp_{\Sigma}(a/M)=\tp_{\Sigma}(b/M)$ and $E$ is a $\Psi$-type-definable over $M$ bounded equivalence relation on $\bar G$, then $E(a,b)$.
\end{clm}

The desired condition (\ref{condition *}) follows from Claim 4, as the relation of lying in the same coset modulo $\bar{G}^{00}_{\Sigma,M}$ is clearly  a $\Psi$-type-definable over $M$ bounded equivalence relation on $\bar G$.
\end{proof}

We finish with a list of corollaries, comments, and questions around Theorem \ref{theorem: local NC} and Proposition \ref{proposition: Bohr below local types}.

\begin{itemize}
\item Unlike in the case of the ambits $S_G(M)$ and $S_{\bar c}(\FC)$, Theorem \ref{theorem: local NC} cannot be strengthened by saying that $u\CM$ is isomorphic to the Bohr compactification of $G$ below $(S_{G,\Delta}(M), \tp_{\Delta}(e/M), G)$. Indeed, take $G:=S_3$, $M:=(G,\cdot, H)$ where $H:=\langle (1,2)\rangle$, and $\Delta:=\{\varphi(x,y)\}$ where $\varphi(x,y)$ is the formula ``$x \in yH$''. Then $(S_{G,\Delta}(M),\tp_\Delta(\id/M),G)$ is isomorphic to $(G/H, \id/H,G)$ and by Example \ref{example: B_1 not equal  to B_2}, we  get that $u\CM \cong G$ while the Bohr compactification below $(G/H, \id/H,G)$ is trivial.

\item Unlike  in the case of the ambits $S_G(M)$ and $S_{\bar c}(\FC)$, the assumption that $(S_{G,\Delta}(M),G)$ is hereditarily amenable cannot be weakened to amenability. We skip an example, as it is pretty long.  One could ask if there is a counterexample with finite $\Delta$.

\item Proposition \ref{proposition: Bohr below local types} applied to $\Delta$ being the collection of all formulas $\varphi(x,y)$ implying $G(x)$ recovers Proposition 3.4 from \cite{GPPb}.

\item Consider the Shelah expansion $M^\textrm{Sh}$ of $M$, i.e. the universe $M$ equipped with predicates for all externally definable subsets in any finite numbers of variables (for our purpose we could only consider externally definable sets in the variables corresponding to $G$). The externally definable Bohr compactification of $G$ considered in \cite{Krupinski2017} is precisely the one below the ambit $(S^{\textrm{qf}}_{G}(M^{\textrm{Sh}}), \tp^{\textrm{qf}}(e/M^{\textrm{Sh}}),G)$ (where ``qf'' stands for ``quantifier free''). Proposition \ref{proposition: Bohr below local types} applied to $M^\textrm{Sh}$ and $\Delta$ being the collection of all quantifier free formulas  $\varphi(x,y)$ implying $G(x)$ (or just the collection of all predicates for externally definable subsets of $G$) yields a new description of the externally definable Bohr compactification of $G$ in terms of a quotient of $\bar{G}$ by a suitable model-theoretic connected component.

\item In Proposition \ref{proposition: Bohr below local types}, can the assumption that the family of the instances of the formulas from $\Sigma$ is closed under inversion be dropped?

\item Is the property that a Boolean $G$-algebra $\mathcal{A}$ of subsets of $G$ is closed under inversion preserved under passing to $\mathcal{A}^d$? For an abelian or a finite group $G$ it is clear, but in general there is no obvious reason why it should be true.

\end{itemize}

\printbibliography
\end{document}